\documentclass[a4paper,12pt]{amsart}
\usepackage{fullpage}
\usepackage[utf8]{inputenc}

\usepackage{amssymb}
\usepackage{amscd}
\usepackage{tikz-cd}
\usepackage[colorlinks]{hyperref}
\usepackage{mathrsfs}
\usepackage{xypic}
\usepackage{xspace}

\renewcommand{\mathcal}{\mathscr}

\renewcommand{\leq}{\leqslant}
\renewcommand{\geq}{\geqslant}

\newtheorem{theo}{Theorem}[section]
\newtheorem{lem}[theo]{Lemma}
\newtheorem{cor}[theo]{Corollary}
\newtheorem{prop}[theo]{Proposition}
\theoremstyle{definition}
\newtheorem{defi}[theo]{Definition}
\newtheorem{remark}[theo]{Remark}
\newtheorem{example}[theo]{Example}

\theoremstyle{remark}

\def \Ql {{\overline{\mathbf Q}_{\ell}}}
\def \bQ {{\overline{\mathbf Q}}}

\def \PH {{}^{\mathfrak p} \mathcal H}

\DeclareMathOperator{\Spec}{Spec}
\DeclareMathOperator{\Tr}{Tr}
\DeclareMathOperator{\Diag}{Diag}
\DeclareMathOperator{\drp}{drop}
\DeclareMathOperator{\jump}{jump}
\DeclareMathOperator{\swan}{swan}
\DeclareMathOperator{\loc}{loc}
\DeclareMathOperator{\FKM}{fkm}

\newcommand{\Zz}{\mathbf{Z}}
\newcommand{\Qq}{\mathbf{Q}}
\newcommand{\Rr}{\mathbf{R}}
\newcommand{\Aa}{\mathbf{A}}
\newcommand{\Pp}{\mathbf{P}}
\newcommand{\Cc}{\mathbf{C}}
\newcommand{\Ff}{\mathbf{F}}
\newcommand{\Gg}{\mathbf{G}}
\DeclareMathOperator{\ft}{FT}

\newcommand{\uple}[1]{\text{\boldmath${#1}$}}
\newcommand{\mtrx}[1]{\uple{#1}}

\def \mcL {\mathcal L}
\def \rank {\operatorname{rank}}

\def\eg{e.g.\kern.3em}
\def\loccit{loc.\kern3pt cit.{}\xspace}
\def\resp{\text{resp.}\kern.3em}

\DeclareMathOperator{\Der}{D_c^{\mathrm{b}}}
\DeclareMathOperator{\rhom}{\mathcal{H}om}
\DeclareMathOperator{\CC}{CC}
\let\SS=\relax
\DeclareMathOperator{\SS}{SS}
\DeclareMathOperator{\cc}{cc}
\DeclareMathOperator{\dual}{D}
\newcommand{\rmH}{\mathrm{H}}
\newcommand{\Gal}{\mathrm{Gal}}

\DeclareMathOperator{\GL}{GL}

\DeclareMathOperator{\PGL}{PGL}
\DeclareMathOperator{\Sp}{Sp}
\DeclareMathOperator{\USp}{USp}
\DeclareMathOperator{\Un}{U}
\DeclareMathOperator{\Ort}{O}
\DeclareMathOperator{\CH}{CH}

\makeatletter
\let\@wraptoccontribs\wraptoccontribs
\makeatother

\begin{document}

\title{Quantitative Sheaf Theory}

\author{Will Sawin}

\address[W.\,Sawin]{Columbia University, 2990 Broadway, New York, NY, USA 10027}

\email{sawin@math.columbia.edu}

\contrib[Mis en forme par]{A. Forey, J. Fresán and E. Kowalski}

\address[A. Forey]{EPFL/SB/TAN, Station 8, CH-1015 Lausanne, Switzerland} 
  \email{arthur.forey@epfl.ch}

\address[J.\,Fres\'an]{CMLS, \'Ecole polytechnique, F-91128 Palaiseau cedex, France}
\email{javier.fresan@polytechnique.edu}

\address[E.\,Kowalski]{D-MATH, ETH Z\"urich, R\"amistrasse 101, CH-8092 Z\"urich, Switzerland} 
  \email{kowalski@math.ethz.ch}

  \thanks{W.\,S. was supported by Dr.\,Max R\"{o}ssler, the Walter
    Haefner Foundation and the ETH Zürich Foundation and by by NSF grant DMS-2101491. A.\,F. and
    E.\,K. are supported by the DFG-SNF lead agency program grant
    200020L\_175755. A.\,F. is supported by SNF Ambizione grant PZ00P2\_193354. J.\,F. is partially supported by the grant
    ANR-18-CE40-0017 of the Agence Nationale de la Recherche.}

\begin{abstract}
  We introduce a notion of complexity of a complex of $\ell$-adic
  sheaves on a quasi-projective variety and prove that the six
  operations are ``continuous'', in the sense that the complexity of
  the output sheaves is bounded solely in terms of the complexity of
  the input sheaves. A key feature of complexity is that it provides
  bounds for the sum of Betti numbers that, in many interesting cases,
  can be made uniform in the characteristic of the base field. As an
  illustration, we discuss a few simple applications to horizontal
  equidistribution results for exponential sums over finite fields.
\end{abstract}


\maketitle
\setcounter{tocdepth}{1}
\tableofcontents

\section{Introduction}

Since its invention, and especially since Deligne's proof~\cite{weilII}
of the strongest form of the Riemann Hypothesis over finite fields,
étale cohomology has exerted a considerable influence on analytic number
theory. Its applications very often rely on estimates for the dimension
of various étale cohomology spaces, which appear in ``implicit
constants'' arising from the Grothendieck\nobreakdash--Lefschetz trace
formula combined with the Riemann Hypothesis and depend on the
characteristic~$p$ of the finite fields under consideration. This
characteristic is typically itself a variable going to infinity, and
getting uniform estimates in terms of~$p$ turns out to be the crucial
difficulty. Except for very simple cases, uniformity of such estimates
is \emph{not} a formal feature of étale cohomology (for examples and
further discussion, see~\cite[11.11]{ant}).  

In recent years, this issue has been particularly visible in a series of
works by Fouvry, Kowalski and Michel (see, e.g.,~\cite{fkm2}
and~\cite{fkm1}) that make extensive use of very general sheaves on
curves in various problems of analytic number theory. Due to the simpler
nature of curves (essentially, the Euler--Poincaré characteristic
controls the sum of Betti numbers, and has an expression in terms of
``simple'' local invariants), they obtained a satisfactory theory,
phrased in terms of a ``complexity'' invariant of an $\ell$-adic sheaf
on a curve over a finite field, which they called the ``(analytic)
conductor''. The key feature of this theory is that \emph{most natural
  operations on sheaves and the analytic resulting estimates} depend on
the ``input'' sheaves only through their conductor (see,
e.g.,~\cite[Th.\,1.5]{fkm2}).

Another application of a suitable version of complexity for
$\ell$-adic sheaves on curves is the proof by
Deligne~\cite{deligne-trace} of the theorem that, for a lisse
$\Ql$-adic Weil sheaf on a normal connected scheme of finite type over
a finite field such that the traces of Frobenius at all closed points are algebraic numbers, the field generated by these traces is a number field. The complexity that we introduce here could
also be used in Deligne's argument.

In this paper, we develop a similar theory for higher-dimensional
quasi\nobreakdash-projective algebraic varieties over any
field. (Being a geometric invariant, the complexity is defined by base
change to an algebraically closure of the base field, so most of this paper
will only deal with algebraically closed fields.) This leads to very general
estimates that solve most of the known problems of estimating Betti
numbers in analytic number~theory.

We now state somewhat informally the definition of complexity and some
of the key statements, focusing for simplicity on sheaves on affine
space~$\Aa^n$. Let~$k$ be an algebraically closed field and~$\ell$ a
prime number different from the characteristic of~$k$. We will define:

\begin{itemize}
\item A non-negative integer~$c(A)$ for any object~$A$ of the bounded
  derived category of constructible sheaves~$\Der(\Aa^n,\Ql)$
  (Definitions~\ref{def-complexity} and~\ref{def-ca}). Letting~$u$
  denote the open immersion of~$\Aa^n$ in~$\Pp^n$, the integer~$c(A)$
  is defined as the maximum over integers~$0\leq m\leq n$ of the sum
  of the Betti numbers of the pullback of the extension by zero~$u_!A$
  to a ``generic'' linear subspace of dimension~$m$ of~$\Pp^n$.
  
\item A non-negative integer~$c(f)$ for any morphism
  $f\colon \Aa^n\to \Aa^m$ (Definition~\ref{def-cuv}). In general, this
  is also defined in terms of sums of Betti numbers, but admits in the
  case at hand a completely explicit bound that only involves $n$, $m$
  and the degrees of the polynomials defining~$f$ (Proposition~\ref{pr-explicit-uniform}).
\end{itemize}

We will then prove the following result (Theorems~\ref{pushpull} and~\ref{theo-vancycles}). In the statement, $\dual(A)$
denotes the Verdier dual of~$A$ and all functors and operations are
considered in the derived sense (so, e.g., we write~$f_*$ instead of
$Rf_*$).

\begin{theo}\label{th-sample}
  For any~$f\colon \Aa^n\to\Aa^m$, for any objects~$A$ and~$B$
  of~$\Der(\Aa^n,\Ql)$, and any object~$C$ of~$\Der(\Aa^m,\Ql)$, the
  following estimates hold:
  \begin{gather*}
    c(\dual(A))\ll c(A),
    \\
    c(A\otimes B)\ll c(A)c(B),\quad\quad c(\rhom(A,B)) \ll c(A)c(B),
    \\
    c ( f^* C) \ll c(f) c(C),\quad\quad c( f^! C) \ll c(f) c(C),
    \\
    c(f_! A) \ll c(f) c(A), \quad\quad c (f_* A) \ll c(f) c(A).
  \end{gather*}
  In all these estimates, the implied constants depend only on~$(n,m)$
  and are effective.
  \par
  Moreover, let~$S$ be the spectrum of a strictly Henselian discrete
  valuation ring with special point~$\sigma$ and generic point~$\eta$,
  and let $\Psi$ and~$\Phi$ denote the nearby and vanishing cycle
  functors from~$\Der(\Aa^n_S,\Ql)$ to~$\Der(\Aa^n_{\sigma}, \Ql)$.  For any
  object $A$ of $\Der(\Aa^n_S,\Ql)$, the following estimates hold: 
  \begin{align*}
    c(\Psi(A))&\ll c(A_\eta),\\
    c(\Phi(A))&\ll c(A_\eta)+c(A_\sigma).
  \end{align*}
\end{theo}

Over finite fields, the conjunction of the Riemann Hypothesis and
the theory of complexity yields the following
``quasi-orthogonality'' statement (Theorem~\ref{th-rh}):

\begin{theo}\label{th-sample-Riemann} 
  Suppose that~$k$ is the algebraic closure of a finite
  field~$\Ff$, and let~$A$ and~$B$ be irreducible perverse sheaves on~$\Aa^n$
  defined over~$\Ff$ that are pure of weight zero, with trace
  functions~$t_A$ and~$t_B$ respectively. Then the estimate
  $$
  \sum_{x\in \Ff^n} |t_A(x)|^2 =1+O(c(A)^2|\Ff|^{-1/2})
  $$
  holds, and the estimate
  $$
  \sum_{x\in \Ff^n} t_A(x)\overline{t_B(x)}\ll c(A)c(B)|\Ff|^{-1/2}
  $$
  holds if~$A$ and~$B$ are not geometrically isomorphic. In both
  estimates, the implied constants depend only on~$n$ and are effective.
\end{theo}

\begin{remark}\label{rem:weightsforperverse}
  Readers from analytic number theory who are unfamiliar with perverse
  sheaves may be surprised by the lack of the averaging factor $1/|\Ff^n|$ in the writing of these sums,
  in comparison with statements like those in~\cite{fkm2}. This is due
  to the normalization inherent to the definition of weights in this
  setting: for instance, for a perverse sheaf $M$ on $\Aa^1$ that is a
  single lisse sheaf sitting in degree $-1$, being pure of weight zero
  means that the eigenvalues of Frobenius at all points have modulus
  $|\Ff|^{-1/2}$ (and not $1$, as is the case for a lisse sheaf that is
  pointwise pure of weight zero).
\end{remark}

We highlight one first rather simple application (see part (1) of
Corollary~\ref{cor-katz}), which gives a positive answer to a question
of Katz~\cite[p.\,8 and 12.6.6]{mmp}.

\begin{theo}\label{th-sample-2}
  Let~$n \geq 1$ and $d \geq 1$ be integers. Let $D(n,d)$ be the space
  of Deligne polynomials of degree~$d$ in~$n$ variables, i.e. those
  whose homogeneous part of degree~$d$ defines a non-singular
  hypersurface in $\Pp^{n-1}$. For each~$f\in D(n,d)(\Ff_p)$,~set
  $$
  S(f;p)=\frac{1}{p^{n/2}}\sum_{x\in\Ff_p^n}e\Bigl(\frac{f(x)}{p}\Bigr),
  $$
  where $e(z)=\exp(2i\pi z)$ for $z\in\Cc$.  The families
  $(S(f,p))_{f\in D(n,d)(\Ff_p)}$ become equidistributed as
  $p\to+\infty$ with respect to the image under the trace of the
  probability Haar measure on the unitary group $\Un_{(d-1)^n}(\Cc)$.
\end{theo}

We now comment on the approach that we use.  Previous Betti number bounds,
such as those of Bombieri~\cite{bombieribetti},
Adolphson--Sperber~\cite{ASbetti} and Katz~\cite{katzbetti}, focused
primarily on bounding cohomology groups involving certain very explicit
sheaves, namely Artin\nobreakdash--Schreier and Kummer sheaves. It is
possible to apply these bounds to a sheaf cohomology problem involving,
say, higher-rank Kloosterman sheaves, but only after unraveling their definition to recast the problem entirely in terms of
Artin\nobreakdash--Schreier sheaves. When more complicated operations are
performed (for example, additive or multiplicative convolution, or Fourier
transform), this process of re-interpretation becomes exceedingly
cumbersome.

Our approach is instead built around the six functors formalism of
étale cohomology, and is closely related to the characteristic classes
constructed by T.\,Saito~\cite{saito1}.  We define the ``complexity'' of
an arbitrary bounded complex of constructible $\ell$\nobreakdash-adic sheaves on a
quasi\nobreakdash-projective variety, and prove that it satisfies
essentially all desired properties suggested by the case of curves and
the requirements of applications to analytic number theory. In
particular, the complexity of common sheaves such as Artin-Schreier,
Kummer and Kloosterman sheaves can be easily calculated, and it turns
out to be bounded independently of the characteristic of the
underlying field, which is the key uniformity property that we seek.

\begin{remark}
  In fact, the complexity of a sheaf on an algebraic variety will also
  depend on a chosen quasi-projective embedding of the variety; this
  seems unavoidable to have a theory with good properties, as we
  explain in Example~\ref{ex-counter}.
\end{remark}

\begin{remark}
  The definition of complexity and the arguments of this paper apply,
  almost without modification, to the derived category of sheaves with
  coefficients in~$\overline{\mathbf F}_\ell$ instead of~$\Ql$. Neither
  version is stronger. Although the Betti numbers of a~$\Ql$\nobreakdash-sheaf are bounded
  by the Betti numbers of the reduction mod $\ell$ of an integral model
  of it, this inequality does not help us transfer statements of the
  form ``a bound for the Betti numbers of this sheaf implies a bound for
  the Betti numbers of that sheaf'' in either direction. We have stated and worked out in detail the~$\Ql$\nobreakdash-version as it is the most directly relevant for applications to analytic number theory, but the $\overline{\mathbf F}_\ell$-version may also be useful for other purposes.
\end{remark}

We believe that this framework has a number of good properties, among which:
\begin{enumerate}
\item Since the deeper aspects of étale cohomology are built primarily
  around the six functors perspective, rather than the cohomology of
  varieties with coefficients in some simple explicit sheaves, this
  framework behaves much better in arguments where sophisticated
  techniques of étale cohomology are used.
\item Many applications of exponential sum bounds from \'{e}tale
  cohomology revolve around exponential sums that are produced from
  simpler ones by applying analytic tools like changes of variables,
  summation over some variables, Fourier transform, etc. Through the
  ``function-sheaf dictionary'', each of these usually corresponds to
  an operation on the sheaf side, which is constructed by means of the
  six functors (e.g., summation corresponds to direct image with compact support,
  etc).  Since we control the growth of the complexity under the six
  functors, we obtain automatically a good control of the estimates in
  such operations.
\end{enumerate}

As we will see, almost all of the bounds for the complexity of the
output sheaf of some cohomology operation are linear in the complexity
of the input sheaf. This is not always needed for
applications, but shows that the theory has good structural properties.

\subsection*{An interpretation}

As suggested by Fouvry, Kowalski and Michel in the special case of
curves, the ``quantitative sheaf theory'' that is developed in this
paper can be thought of as defining the complexity of 
$\ell$-adic sheaves in such a way that most (if not all) usual
operations in étale cohomology are ``continuous'', in the sense that
applying the operation to a sheaf with a given complexity will
lead to another one with complexity bounded only in terms
of the initial one. Thus, we think of the complexity as being similar to a
(semi)-norm on a topological vector space, with functors on categories
of sheaves playing the role of (often linear) maps between vector
spaces. For instance, the ``continuity'' of
Deligne's~$\ell$\nobreakdash-adic Fourier transform (which was first observed in
dimension one in~\cite[Prop.\,8.2]{fkm1}) turns out to be one of the
most essential features of applications of étale cohomology to
analytic number theory.

\subsection*{Outline of the paper.}

Although the complexity is defined in terms of sums of Betti numbers, the proof of its main properties deeply relies on T.\,Saito's
construction~\cite{saito1} of the characteristic cycle of $\ell$-adic
complexes. We survey what we require from this theory in
Section~\ref{sec-saito}, and prove a small complement on
characteristic cycles of tensor products (Theorem~\ref{geneufor}). In Section~\ref{sec-generic}, we formally define the complexity on projective space, we establish a few simple
lemmas concerning ``generic'' injective linear maps, and most
importantly we connect this approach with the characteristic cycle (Proposition~\ref{bilinearbettibound}). Section~\ref{sec-test} is of
technical nature: we define and prove the existence of certain objects
called ``test sheaves'' that will ultimately lead to a comparison of
the complexity with a norm of the characteristic
cycle. Section~\ref{sec-tensor} uses these tools to establish the
first fundamental result, namely a bilinear bound for the complexity
of the tensor product (Theorem~\ref{main}). Then
Section~\ref{sec-6ops} can rather quickly exploit the formalism of
étale cohomology to establish the general version of
Theorem~\ref{th-sample}, namely Theorems~\ref{pushpull} and~\ref{theo-vancycles}; later
subsections derive various other ``continuity'' properties. Finally,
Section~\ref{creation} gives some fundamental examples (such as
Artin--Schreier and Kummer sheaves) and summarizes a few direct
applications (including forms of the Riemann Hypothesis, such as Theorem~\ref{th-sample-Riemann}, the finiteness statement of
Corollary~\ref{cor-finiteness}, and a form of Deligne's
equidistribution theorem, from which Theorem~\ref{th-sample-2} folllows). In the concluding Section~\ref{sec-effective},  we explain how all the basic estimates can be stated with explicit constants. 

\subsection*{Notation and conventions}

\subsubsection*{Algebraic geometry}
We fix throughout a prime number $\ell$ and we denote by $k$ a field, algebraically closed unless otherwise specified, in which $\ell$ is invertible. 

By an \emph{algebraic variety} over a (not necessarily algebraically
closed) field~$k$, we mean a reduced and separated scheme of finite
type over the spectrum of~$k$.

By a \emph{geometric generic point} of an irreducible variety $X$ over a separably closed
field~$k'$, we mean, as is customary in the theory of \'{e}tale cohomology, a map $\Spec k' \to X$ such that the image of the underlying set-theoretic map consists of the generic point of $X$.

Let~$X$ be a scheme of finite type over~$k$. We denote by~$\Der(X)$
the bounded derived category of constructible complexes of $\Ql$-sheaves on~$X$
(see, e.g.,~\hbox{\cite[II.5]{KW_weilconj}}).
We will usually write distinguished triangles in this category simply as
\[
A\longrightarrow B\longrightarrow C.
\]

For any object~$K$ of~$\Der(X)$ and any integer~$i\in\Zz$, we denote
by~$\rmH^i(X,K)$ and $\rmH^i_c(X,K)$ the étale
cohomology and the étale cohomology with compact support groups of~$X$ with coefficients in~$K$, and we write
\[
h^i(X,K)=\dim \rmH^i(X,K),\quad\quad h^i_c(X,K)=\dim \rmH^i_c(X,K)
\]
for the corresponding Betti numbers. 

When applied to objects of~$\Der(X)$, the symbols $f_!$ and $f_*$
always refer to the derived functors; the tensor product and the hom
functor of objects of $\Der(X)$ are also always derived functors. We
denote by $\dual(A)$ the Verdier dual of an object $A$ of~$\Der(X)$.

Given an algebraic variety~$X$ over~$k$ and objects $A$ and $B$ of
$\Der(X)$, the \emph{shriek tensor product} of~$A$ and~$B$ is the
object
\[
A\otimes_! B=\Delta^!(A\boxtimes B)
\]
of $\Der(X)$, where $\Delta\colon X\to X\times X$ denotes the diagonal
embedding. It is related to the usual tensor product by the duality 
\[
\dual(A\otimes_! B)=\dual(A) \otimes \dual(B). 
\]

We often use the projection formula in the derived category: for a morphism
$f\colon X\to Y$ of algebraic varieties over~$k$, and for objects~$A$ of~$\Der(X)$ and~$B$
of~$\Der(Y)$, there is a canonical isomorphism
\begin{displaymath}
  f_!(A\otimes f^* B)\simeq f_!A\otimes B
\end{displaymath} in the category~$\Der(Y)$ (see, e.g., \cite[Th.\,7.4.7\,(i)]{fu_etale_co}).

We also recall the excision triangle: let
$i\colon Z\to X$ be a closed immersion and~\hbox{$j\colon U\to X$}
the complementary open immersion, all varieties being defined over~$k$. For any
object~$A$ of~$\Der(X)$ and any morphism~$f\colon X\to Y$ over~$k$, there is a distinguished triangle
\begin{displaymath}
  (f\circ j)_!j^*A\longrightarrow f_!A\longrightarrow (f\circ i)_!i^*A
\end{displaymath}
in the category~$\Der(Y)$ (see, e.g.,~\cite[Th.\,7.4.4\,(iii)]{fu_etale_co}).

\subsubsection*{Finite fields}

In some sections (e.g., Sections~\ref{sec-indep} and~\ref{sec-rh}), we
will work over finite fields.  We usually denote by~$\Ff$ such a 
field, which is always assumed to have characteristic different from~$\ell$.
For integers $n\geq 1$, we then denote by~$\Ff_n$ the extension
of~$\Ff$ of degree~$n$ inside some fixed algebraic closure of~$\Ff$ (which
often will be the field~$k$).

Let~$X$ be an algebraic variety over~$\Ff$.  For any object~$A$
of~$\Der(X)$ and any finite extension~$\Ff_n$ of~$\Ff$, we denote by
\[
t_A(\cdot;\Ff_n) \colon X(\Ff_n)\longrightarrow \Ql
\]
the \emph{trace function} of~$A$ on~$\Ff_n$. We refer the reader to~\cite[\S
1]{laumon} for the basic formalism of trace functions in this
generality. We will also write $t_A(x)=t_A(x;\Ff)$ for~\hbox{$x\in X(\Ff)$}.

In all arguments involving the formalism of weights (in the sense of
Deligne), we will fix an isomorphism \hbox{$\iota\colon \Ql\to \Cc$}
and use it to identify both fields, viewing in particular the trace
functions as taking complex values. Weights are then considered to be
defined only with respect to~$\iota$, e.g. we write ``pure of weight
zero'' instead of ``$\iota$-pure of weight zero''.

For a real number $w$ and an element $|\Ff|^{-w}$ in $\Ql$
corresponding to a choice of~$|\Ff|^{-w}$ in~$\Cc$ through the
isomorphism $\iota$, we denote by $\Ql(w)$ the pullback to $X$ of the
rank\nobreakdash-one~$\ell$\nobreakdash-adic sheaf on $\Spec(\Ff)$ on
which Frobenius acts through multiplication by $|\Ff|^{-w}$.  This
allows one to define twists $A(w)=A \otimes \Ql(w)$ for any object~$A$
of $\Der(X)$, and hence to reduce questions about pure sheaves of some
weight to those of weight zero.

\subsubsection*{The Euler--Poincaré characteristic}
We recall the Euler--Poincaré characteristic formula for a perverse
sheaf on a smooth curve (see, e.g.,~\cite[Th.\,2.2.1.2]{laumon} for the
projective case, from which the general case below follows by considering the
extension by zero to the compactification).  Let $k$ be an algebraically
closed field of characteristic $p>0$, let~$C$ be a smooth curve
over~$k$, and denote by~$\overline{C}$ the smooth projective
compactification of~$C$. Given a perverse sheaf~$A$ on~$C$, we denote
by
$$
\rank(A) = \dim \mathcal H^{-1}(A)_\eta
$$
the \emph{generic rank} of $A$ and, for a closed point $x$ of $C$, we write
\begin{align*}
  \drp_x (A)&=\rank(A) - \dim \mathcal H^{-1}(A)_x
  \\
  \jump_x (A)&=\dim \mathcal H^0(A)_x.
\end{align*}
For a closed point $x$ of $\overline{C}$, we denote by $\swan_x(A)$ the
\emph{Swan conductor} at $x$ of the cohomology sheaf $\mathcal H^{-1}(A)$. 
We further set
$$
\loc(A)= \sum_{x \in C} \left( \drp_x (A)+ \jump_x (A) + \swan_x(A)
\right) + \sum_{x \in \overline{C}- C } \swan_x(A), 
$$
where both sums run over closed points. 

With the above notation, the Grothendieck--Ogg--Shafarevitch formula for the
Euler\nobreakdash--Poincaré characteristic with compact support of $A$ takes the form
\begin{displaymath} 
\chi_c(C,A)=\rank(A)\chi_c(C,\Ql[1])-\loc(A).
\end{displaymath} The same result holds for the usual Euler--Poincaré characteristic $ \chi(C,A)$ since both are in fact equal for any constructible sheaf on any variety by a theorem of Laumon \cite{LaumonEuler}. 

If~$\mathcal{F}$ is a middle-extension sheaf on~$C$ (by which we mean that there exists a non-empty open subset $j \colon U \hookrightarrow C$ such that $\mathcal{F}$ is lisse on $U$ and the adjunction morphism is an isomorphism $\mathcal{F}\simeq j_\ast j^\ast \mathcal{F}$), then
$A=\mathcal{F}[1]$ is a perverse sheaf on~$C$ satisfying
$\mathcal{H}^{-1}(A)=\mathcal{F}$. We use the notation
$\rank(\mathcal{F})$, $\drp_x(\mathcal{F})$, $\swan_x(\mathcal{F})$
and $\loc(\mathcal{F})$ accordingly; note that in this
case~$\jump_x(A)=0$ holds for all~$x$.

\subsubsection*{Asymptotic notation}

For complex-valued functions~$f$ and~$g$ defined on a set (or on 
objects of a category, in which case the values of~$f$ and~$g$ are
assumed to only depend on their isomorphism classes), the notation
$f\ll g$ and $f=O(g)$ are synonymous; they mean that there exists a
real number $c\geq 0$ such that, for all $x$ in the relevant set (or
all objects in the category), the inequality $|f(x)|\leq cg(x)$ holds. We call a
value of~$c$ an ``implied constant'', and we may point out its
(in)dependency on some additional parameters. We also write
$f \asymp g$ whenever both $f \ll g$ and $g \ll f$ hold.


\subsection*{Remarks on the text.}

All the important ideas of this paper are solely due to W.\,Sawin, and
were worked out in 2015 and 2016 while he was an ETH--ITS Junior
Fellow.  The current text was written by A.\,Forey, J.\,Fresán and
E.\,Kowalski, based on the original draft by W.\,Sawin, in view of the
applications to equidistribution results for exponential sums on
commutative algebraic groups in their work~\cite{ffk}. The
quantitative form of the generic base change theorem is due to
A.\,Forey.

\subsection*{Acknowledgements.} We warmly thank the anonymous referees for their thorough and friendly reading of a first version of this paper. Their many suggestions greatly helped us improve the presentation and remove a number of inaccuracies.


\section{Characteristic cycles}\label{sec-saito}

In this section, we recall some properties of the characteristic cycles and the characteristic classes of complexes of étale sheaves as defined by Beilinson \cite{beil_SS} and T. Saito~\cite{saito1}. We also prove a small complement (Theorem \ref{geneufor}) regarding the compatibility of characteristic cycles with tensor products. 

Let $X$ be a smooth scheme over a perfect field $k$. In \cite{saito1},
characteristic cycles are defined for complexes of sheaves of
$\Lambda$-modules on $X$, where $\Lambda$ is a finite local ring whose residue characteristic $\ell$ is invertible in $k$ (e.g., $\Lambda=\Zz/\ell^n\Zz$). It was observed by
Umezaki, Yang and Zhao~\cite[Section 5]{CC-epsilon} that the whole
theory can be readily adapted to the case of $\Ql$-coefficients. In
what follows, we implicitly rely on \cite{CC-epsilon} in order to
apply Saito's results to $\Lambda=\Ql$.

Assume $X$ has pure dimension $n$. The \emph{characteristic
  cycle} $\CC(A)$ of an object~$A$ of~$\Der(X)$ is an algebraic cycle on
the cotangent bundle $T^*X$ of $X$ of the form
\[
\CC(A)=\sum_i m_i C_i, 
\] 
where $m_i$ is an integer and $C_i$ is a closed conical (i.e. stable
under the natural action of~$\Gg_m$) subset of $T^*X$ of dimension
$n$. We refer the reader to \cite[Def.\,5.10]{saito1} for the definition
of the multiplicities~$m_i$, and simply indicate some of the relevant
properties of the characteristic cycle. 

Another invariant of $A$ is its \emph{singular support} $\SS(A)$, which
is a closed conical subset of~$T^* X$ of dimension $n$, containing the
support of the characteristic cycle, previously defined by Beilinson in
\cite{beil_SS}. Unlike what happens for the characteristic cycle, two
$\overline{\mathbf Z}_\ell$\nobreakdash-sheaves that are isomorphic
after tensoring with $\Ql$ may have distinct singular supports, so before stating results involving the singular support we will choose a $\overline{\mathbf Z}_\ell$-structure
$A_{\bar{\Zz}_{\ell}}$ (though the particular choice will matter little), and set
$$
\SS(A)=\SS(A_{\bar{\Zz}_{\ell}}\otimes\bar{\mathbf{F}}_{\ell}).
$$

If the object $A$ is perverse, then $\CC(A)$ is effective (i.e., we have
$m_i\geq 0$ for every $i$) by~\cite[Prop.\,5.14]{saito1}.

The main property of the characteristic cycle is an index formula (\cite[Th.\,7.13]{saito1}) according to which, if $X$ is a smooth
projective variety over an algebraically closed field~$k$, the
Euler--Poincaré characteristic of $A$ is given by the
intersection number
\begin{displaymath}
\chi(X,A)=\CC(A)\cdot [T^*_X X]. 
\end{displaymath}
In this formula, $T^*_X X$ denotes the zero section of $T^*X$, and the
intersection is well-defined since the support of $\CC(A)$ has pure
dimension $n$.

\begin{example} Let $X$ be a smooth projective curve and $\mathcal{F}$ a
  lisse sheaf on a dense open subset $j\colon U\to X$. The
  characteristic cycle of the perverse sheaf $A=j_!\mathcal{F}[1]$ is
  given by
\[
  \CC(A)=\rank \mathcal{F}\cdot[T^*_X X]+\sum_{x\in X\backslash
    U} (\rank \mathcal{F}+\swan_x A)[T^*_x X],
\]
where $T^*_x X$ stands for the conormal bundle of $x$ in $T^*X$.  In
this situation, the index formula amounts to the
Grothendieck--Ogg--Shafarevich formula.
\end{example}

Another deep result is the compatibility of the characteristic cycle
with pullbacks. Let~\hbox{$h \colon W\to X$} be a morphism between
smooth schemes of pure dimensions $m$ and~$n$ respectively and
$C\subset T^*X$ a closed conical subset. Let
\[
h^*C=W \times_X C \subset W \times_X T^*X
\]
be the pullback of $C$, and denote by $K \subset W \times_X T^*X$ the
preimage of the zero section~$T^*_W W$ by the map
$dh\colon W\times_X T^*X\to T^*W$. Following \cite[Def.\,3.3 and
7.1]{saito1}, we say that the map~$h$ is \emph{properly
  $C$\nobreakdash-transversal} if the intersection~$h^*C \cap K$ is a
subset of $W\times_X T^*_XX$ and each irreducible component of $h^*C$
has dimension $m$. This implies in particular that the restriction of
$dh$ to $h^*C$ is finite (by \cite[Lem.\,1.2.\,(ii)]{beil_SS} or
\hbox{\cite[Lem.\,3.1]{saito1}}). Hence, given a cycle~$Z$ supported on
the irreducible components of $C$, we can define its
pullback~$h^!Z \subset T^*W$ as $(-1)^{m-n}$ times the push-forward
along the finite map~$h^*C\to T^*W$ of the preimage of $Z$ in
$W\times_X T^*X$ (see \hbox{\cite[Def.\,7.1]{saito1}}).

The main theorem of Beilinson and Saito \cite[Th.\,7.6]{saito1} regarding pullbacks says that, if~$A$ is an object of $\Der(X, \overline{\mathbf Z}_\ell)$ and $h\colon W\to X$ is a properly $\SS(A)$-transversal morphism, then the following equality holds:
\[
\CC(h^*A)=h^!\CC(A). 
\]
For example, smooth morphisms are  properly $\SS(A)$-transversal by~\hbox{\cite[Lem.\,3.4]{saito1}}, and in that case the theorem essentially amounts to the statement that, for smooth varieties $X$ and $Y$ and objects $A$ of $\Der(X)$ and $B$ of $\Der(Y)$, the~equality  
\[
\CC(A\boxtimes B)=\CC(A)\times \CC(B),
\]
holds in $T^\ast(X \times Y)=T^\ast X \times T^\ast Y$, see \cite[Th.\,2.2]{saito2} for details. 

In what follows, vector bundles on schemes are viewed as schemes in the usual way. A sum of vector bundles $V + W$ is isomorphic to their fiber product over the base, whereas the product  $V \times W$ is their product as schemes, and hence is a vector bundle on the product of the bases. As for $T^*X$, a closed subset of a vector bundle is said to be \emph{conical} if it is invariant under the scaling action of $\Gg_m$. Given a vector bundle $V$ on a variety~$X$, we denote by $\overline{V} = \Pp (V + \mathcal O_X)$ the projective bundle compactifying $V$, which admits a decomposition $\overline{V} = V \cup \Pp(V)$.

The main result of this section is the following theorem concerning characteristic cycles of tensor products of sheaves. 

\begin{theo}\label{geneufor} Let $X$ be a smooth variety over $k$ of pure dimension $n$. Let $A$ and $B$ be objects of $\Der(X,\overline{\mathbf Z}_\ell)$. 
Consider the summation and the inclusion maps  
\[
s\colon T^*X + T^*X \to T^*X, \qquad i\colon  T^*X + T^*X \to T^*X \times T^*X.
\]
Assume that $\SS(A) \cap \SS(B)$ is supported on the zero section and that each irreducible component of $\SS(A) \times_X \SS(B)$ has dimension at most $n$. Then:
\begin{enumerate}
\item\label{char:item1} Each irreducible component of $\SS(A) \times_X \SS(B)$ has dimension equal to $n$.
\item\label{char:item2} The equality $i^*(\SS(A) \times \SS(B))=\SS(A) \times_X \SS(B)$ holds. 
\item\label{char:item3} The restriction of $s$ to $\SS(A) \times_{X} \SS(B)$ is a finite map to $ T^*X$.
\item\label{char:item4}  The following equality holds
 \begin{displaymath}
 \CC(A \otimes B)=(-1)^n s_* i^* \left(\CC(A) \times \CC(B)\right),
 \end{displaymath} where the pullback and the pushforward are taken in the sense of intersection theory. (By \eqref{char:item1} and \eqref{char:item2}, the inverse image $i^* \left(\SS(A) \times \SS(B)\right)$ has the expected dimension, so the intersection-theoretic pullback is well-defined.)
\item\label{char:item5} Assume that $X$ is projective. Let $\overline{\CC(A)}$ and $\overline{\CC(B)}$ be the closures of $\CC(A)$ and $\CC(B)$ inside the projective bundle $\overline{T^*X}$. Then the equality
\begin{displaymath}
\chi(X,A \otimes B) =(-1)^n \CC(A) \cdot \CC(B) =(-1)^n \overline{\CC(A)} \cdot \overline{\CC(B)}
\end{displaymath}
holds, where the dots denote intersection numbers of algebraic cycles.
\end{enumerate}
\end{theo}

Note that proving properties \eqref{char:item1},  \eqref{char:item2} and  \eqref{char:item3} amounts to checking that the diagonal map is properly $\SS(A) \times \SS(B)$-transversal. The proof of the theorem relies on the following lemma: 

\begin{lem}\label{simpinteq} Let $X$ be a variety over $k$ and let $V$ be a vector bundle on $X$. Let $C_1,C_2 \subset V$ be conical subsets and $\overline{C_1}$ and $\overline{C_2}$ their closures inside the projective bundle $\overline{V}$.  The following two conditions are equivalent:
\begin{enumerate}
\item\label{equiv:cond1} $C_1 \cap C_2 $ is contained in the zero section of $V$.
\item\label{equiv:cond2} $\overline{C_1} \cap \overline{C_2}$ does not intersect $\Pp(V)$ inside $\overline{V}$.
\end{enumerate}
\end{lem}

\begin{proof} The intersection $\overline{C_1}\cap \overline{C_2}$ is proper, as a closed subset of $\overline{V}$. Taking the decomposition $\overline{V} = V \cup \Pp(V)$ into account, if condition \eqref{equiv:cond2} holds, then $\overline{C_1}\cap \overline{C_2}$ is also affine, being a closed subset of $V$. Therefore, $\overline{C_1}\cap \overline{C_2}$ is a finite conical subset of $V$, and is hence contained in the zero section. Conversely, \eqref{equiv:cond1} implies \eqref{equiv:cond2} because if $x$ is a point of $\overline{C_1} \cap \overline{C_2}$ that lies in $\Pp(V)$, then the line in $V$ corresponding to $x$ is contained in $C_1 \cap C_2$, which therefore is not just the zero section.
\end{proof}

\begin{proof}[Proof of Theorem \ref{geneufor}] 

We first prove statement \eqref{char:item1}. Since $\SS(A) \times_X \SS(B)$ is the intersection of the inverse images of $\SS(A)$ and $\SS(B)$ inside $T^*X + T^*X$ (i.e. the intersection of two $2n$\nobreakdash-dimensional schemes inside a smooth variety of dimension $3n$), each of its irreducible components has dimension at least $n$, and hence equal to $n$ by assumption. 

Statement \eqref{char:item2} follows from the fact that pullback by the map $i$ identifies the sum $T^*X + T^*X$ with $T^*X \times_X T^*X$ seen as a closed subset of $T^*X \times T^*X$. In the rest of the proof, we will make this identification. 

To prove \eqref{char:item3}, observe that $s \colon T^*X+T^*X \to T^*X$ is a map of vector bundles on~$X$.  Since $\SS(A)$ and $\SS(B)$ are conical  and $\SS(A)\cap \SS(B)$ is a subset of the zero section~$T^*_X X$ of~$T^*X$, the intersection of $\SS(A) \times_{X} \SS(B)$ and the pullback along~$s$ of $T^*_X X$ is contained in the zero section of $T^*X+T^*X$. Hence, the restriction of~$s$ to $\SS(A) \times_{X} \SS(B)$ is finite by \cite[Lem.\,3.1]{saito1}. 

We now turn to \eqref{char:item4}. Recall that the characteristic cycle of the external product~$A \boxtimes B$ is equal to $\CC(A) \times \CC(B)$ and that the tensor product $A \otimes B$ is given by~$\Delta^*(A \boxtimes B)$, where~$\Delta\colon X \to X \times X$ denotes the diagonal map. We will compute $\CC(A \otimes B)$ using the compatibility of characteristic cycles with pullbacks, as recalled above. For this, we need to show that $\Delta$ is a properly $\SS(A) \times \SS(B)$-transversal morphism. 

From part \eqref{char:item2}, we know that every irreducible component
of $\Delta^* (\CC(A) \times \CC(B))$ has dimension $n$. For the
transversality condition, we need to consider the maps
\[
T^*X \stackrel{s}{\longleftarrow} T^*X+T^*X\stackrel{i}{\longrightarrow} T^*(X\times X)
\]
defined by the diagonal $\Delta$, and check that the intersection of
$i^\ast(\SS(A) \times \SS(B))$ with the preimage of the zero section
$T^*_XX$ by the map $s$ is contained in the zero section of
$T^*X+T^*X$. But this is precisely the condition we checked in the proof
of \eqref{char:item3}, and is in fact equivalent to \eqref{char:item3}
by \cite[Lem.\,3.1]{saito1}. Hence, \cite[Th.\,7.6]{saito1} applies to
$\Delta$ and $A\boxtimes B$, which yields the equality
\[
\CC(\Delta^*(A\boxtimes B))=\Delta^! \CC(A\boxtimes B).
\] Combining this with the first part of the proof and the definition of $\Delta^!$, we get
\begin{align*}
\CC(A\otimes B) &=\CC(\Delta^*(A\boxtimes B))=\Delta^! \CC(A\boxtimes B)=(-1)^n s_* i^* (\CC(A)\times \CC(B)).
\end{align*}

Let us finally prove \eqref{char:item5}. By the index formula \cite[Th.\,7.13]{saito1}, the Euler--Poincaré characteristic of~$A \otimes B$ is the intersection number of $\CC(A \otimes B)$ with the zero section $[T^*_X X]$. Using part~\eqref{char:item4}, it is hence given by 
\[
(-1)^n  [T^*_X X] \cdot s_* i^* ( \CC(A) \times \CC(B)).
\] By the projection formula in intersection theory (see, e.g.,~ \cite[Prop.\,8.3\,(c)]{fulton}), the equality 
\[
[T^*_X X] \cdot s_* i^* ( \CC(A) \times \CC(B))=  s^* [T^*_X X] \cdot  i^* ( \CC(A) \times \CC(B))
\] holds, as long as the restriction of $s$ to the support of $i^* ( \CC(A) \times \CC(B))$ is a proper map. Taking into account that the support of $i^* (\CC(A) \times \CC(B))$ is $\SS(A) \times_X \SS(B)$, this is indeed the case by part~\eqref{char:item3}. Now, $s^* [T^*_X X]$ is the closed subset in $T^*X + T^*X$ consisting of those elements with zero sum. Being conical, $\CC(B)$ is in particular invariant under multiplication by $-1$, hence the equality 
\[
s^* [T^*_X X] \cdot  i^* ( \CC(A) \times \CC(B))=  d_*[T^*X] \cdot i^* ( \CC(A) \times \CC(B)), 
\] where $d\colon T^*X \to T^*X + T^*X$ is the diagonal map. Since $d$ is a closed immersion, the projection formula yields
\[
d_* [T^*X] \cdot  i^* ( \CC(A) \times \CC(B)) = [T^*X] \cdot d^* i^* (\CC(A) \times \CC(B)). 
\] Furthermore, $i \circ d$ is the diagonal map $\delta\colon T^*X \to T^*X \times T^*X$, so the equality
\begin{align*}
[T^*X] \cdot d^* i^* (\CC(A) \times \CC(B))&= [T^*X] \cdot \delta^* (\CC(A) \times \CC(B)) \\
&= \CC(A) \cdot \CC(B)
\end{align*} holds. Combining all these identities we get the first equality of \eqref{char:item5}. 
Finally, 
\[
\overline{\SS(A)} \cap \overline{\SS(B)}= \SS(A) \cap \SS(B)
\] follows from Lemma \ref{simpinteq}, and this implies the equality of intersection numbers 
\[
  \CC(A) \cdot \CC(B) =\overline{\CC(A)} \cdot \overline{\CC(B)},
\] which completes the proof. 
\end{proof}

The following definition is taken from Saito \cite[Def.\,6.7]{saito1} and relies on the isomorphism 
  \[
    \CH_*(\Pp^n)\longrightarrow  \CH_n\left( \overline{ T^*\Pp^n}\right), \quad (a_i)\longmapsto \sum_i p^*a_ih^i, 
  \] where $p\colon \overline{ T^*\Pp^n}\to \Pp^n$ is the 
  projection and~$h$ is the first Chern class of the dual of the
  universal sub line bundle of
  $\overline{T^*\Pp^n}\times_{\Pp^n} (T^*\Pp^n+\mathcal{O}_{\Pp^n})$
  (see~\cite[(6.12)]{saito1}).

\begin{defi}[Characteristic class]\label{def-cc}
The \emph{characteristic class} $\cc(A)\in \CH_*(\Pp^n)$ of an object
  $A$ of~$\Der(\Pp^n)$ is the image of $\overline{\CC(A)} \in \CH_n\left( \overline{ T^*\Pp^n}\right)$ under the
  inverse of the above isomorphism.
  \par
  The total Chow group $\CH_*(\Pp^n)$ is isomorphic to $\Zz^{n+1}$, with
  generators the classes of linear subspaces of dimension $0$ to
  $n$. Using this, we will view $\cc(A)$ as an element
  of~$\Zz^{n+1}$. The lattice $\Zz^{n+1}$ inherits the intersection
  pairing of $\CH_n\left( \overline{ T^*\Pp^n}\right)$, and this pairing
  is independent of the base field $k$.
\end{defi}

\begin{lem}\label{constantsheafbasis} Let $0 \leq m \leq n$ be an integer, and let $ K_m$ denote the constant $\Ql$\nobreakdash-sheaf
  on an $m$\nobreakdash-dimensional linear subspace of $\Pp^n$, extended
  by zero and placed in degree $-m$. Then the characteristic cycle
  $\CC( K_m)$ is the conormal bundle of that $m$-dimensional subspace,
  and the characteristic classes $\cc(K_0), \dots, \cc( K_n)$ form a
  basis of $\CH_*(\Pp^n)$. In particular,
  $\cc(K_0),\dots,\cc(K_n)$ form a basis of $\mathbf{Z}^{n+1}$ independent of
  $k$ and $\ell$.
\end{lem}

\begin{proof} The first statement follows from \cite[Lem.\,5.11(1) and
  5.13(2)]{saito1}. Recall that $\CH_*(\Pp^n)$ has rank $n+1$, so to
  show that the characteristic classes $\cc( K_0)$, \ldots, $\cc( K_n)$
  form a basis, it suffices to prove that they are linearly
  independent. We will do this by proving that the matrix of
  intersection numbers $\cc( K_i)\cdot \cc( K_j)$ is invertible. If
  $i+j< n$, then generic subspaces of dimensions $i$ and $j$ do not
  intersect, so neither do their conormal bundles, and hence the
  intersection number is $0$. If $i+j=n$, then a generic
  $i$\nobreakdash-dimensional subspace and a generic $j$-dimensional
  subspace intersect transversely at a single point, so their conormal
  bundles intersect transversely at a single point as well, and hence
  have intersection number $1$. The intersection matrix is thus
  invertible. The last sentence is just a restatement of this, except
  for the independence of~$k$ and $\ell$, which follows from observing
  that the isomorphism \cite[(6.12)]{saito1}, applied to the closure of
  the conormal bundle of an $m$-dimensional subspace, can be defined
  integrally and so is independent of the characteristic.
 \end{proof}


\section{Complexity and generic linear maps}\label{sec-generic}

In this section, we define the complexity of a complex of sheaves on projective space and we establish a few results about generic linear maps between projective spaces and their relationship with characteristic cycles.

\begin{defi} Let $k$ be a field.  Let $0 \leq m \leq n$ be integers. Let~$M^{n+1,m+1}_k$ be the variety of $(n+1) \times (m+1)$ matrices of maximal rank, so that in particular $M^{n+1,n+1}_k$ is equal to $\GL_{n+1, k}$. Given an extension $k'$ of $k$ and a $k'$-point $\mtrx{a} \in M^{n+1,m+1}(k')$, we denote by~$l_{\mtrx{a}}\colon \Pp^m_{k'} \to \Pp^n_{k'}$ the associated linear map. 
\end{defi}

In the case $m=n$, in which $\mtrx{a}$ is an invertible matrix and $l_{\mtrx{a}}$ is an automorphism, we may also use $l_{\mtrx{a}}$ to refer to the induced map on any scheme depending functorially on $\Pp^n$ (e.g., its cotangent bundle).

Here is the key definition of this article. Although it is closely related to characteristic cycles, as we will see, we have chosen to define it rather in terms of Betti numbers, so it can be used more directly in situations where some Betti numbers can be computed. 

\begin{defi}[Complexity of a complex of sheaves on projective space]\label{def-complexity}
  Let $k$ be a field and let $n \geq 0$ be an integer. For each $0 \leq m \leq n$, let $\mtrx{a}_m$ be a geometric generic point of $M^{n+1,m+1}_k$ defined over an algebraically closed field $k'$. The \emph{complexity} $c(A)$ of an object $A$ of $\Der(\Pp^n_k)$ is the non-negative integer
  \[ c(A)=\max_{0\leq m\leq n} \sum_{i \in \Zz} h^i(\Pp^m_{k'},
    l_{\mtrx{a}_m}^* A).\]
\end{defi}

When working with this definition, we will often use the projection formula
to write
\[h^i(\Pp^m_{k'}, l_{\mtrx{a}_m}^* A) = h^i(\Pp^n_{k'}, A \otimes
  l_{\mtrx{a}_m *} \Ql),
\]
which gives a slightly different expression for the complexity. Since
étale cohomology is invariant under base change of algebraically
closed fields, the value of~$c(A)$ is independent of the choice of the field of definition $k'$, and we will often drop it
from the notation.

Instead of considering geometric generic points of~$M^{n+1,m+1}$ in the definition of complexity, one can alternatively consider closed points in a suitable dense open subset of~$M^{n+1,m+1}$, as the following lemma shows. 

\begin{lem}\label{lem-complexity-open-set}
  Let $k$ be a field and $\bar k$ an algebraic closure of $k$. Let
  $0 \leq m \leq n$ be an integer, and let $A$ be a complex on
  $\Pp^n_k$. There exists a dense open subset $U\subset M^{n+1,m+1}_k$
  such that the equality
    \begin{displaymath}
     h^i(\Pp^{m}_{k'},l_\mtrx{a}^*A)=h^i(\Pp^{m}_{\bar k},l_\mtrx{b}^*A)
  \end{displaymath} holds for any geometric generic point
  $\mtrx{a}$ of $M^{n+1,m+1}$ defined over an algebraically closed field $k'$ and for every $\mtrx{b}\in U(\bar k)$ and $i\in \Zz$. 
\end{lem}

\begin{proof}
  For $\mtrx{b}\in M^{m+1,n+1}(\bar k)$, we view the linear map $l_\mtrx{b}$ as the composition of the map~$x \mapsto (x,\mtrx{b})$ from $\Pp^m$ to $\Pp^m \times M^{m+1,n+1}$ with the matrix multiplication map
 \[
 \mathrm{mult} \colon \Pp^m \times M^{m+1,n+1} \to \Pp^n.
 \] Let $p\colon \Pp^m \times M^{m+1,n+1} \to M^{m+1,n+1}$ be the projection.  Let $U\subset M^{n+1,m+1}_k$ be a dense open set such that the 
complex $p_* \mathrm{mult}^* A$ has lisse cohomology sheaves on~$U$. It follows from the proper base change theorem that the equality
  \[
  h^i(\Pp^{m}_{k'},l_\mtrx{a}^*A)=h^i(\Pp^{m}_{\bar k},l_\mtrx{b}^*A)
  \] holds for every
  $b\in U(\bar k)$.
\end{proof}

We also note for later use the inequality $c(l_\mtrx{b}^*A) \leq c(A)$ for all $\mtrx{b} \in U(\bar k)$, which is a straightforward consequence of the previous lemma and the definition of complexity. 

\begin{prop}
\label{rm-zero}
Let $A$ be an object of $\Der(\Pp^n_k)$. The complexity $c(A)$ vanishes if and only if~$A=0$. More precisely, let $d_j$ be the dimension of the support of the cohomology sheaf~$\mathcal{H}^j(A)$ and let $r_j$ be the maximum of the generic ranks of the restrictions of $\mathcal{H}^j(A)$ to the irreducible components of maximal dimension of its support. Setting $d=\max_{j\in \Zz} d_j$ and
$r=\sum_{\substack{j\in \Zz \\ d_j=d}} r_j,$
the inequality $c(A)\geq r$ holds.
\end{prop}

\begin{proof}
Set $m=n-d$ and let $\mtrx{a}_m$ be a geometric generic point of $M^{n+1,m+1}_k$ defined over an algebraically closed field $k'$. The cohomology sheaf $\mathcal{H}^j(l_{\uple{a}_m}^*A)=l_{\uple{a}_m}^*\mathcal{H}^j(A)$ vanishes in degrees $j$ such that $d_j<d$, whereas it has finite support $S_j$ for all $j$ such that~$d_j=d$. Therefore, the Betti number $h^i(\Pp^n_{k'}, \mathcal{H}^j(l_{\uple{a}_m}^*A))$ vanishes for $i\neq 0$ and is equal to the sum of the dimensions of the stalks of $\mathcal{H}^j(l_{\uple{a}_m}^*A)$ at points of $S_j$ for $i=0$. By definition, there is a point of $S_j$ at which the stalk has dimension $r_j$, and hence the inequality~$h^i(\Pp^n_{k'}, \mathcal{H}^0(l_{\uple{a}_m}^*A))\geq r_j$ holds for all $j$ such that $d_j=d$.

It follows from the vanishing of cohomology in non-zero degrees that the spectral sequence  
$$
\rmH^i(\Pp^n_{k'}, \mathcal{H}^j(l_{\uple{a}_m}^*A))\implies \rmH^{i+j}(\Pp^n_{k'}, l_{\uple{a}_m}^*A)
$$
degenerates at the first page, hence an isomorphism
$$
\rmH^{j}(\Pp^n_{k'}, l_{\uple{a}_m}^*A)\cong\rmH^0(\Pp^n_{k'}, \mathcal{H}^j(l_{\uple{a}_m}^*A)).
$$
By definition of complexity, we then get 
$$
c(A)\geq \sum_{j\in \Zz} h^{j}(\Pp^n_{k'}, l_{\uple{a}_m}^*A)=\sum_{j\in \Zz} h^{0}(\Pp^n_{k'}, \mathcal{H}^j(l_{\uple{a}_m}^*A))\geq\sum_{\substack{j\in \Zz \\ d_j=d}}  r_j=r,
$$
which concludes the proof, since $r=0$ if and only if $A=0$. 
\end{proof}

We will make several arguments involving generic points on different spaces of linear maps.  The next two lemmas describe all the properties we will need about relationships between different generic points. 
 
\begin{lem}\label{transcendenceswapping} Let $X$ and $Y$ be geometrically irreducible affine varieties over a field~$k$. Let~$x \in X(k')$ and $y \in Y(k')$ be points defined over a field extension~$k'$ of $k$. Let~$k(x)$ and~$k(y)$ be the fields generated by the coordinates of $x$ and $y$ respectively. If $x$ is a geometric generic point of $X_k$ and $y$ is a geometric generic point of $Y_{k(x)}$, then $x$ is a geometric generic point of $X_{k(y)}$. \end{lem}

\begin{proof}
  Letting $\xi$ and $\eta$ denote the generic points of $X$ and $Y$ respectively, the result follows from the observation that if we
  identify $X\times \eta$ and $\xi\times Y$ with (irreducible)
  subschemes of~$X\times Y$, then the generic points of $X\times \eta$ and
  $\xi\times Y$ are the same as the generic point of $X\times Y$. 
\end{proof}

\begin{lem}\label{generic-mixing} Let $\mtrx{a}$ be an element of $M^{n+1,m+1}(k)$ and let $\mtrx{g}$ be a geometric generic point of~$\GL_{n+1,k}$. Then $\mtrx{g} \mtrx{a}$ is a geometric generic point of $M^{n+1,m+1}_k$. \end{lem}

\begin{proof} 
Because $\GL_{n+1, k}$ acts transitively on $M^{n+1,m+1}_k$, the map $g \mapsto g \mtrx{a}$ from $\GL_{n+1,k}$ to~$M^{n+1,m+1}_k$ is surjective, and hence dominant, so the image of $\mtrx{g}$ under this map is a geometric generic point of $M^{n+1,m+1}_k$.
\end{proof}

\begin{lem}\label{genlem2} Let $k$ be a field and let $0 \leq m \leq n$ be integers. Let $\mtrx{a}$, $\mtrx{g}$, and $\mtrx{h}$ be points of~$M^{n+1,m+1}_k$, $\GL_{n+1, k}$, and $\GL_{m+1, k}$ respectively, all defined over an algebraically closed field $k'$. Let $\mtrx{b} = \mtrx{g} \mtrx{a} \mtrx{h}^{-1}$. Let $k(\mtrx{a})$ be the subfield of $k'$ generated by the coordinates of $\mtrx{a}$, and similarly for $k(\mtrx{g}), k(\mtrx{a},\mtrx{g}),$
  and so on.  Assume that ${\mtrx{g}}$ is a geometric generic point
  of~$\GL_{n+1,k(\mtrx{a})}$ and that $\mtrx{h}$ is a geometric generic
  point of $\GL_{m+1,k(\mtrx{a},\mtrx{g})}$. Then
  $\mtrx{b}$ is a geometric generic point of
  $M^{n+1,m+1}_{k(\mtrx{a},\mtrx{h})}$ and $\mtrx{h}$ is a geometric
  generic point of~$\GL_{m+1,k(\mtrx{a},\mtrx{b})}$.
 
\end{lem}

\begin{proof} By Lemma \ref{transcendenceswapping}, because $\mtrx{g}$ is generic over $k(\mtrx{a})$ and $\mtrx{h}$ is generic over $k(\mtrx{a},\mtrx{g})$, $\mtrx{g}$ is a geometric generic point of $\GL_{n+1,k(\mtrx{a},\mtrx{h})}$.  

Because $\GL_{n+1, k}$ acts transitively on $M^{n+1,m+1}_k$ by left multiplication, applying a geometric generic point $\mtrx{g}$ of $\GL_{n+1,k(\mtrx{a},\mtrx{h})}$ to a point $ \mtrx{a} \mtrx{h}^{-1}$ of $M^{n+1,m+1}_{k(\mtrx{a},\mtrx{h})}$ produces a geometric generic point  $\mtrx{b}= \mtrx{g} \mtrx{a}\mtrx{h}^{-1}$ of $M^{n+1,m+1}_{k(\mtrx{a},\mtrx{h})}$. 

By Lemma \ref{transcendenceswapping}, because $\mtrx{h}$ is geometric
generic over $k(\mtrx{a})$ and $\mtrx{b}$ is geometric generic over
$k(\mtrx{a},\mtrx{h})$, it follows that $\mtrx{h}$ is geometric generic
over $k(\mtrx{a},\mtrx{b})$.
\end{proof}

\begin{defi} For $\lambda_1,\dots, \lambda_{n+1}$ in $k^{\times}$, let
  $\Diag(\lambda_1, \dots,\lambda_{n+1})$ be the diagonal matrix
  in~$\GL_{n+1}(k)$ whose diagonal entries are
  $\lambda_1,\dots,\lambda_{n+1}$. \end{defi}

\begin{lem}\label{final generic} Let $k$ be a field and let $n \geq 0$ be an integer. Let $\mtrx{g}$ be a geometric generic point of $  \GL_{n+1,k}$, and let $\lambda_1,\dots,\lambda_{n+1}$  be independent transcendentals over the field $k(\mtrx{g})$ generated by the matrix coefficients of $\mtrx{g}$. Then: 
\begin{enumerate}
 
\item $\mtrx{g} \Diag(\lambda_1,\dots,\lambda_{n+1}) \mtrx{g}^{-1}$ is a geometric generic point of $\GL_{n+1,k}$.

\item $\mtrx{g} \Diag(1,\dots,1,\lambda_{m+1},\dots,\lambda_{n+1})$ is a geometric generic point of $\GL_{n+1,k}$ for all $0 \leq m \leq n+1$. 
\end{enumerate}
\end{lem}

\begin{proof} Because $(\mtrx{g}, \lambda_1,\dots,\lambda_{n+1})$ is a  geometric generic point of $\GL_{n+1, k} \times \Gg_{m, k}^{n+1}$, to check that an element of $\GL_{n+1,k(\mtrx{g},\lambda_1,\dots,\lambda_{n+1})}$, expressed as a function of the tuple $(\mtrx{g}, \lambda_1,\dots,\lambda_{n+1})$ is a geometric generic point of $\GL_{n+1,k}$, it suffices to check that the function in question is a dominant map from $\GL_{n+1,k} \times \Gg_{m, k}^{n+1}$ to  $\GL_{n+1,k}$. For $\mtrx{g} \Diag(1,\dots,1,\lambda_{m+1},\dots,\lambda_{n+1})$, dominance is simply the fact that $\GL_{n+1, k}$ acts transitively on itself by left multiplication and for~$\mtrx{g} \Diag(\lambda_1,\dots,\lambda_{n+1}) \mtrx{g}^{-1}$ dominance is the well-known fact that generic matrices are diagonalizable.
\end{proof}

\begin{lem}\label{generic-intersection} Let $k$ be a field and let $m$
  and $m'$ be integers such that $0 \leq m,m'\leq n$ and $0 \leq m+m'-n$. Let $\mtrx{a}$
  and $\mtrx{a}'$ be independent geometric generic points of
  $M^{n+1,m+1}_k$ and $M^{n+1,m'+1}_k$ defined over a field $k'$ (i.e. such that $\mtrx{a}'$ is a geometric generic of $M^{n+1,m'+1}_{k(\mtrx{a})}$). Then there exists a geometric generic
  point $\mtrx{b}$ of $M^{n+1, m+m'-n+1}_k$ such that the image of
  $l_{\mtrx{b}}$ is the intersection of the images of $l_{\mtrx{a}}$
  and $l_{\mtrx{a}'}$ and such that
  $l_{\mtrx{a}*} \Ql \otimes l_{\mtrx{a}'*} \Ql = l_{\mtrx{b}*}
  \Ql$. \end{lem}

\begin{proof} View $\mtrx{a}$ and $\mtrx{a}'$ as matrices of elements defined over $k'$, and $l_\mtrx{a}(x) - l_{\mtrx{a}'}(y)$ as a linear map from $k'^{ m+1 + m'+1}$ to $k'^{n+1}$. Because $\mtrx{a}$ and $\mtrx{a}'$ are generic  and $m+ m' -n \geq 0$, this map is surjective, with kernel of dimension $m+m' +1-n$. We can choose a basis for this kernel by row reduction, and then let the columns of $\mtrx{b}$ be $\mtrx{a}$ (equivalently $\mtrx{a}')$ applied to the vectors in this basis.

We can check that the point $\mtrx{b}$ is generic by restricting to the special case where the first $m +m'+1-n$ columns of $\mtrx{a}$ agree with the first $m+m'+1-n$ columns of~$\mtrx{a}'$, and the other columns are generic. In this case, $\mtrx{b}$ will consist of exactly these columns. Since some specialization of generic matrices maps to a generic matrix, generic matrices also map to a generic matrix.

The last claim, on pushforwards, follows from the definition and the calculation of the intersection of the image. \end{proof}

\begin{lem}\label{genericperversity} Let $A$ and $B$ be perverse sheaves on $\Pp^n_k$. For a geometric generic point $\mtrx{g} \in \GL_{n+1}(k')$, the object $A \otimes l_{\mtrx{g}}^* B[-n]$ is perverse. \end{lem}

\begin{proof} The external product $A \boxtimes B$ is perverse, and the sheaf $A \otimes l_{\mtrx{g}}^* B$ on $\Pp^n \times \GL_{n+1}$ is the pullback of $A \boxtimes B$ by a smooth morphism of relative dimension $(n+1)^2-n$, so by \cite[4.2.5]{BBDG} is perverse after shifting by $(n+1)^2-n$. We then take the pullback to the fiber over a geometric generic point of $\GL_{n+1}$, which preserves perversity after shifting by $-(n+1)^2$ by definition of perversity. \end{proof}

\begin{lem}\label{generic-dual} Let $\mtrx{a}$ be a geometric generic point of $M^{n+1,m+1}_k$. Let $A$ be an object of~$\Der(\Pp^n_k)$. 
  Then $\dual( l_{\mtrx{a}}^* A) = l_{\mtrx{a}}^* \dual(A)[2(m-n)]$,
  and if $A$ is perverse, then $l_{\mtrx{a}}^* A [m-n]$ is perverse.
\end{lem}

\begin{proof} We view the linear map $l_{\mtrx a}$ as the composition of the map $x \mapsto (x,\mtrx{a})$ from~$\Pp^m$ to
  \hbox{$\Pp^m \times M^{m+1,n+1}$} with the matrix multiplication map
  $\Pp^m \times M^{m+1,n+1} \to \Pp^n$. Observe that matrix multiplication 
  is a smooth morphism of relative dimension \hbox{$(m+1)(n+1) +m-n$} and the pullback
  to the geometric generic point is an inverse limit of open
  immersions. As in the previous lemma, the former preserve perversity
  after shifting by $(m+1)(n+1) +m-n$, and the latter after shifting by
  $-(m+1)(n+1)$. The statements for duals follow from the equality
  $f^!=f^*(d)[2d]$ for a smooth morphism $f$ of relative dimension
  $d$.
\end{proof}

\begin{lem}\label{lem-indepl}
  Let $0 \leq m \leq n-1$ be an integer, and let $A$ be a perverse sheaf on~$\Pp^n_k$. There exists a dense open subset
  $U\subset M^{n+1,m+1}_k$ such that, for any geometric generic point
  $\mtrx{a}$ of $M^{n+1,m+1}_k$ defined over an algebraically closed field $k'$ and for any $\mtrx{b}\in U(\bar{k})$, the complex
  $l_\mtrx{b}^*A[m-n]$ is perverse and satisfies
  \begin{displaymath}
    h^i(\Pp_{k'}^{m},l_\mtrx{a}^*A)=h^i(\Pp^{m}_{\bar{k}},l_\mtrx{b}^*A)
  \end{displaymath}
  for every $i\in \Zz$.
\end{lem}

\begin{proof}
  Similarly to the proof of Lemma \ref{generic-dual}, we view the linear
  map $l_\mtrx{b}$ as the composition of the
  map~$x \mapsto (x,\mtrx{b})$ from $\Pp^m$ to
  $\Pp^m \times M^{m+1,n+1}$ with the matrix multiplication
  map~\hbox{$\mathrm{mult} \colon \Pp^m \times M^{m+1,n+1} \to
    \Pp^n$}. Since $\mathrm{mult}$ is smooth, the pullback of $A$ along it is
  perverse up to shift by $(m+1)(n+1)- (n-m)$, so it remains to show
  that, for any perverse $\Ql$-sheaf
  $\mathrm{mult}^* A [ (m+1)(n+1) - (n-m)]$ on
  $\Pp^m \times M^{m+1,n+1}$, for $\mtrx{b}$ in some dense open subset
  $U$, the pullback to the fiber over $\mtrx{b}$ is perverse up to shift
  by $(m+1)(n+1)$.
  
  To do this, we choose a $\overline{\mathbf Z}_{\ell}$-structure on
  $\mathrm{mult}^* A [ (m+1)(n+1) - (n-m)]$, obtaining a perverse
  $\overline{\mathbf Z}_{\ell}$-sheaf $K$, and we let $U$ be an open set
  over which $K\otimes\bar{\mathbf{F}}_{\ell}$ is locally acyclic. For
  any point $\mtrx{b}\in U(\bar{k})$, after passing to a further open subset we
  may assume the closure of $\mtrx{b}$ is smooth. Then, by
  \cite[Cor.\,8.10]{saito1}, the immersion $x \mapsto (x,\mtrx{b})$ is
  $K$-transversal and thus the pullback along this immersion is a
  shifted perverse sheaf, as desired.
  
  The second property is granted by Lemma \ref{lem-complexity-open-set}.
\end{proof}

We need the following corollary of Theorem \ref{geneufor}.

\begin{cor}\label{projectiveeulerformula} Let $A$ and $B$ be 
  objects of $\Der(\Pp^{n}_k)$. For each geometric generic point $\mtrx{g}$ of $\GL_{n+1,k}$ over an algebraically closed field $k'$, the following equality holds: 
  \begin{displaymath} 
  \chi( \Pp^{n}_{k'}, A \otimes l_{\mtrx{g}}^* B) =(-1)^n
    \overline{\CC(A)} \cdot \overline{\CC(B)}.
  \end{displaymath}
\end{cor}

\begin{proof} We will prove the statement for objects $A$ and $B$ of $\Der(\Pp^{n}_k,\overline{\mathbf Z}_\ell)$; the result then follows for $A$ and $B$ in $\Der(\Pp^{n}_k,\Ql)$ by choosing an integral structure and noting that both the Euler characteristic and the characteristic cycle are preserved by inverting~$\ell$.

We will check that the singular supports of $A$ and
  $l_{\mtrx{g}}^* B$ fulfil the conditions of Theorem~\ref{geneufor}. From \eqref{char:item5} of \loccit, we will then get the equality
  \[
    \chi( \Pp^{n}_{k'}, A \otimes l_{\mtrx{g}}^* B) = (-1)^n \overline{\CC(A)} \cdot
    \overline{\CC(l_{\mtrx{g}}^* B)} = (-1)^n \overline{\CC(A)} \cdot
    \overline{\CC(B)}, 
  \]
  where the second identity follows from deformation-invariance of intersection numbers, on noting that the algebraic cycles $\overline{\CC(B)}$ and $ \overline{\CC(l_{\mtrx{g}}^* B)} = l_{\mtrx{g}}^*\overline{\CC(B)}$ lie in a family parameterized by $\GL_{n+1, k}$. 

The first condition is that $\SS(A) \cap \SS( l_{\mtrx{g}}^* B)$ is
contained in the zero section. In view of Lemma \ref{simpinteq}, this amounts to showing that $\overline{\SS(A)} \cap
\overline{\SS(l_{\mtrx{g}}^* B)}$ does not intersect
$\Pp(T^*\Pp^n)$ inside $ \overline{ T^*\Pp^n}$. Note the equality
$\overline{\SS(l_{\mtrx{g}}^* B)} =l_{\mtrx{g}}^* \overline{\SS(
  B)}$. The space of triples $(x, y, g)$ such that
  \[
  x \in  \overline{\SS(A)} \cap \Pp(T^*\Pp^n), \quad y \in  \overline{\SS(B)} \cap \Pp(T^*\Pp^n), \quad g \in \GL_{n+1} \quad \text{and}\quad l_g(y) =x
\] is a bundle over
$\overline{\SS(A)} \cap \Pp(T^*\Pp^n) \times \overline{\SS(B)} \cap
\Pp(T^*\Pp^n)$. Since $\GL_{n+1}$ acts transitively on~$\Pp(T^*\Pp^n)$,
all its fibers have dimension $(n+1)^2 - (2n-1)$, and hence the space
itself has dimension at most\begin{multline*} \dim \left(
    \overline{\SS(A)} \cap \Pp(T^*\Pp^n) \right) + \dim\left(
    \overline{\SS(B)} \cap \Pp(T^*\Pp^n) \right) + (n+1)^2 - (2n-1)
  \\
  = (n-1)+(n-1)+(n+1)^2 - (2n-1) =(n+1)^2 -1.
\end{multline*}
This dimension being less than $(n+1)^2 = \dim \GL_{n+1}$, the fiber
over generic $\mtrx{g}$ is empty, which means that
$\overline{\SS(A)} \cap \overline{\SS(l_{\mtrx{g}}^* B)}$ does not
intersect $\Pp(T^*\Pp^n)$.

The second condition is that every irreducible component of
\[
\SS(A) \times_{\Pp^n} \SS( l_{\mtrx{g}}^* B) =\SS(A)
\times_{\Pp^n}l_{\mtrx{g}}^* \SS( B)
\] has dimension at most
$n$. It is sufficient to prove that every irreducible component of the
scheme over $\GL_{n+1}$ whose fiber over a point $\mtrx{h}$ is
$\SS(A) \times_{\Pp^n}l_{\mtrx{h}}^* \SS( B) $ has dimension at
most $(n+1)^2+n$, because then the generic fiber has dimension at most~$n$. This scheme maps to $\SS(A) \times \SS(B)$, which has dimension
$2n$, and the fiber over any point consists of all elements of
$\GL_{n+1}$ that send one point in $\Pp^n$ to another, so the
fiber has dimension~$(n+1)^2-n$ and the total space has dimension
$(n+1)^2+n$, as desired.
\end{proof}

\begin{lem}\label{pullbacklinearity}
  There exists a unique linear map $f_n\colon \Zz^{n+1} \to \Zz^n$
  such that, for any perverse sheaf~$A$ on $\Pp^n_k$ and any geometric
  generic point $\mtrx{b}$ of $M^{n+1,n}_k$, the equality  
  \[
    \cc(l_{\mtrx{b}}^*A)= f_n\left( \cc(A)\right)
  \]
  holds. This linear map does not depend on~$k$ and~$\ell$.
\end{lem}

\begin{proof}Let $K_m$ be the perverse sheaf on $\Pp^{n-1}$ described in Lemma \ref{constantsheafbasis}. 

It does not matter what geometric generic point $\mtrx{b}$ of $M^{n+1,n}_k$ we take, because any pair of geometric generic points defined over two different fields are isomorphic after a suitable extension of both fields. Let $\mtrx{a}$ be a geometric generic point of $M^{n+1,n}_{k}$, $\mtrx{g}$ a geometric generic point of $\GL_{n+1,k(\mtrx{a})}$, and $\mtrx{h}$ a geometric generic point of $\GL_{n,k(\mtrx{a},\mtrx{g})}$. Let $\mtrx{b}= \mtrx{g}\mtrx{a}\mtrx{h}^{-1}$. Then Lemma \ref{genlem2} guarantees that $\mtrx{b}$ is a geometric generic point of $M^{n+1,n}_k$ and so we may work with~$\mtrx{b}$.

Furthermore, $\mtrx{h}$ is a geometric generic
point of $\GL_{n,k(\mtrx{b})}$ by Lemma \ref{genlem2}, and hence Corollary \ref{projectiveeulerformula} yields the equality
\[
  \cc(l_{\mtrx{b}}^* A) \cdot \cc(K_m)= \chi\left(\Pp^{n-1}_{k'},
    l_{\mtrx{b}}^*A \otimes l_{\mtrx{h}^{-1}}^* K_m\right). 
\]
Furthermore, since
$\mtrx{h}$ is generic, the right-hand side is equal to
\begin{align*}
  \chi\left(\Pp^{n-1}_{k'}, l_{\mtrx{h}}^* l_{\mtrx{b}}^* A \otimes
  K_m\right)
  &
    = \chi\left(\Pp^{n-1}_{k'}, l_{\mtrx{a}}^* l_{\mtrx{g}}^* A
    \otimes K_m \right)
  \\
  &=\chi\left(\Pp^{n}_{k'}, l_{\mtrx{g}}^* A \otimes l_{\mtrx{a}*} K_m
    \right) 
    \\
    &= \cc(A) \cdot \cc(l_{\mtrx{a}*} K_m )
\end{align*}
by the projection formula and Corollary \ref{projectiveeulerformula}, along with the fact that $\mtrx{g}$ is a geometric
generic point of~$\GL_{n+1,k(\mtrx{a})}$. Thus,
$\cc(l_{\mtrx{b}}^* A) \cdot \cc(K_m)$ is a linear function of~$\cc(A)$.
  
Because the $\cc(K_m)$ form a basis, by Lemma
\ref{constantsheafbasis}, the pairings
$\cc(l_{\mtrx{b}}^* A) \cdot \cc(K_m)$ of $\cc(l_{\mtrx{b}}^* A)$ with
that basis can be used as coordinates of $\cc(l_{\mtrx{b}}^*
A)$. Because these pairings are linear functions of $\cc(A)$, the
class $\cc(l_{\mtrx{b}}^* A)$ is a linear function of $A$.
  
By Lemma \ref{constantsheafbasis}, the characteristic classes $\cc(K_m)$ and
$\cc(l_{\mtrx{a}*} K_m)$ (which is just $\cc(K_m)$ in a different
projective space) are independent of $\ell$ and $k$, and so the linear
map is independent of $\ell$ and $k$.
\end{proof}

\begin{defi}
  For~$n\geq 0$, let $f_n\colon \Zz^{n+1}\to \Zz^n$ be the linear map
  uniquely determined in Lemma~\ref{pullbacklinearity}. We define inductively a bilinear form
  $$
  b_n \colon \Zz^{n+1}\times \Zz^{n+1}\to \Zz
  $$ by setting~$b_0(x,y)=xy$ for $(x,y)\in\Zz\times \Zz$ and
  $$
  b_n(x,y)=x\cdot y+4b_{n-1}(f_n(x),f_n(y))
  $$
  for~$n\geq 1$ and $(x, y) \in \Zz^{n+1}\times \Zz^{n+1}$, where~$x\cdot y$ is the scalar product
  on~$\Zz^{n+1}$.
\end{defi}

\begin{prop}\label{bilinearbettibound} 
  For any perverse sheaves $A$ and $B$ in $\Der(\Pp^n_k)$, and for any
  geometric generic point $\uple{g}$ of $\GL_{n+1,k}$ defined over $k'$, the following holds: 
  \[ \sum_{i \in \Zz} h^i (\Pp^n_{k'}, A \otimes l_{\mtrx{g}}^* B)
    \leq b_n\left( \cc(A) , \cc(B) \right).\]
\end{prop}

\begin{proof} We prove this proposition by induction on $n$.  The case
  $n=0$ is trivial, as then $A$ and $B$ are simply vector spaces,
  their characteristic classes are their dimensions, and the sum of
  Betti numbers is the product of the dimensions.

  Assume the inequality holds in dimension $n-1$ and let~$A$ and~$B$
  be perverse sheaves in~$\Der(\Pp^n)$.  By Lemma
  \ref{genericperversity}, the object
  $ A \otimes l_{\mtrx{g}}^* B[-n]$ is perverse. Set
  $$
  \sigma= \sum_{i \in \Zz} h^i (\Pp^n_{k'}, A \otimes l_{\mtrx{g}}^* B),
  $$
which is the quantity we want to estimate. By the definition of the Euler--Poincar\'e characteristic, we have
  \begin{align*} h^{-n}(\Pp^n_{k'}, A \otimes l_{\mtrx{g}}^* B)
    &=(-1)^n\chi(\Pp^n_{k'}, A \otimes l_{\mtrx{g}}^* B) + \sum_{i
      \neq
      -n} (-1)^{i+n+1} h^i (\Pp^n_{k'}, A \otimes l_{\mtrx{g}}^* B) \\
    & \leq (-1)^n\chi(\Pp^n_{k'}, A \otimes l_{\mtrx{g}}^* B) +
      \sum_{i \neq -n}h^i (\Pp^n_{k'}, A \otimes l_{\mtrx{g}}^* B), 
  \end{align*}
  hence the inequality 
  \begin{align*}
    \sigma &=\sum_{i <-n} h^i (\Pp^n_{k'}, A \otimes l_{\mtrx{g}}^* B) +
             h^{-n} (\Pp^n_{k'}, A \otimes l_{\mtrx{g}}^* B) + \sum_{i>
             -n} h^i (\Pp^n_{k'}, A \otimes l_{\mtrx{g}}^* B)
    \\
           &\leq  2 \sum_{i <-n} h^i(\Pp^{n}_{k'},  A \otimes l_{\mtrx{g}}^* B) +  (-1)^n \chi(\Pp^n_{k'}, A \otimes l_{\mtrx{g}}^* B)  + 2\sum_{i >-n} h^i(\Pp^{n-1}_{k'}, A \otimes l_{\mtrx{g}}^* B).
  \end{align*}

  Let now $\mtrx{a} \in M^{n+1,n}(k')$ be a geometric generic point of
  $M^{n+1,n}_{k(\mtrx{g})}$ defined over an algebraically closed field
  $k'$. By excision and the $\mathrm{t}$-exactness of affine morphisms,
  the canonical map
  $$
  \rmH^i(\Pp^n_{k'}, A \otimes l_{\mtrx{g}}^* B) \to \rmH^i(\Pp^{n-1}_{k'},
  l_{\mtrx{a}}^* (A \otimes l_{\mtrx{g}}^* B))
  $$
  is an isomorphism in degrees less than $-n$.  Because duality
  exchanges $l^*_{\mtrx{a}}$ and $l^!_{\mtrx{a}}$, and because the
  dual of $A \otimes l_{\mtrx{g}}^*B[-n]$ is also perverse, the
  canonical map
  $$
  \rmH^i(\Pp^{n-1}_{k'}, l_{\mtrx{a}}^! (A \otimes l_{\mtrx{g}}^* B))
  \to \rmH^i(\Pp^n_{k'}, A \otimes l_{\mtrx{g}}^* B)
  $$
  is an isomorphism in degrees greater than $-n$.
  We thus obtain
 $$
    \sigma \leq 2 \sum_{i \in \Zz} h^i(\Pp^{n-1}_{k'}, l_{\mtrx{a}}^*
    (A \otimes l_{\mtrx{g}}^* B)) + (-1)^n \chi(\Pp^n_{k'}, A \otimes
    l_{\mtrx{g}}^* B) + 2\sum_{i \in \Zz} h^i(\Pp^{n-1}_{k'},
    l^!_{\mtrx{a}} (A \otimes l_{\mtrx{g}}^* B)).
 $$


  For a geometric generic point $\mtrx{a}$, the functor
  $l^!_{\mtrx{a}}$ coincides with $l^*_{\mtrx{a}}$, up to a shift and
  Tate twist, and hence
  $$
  \sigma \leq 4 \sum_{i \in \Zz} h^i(\Pp^{n-1}_{k'}, l_{\mtrx{a}}^* (A
  \otimes l_{\mtrx{g}}^* B)) + (-1)^n \chi(\Pp^n_{k'}, A \otimes
  l_{\mtrx{g}}^* B).
  $$


  Now we observe that
  $l_{\mtrx{a}}^* (A \otimes l_{\mtrx{g}}^* B) = l_{\mtrx{a}}^*A
  \otimes l_{\mtrx{a}}^* l_{\mtrx{g}}^* B$. Take $\mtrx{h}$ to be a
  geometric generic point of $\GL_{n, k(\mtrx{a}, \mtrx{g})}$ and
  $\mtrx{b} = \mtrx{g} \mtrx{a} \mtrx{h}^{-1}$ . Then
  $l_{\mtrx{a}}^* l_{\mtrx{g}}^* B = l_{\mtrx{h}}^* l_{\mtrx{b}}^*
  B$. By Lemma \ref{genlem2}, $\mtrx{h}$ is generic over
  $k(\mtrx{a},\mtrx{b})$. Together with the induction hypothesis and
  Lemma~\ref{pullbacklinearity}, we deduce that
  $$
  \sum_{i \in \Zz} h^i(\Pp^{n-1}_{k'}, l_{\mtrx{a}}^* A \otimes
  l_{\mtrx{h}}^* l_{\mtrx{b}}^* B) \leq b_{n-1} ( \cc(l_{\mtrx{a}}^*
  A),\cc(l_{\mtrx{b}}^* B))= b_{n-1} ( f_n(\cc(A)),f_n(\cc(B))).
  $$

  Finally, in view of the equality 
  $$
  (-1)^n \chi(\Pp^n_{k'}, A \otimes l_{\mtrx{g}}^* B)= \cc(A)\cdot
  \cc(B)
  $$
  given by Lemma~\ref{projectiveeulerformula}, we obtain
  $\sigma\leq b_n(\cc(A),\cc(B))$ by definition of~$b_n$.
\end{proof}


\section{Test sheaves}\label{sec-test}

The main result of this section is the construction, achieved in Corollary \ref{testsheafindependence}, of a specific basis of the group $\CH(\Pp^n_k)$, which will play an essential role in controlling the complexity of a tensor product in the next section. 


\begin{defi}[Test sheaf] A \emph{test sheaf} on $\Pp^n_k$ is a
  perverse sheaf $A$ on $\Pp^n_k$ such that, for any field extension
  $k'$ of $k$, any perverse sheaf $B$ on $\Pp^n_{k'}$, and any generic
  point $\mtrx{g}$ of~$\GL_{n+1,k'}$ defined over an algebraically
  closed field extension $k^{''}$ of $k'$, the cohomology group
  $\rmH^i(\Pp^n_{k'}, A\otimes l_{\mtrx{g}}^* B)$ vanishes in all degrees
  $i \neq -n$.
\end{defi}

For example, a skyscraper sheaf~$A$ on $\Pp^n_k$ is a test sheaf, since
the object $A\otimes l_{\mtrx{g}}^* B$ is a skyscraper sheaf sitting in
degree~$-n$ in that case.

\begin{remark}
  The key properties of a test sheaf~$A$ are the following: 
  \begin{enumerate}
  \item the functor defined on perverse sheaves by
  $B\mapsto \rmH^{-n}(\Pp^n_{k'},A\otimes l_{\mtrx{g}}^*B)$ is exact,
 \item the cohomology
  groups $\rmH^i(\Pp^n_{k'},A\otimes l_{\mtrx{g}}^*B)$ and $\rmH^{-n}(\Pp^n_{k'},A\otimes \PH^i(B))$ agree, up to renumbering, for all objects~$B$ of $\Der(\Pp^n_k)$. 
  \end{enumerate}
\end{remark}

Test sheaves are most useful when they are of the form highlighted in the next definition:

\begin{defi}[Strong test sheaf]\label{def-strong}
  Let~$d\geq 1$ be an integer.  A \emph{strong test sheaf of
    depth~$d$} on $\Pp^n_k$ is a test sheaf~$A$ that admits a
  filtration
\[
  0 = F_0 \subset F_1 \dots \subset F_d = A
\]
such that, for each $1 \leq j \leq d$, the quotient $F_j/ F_{j-1}$ is
isomorphic to a perverse sheaf (see Lemma~\ref{generic-dual}) of the form $l_{\mtrx{a}_j * } \Ql[m_j]$ for some
$0\leq m_j \leq n$ and some $\mtrx{a}_j \in M^{n+1,m_j+1}(k)$.
\end{defi}

Strong test sheaves have the following crucial property:

\begin{prop}\label{characteristicclasscontrol}
  Let $A$ be a strong test sheaf of depth~$d$ on~$\Pp^n_k$.  The inequality
  \[
    \sum_{i \in \Zz} \left|\cc(A) \cdot \cc(\PH^i(B))\right| \leq d
    c(B)
  \]
 holds for any field extension $k'$ of $k$ and any object $B$ of
  $\Der(\Pp^n_{k'})$.
\end{prop}

\begin{proof}
  Denote by~$(F_j)_{1\leq j\leq d}$ the terms of the filtration from
  Definition~\ref{def-strong}.
  \par
  Let $\mtrx{g}$ be a generic point of $\GL_{n+1,k'}$. By Corollary
  \ref{projectiveeulerformula}, we have:
\begin{align*}
  \sum_{i \in \Zz} \left|\cc(A) \cdot \cc(\PH^i(B))\right| &= \sum_{i
    \in \Zz} | \chi(\Pp^n_{k'}, A \otimes l_{\mtrx{g}}^* \PH^i(B))| \\
  &\leq \sum_{i \in \Zz} \sum_{j \in \Zz} h^j ( \Pp^n_{k'}, A \otimes
  l_{\mtrx{g}}^* \PH^i(B)).
\end{align*}
\par
Since $A$ is a test sheaf, the vanishing
$\rmH^j ( \Pp^n_{k'}, A \otimes l_{\mtrx{g}}^* \PH^i(B))=0$ holds for all
$j\neq -n$, and hence the spectral sequence computing
$\rmH^i (\Pp^n_{k'}, A \otimes l_{\mtrx{g}}^* B)$ via the perverse
filtration on $B$ degenerates at the first page. It follows that
\[
  \sum_{i \in \Zz} \sum_{j \in \Zz} h^j ( \Pp^n_{k'}, A \otimes
  l_{\mtrx{g}}^* \PH^i(B)) = \sum_{i \in \Zz} h^i (\Pp^n_{k'}, A
  \otimes l_{\mtrx{g}}^* B).
\]
\par
On the other hand, the spectral sequence associated to the
filtration $(F_j)$ of $A$ yields the inequality of sums of Betti numbers
\[
  \sum_{i \in \Zz} h^i (\Pp^n_{k'}, A \otimes l_{\mtrx{g}}^* B) \leq
  \sum_{j=1}^d \sum_{i \in \Zz} h^i (\Pp^n_{k'}, (F_j/F_{j-1}) \otimes
  l_{\mtrx{g}}^* B).
\]
By assumption, there exist isomorphisms
$F_j /F_{j-1}= l_{\mtrx{a}_j*} \Ql[m_j]$ for some integer \hbox{$0 \leq m_j \leq n$} and
$k$-point $\mtrx{a}_j$ of $M^{n+1,m_j+1}$, so that the following equalities hold:
\begin{align*}
  \sum_{j=1}^d \sum_{i \in \Zz} h^i (\Pp^n_{k'},(F_j/F_{j-1}) \otimes
  l_{\mtrx{g}}^* B)
  &= \sum_{j=1}^d \sum_{i \in \Zz} h^i (\Pp^n_{k'},
  l_{\mtrx{a}_j *} \Ql \otimes l_{\mtrx{g}}^* B)
  \\
  &= \sum_{j=1}^d \sum_{i \in \Zz} h^i (\Pp^{m_j} _{k'},
  l_{\mtrx{a}_j}^* l_{\mtrx{g}}^* B).
\end{align*}
\par
Since $\mtrx{g}$ is generic over $k'$, the product
$\mtrx{g} \mtrx{a}_j$ is a generic point of $M^{n+1,m+1}_{k'}$, and
therefore  the definition of $c(B)$ implies the inequality
\[
  \sum_{j=1}^d \sum_{i \in \Zz} h^i (\Pp^{m_j} _{k'}, l_{\mtrx{a}_j}^*
  l_{\mtrx{g}}^* B)\leq \sum_{j=1}^d c(B) = d c (B), 
\]
which concludes the proof. 
\end{proof}

We will now construct strong test sheaves forming a basis
of~$\CH(\Pp^n_k$).

\begin{lem}\label{shriektensor}
  Let $A$ and $B$ be objects of $\Der(\Pp^n_k)$. Let $\mtrx{g}$ be a
  generic point of $\GL_{n,k}$. Then there is a canonical
  isomorphism
  \[
A \otimes l_{\mtrx{g}}^* B \cong (A \otimes_! l_{\mtrx{g}}^* B)[-2n].
\]
\end{lem}

\begin{proof} Let
  $l\colon \Pp^n \times \GL_{n+1} \to \Pp^n \times \GL_{n+1}$ be the
  universal morphism given by
  $$
  l(x,h) = (h\cdot x,h)
  $$
  and let $\pi\colon \Pp^n \times \GL_{n+1} \to \Pp^n$ be the
  projection. Then $A \otimes l_{\mtrx{g}}^* B$ and
  $A \otimes_! l_{\mtrx{g}}^* B$ are the generic fibers of
  $\pi^* A \otimes l^* \pi^* B$ and $(\pi^*A \otimes_! l^* \pi^* B)[-2(n+1)^2]$
  respectively, so it suffices to prove that the objects 
  $\pi^* A \otimes l^* \pi^* B$ and $\pi^* A \otimes_! l^* \pi^* B$
  are isomorphic up to shift. 
  
  For this, consider the morphism
  $(\pi, \pi \circ l)\colon \Pp^n \times \GL_{n+1} \to \Pp^n \times
  \Pp^n$, which is smooth because $\GL_{n+1}$ acts transitively on $\Pp^n$. Since $\otimes$ commutes with smooth pullbacks and~$\otimes_!$ commutes with smooth pullbacks up to shifting by twice the relative dimension, we obtain isomorphisms
\begin{align*}
  \pi^* A \otimes l^* \pi^* B &=(\pi, \pi \circ l)^*
  \left(\mathrm{pr}_1^* A \otimes \mathrm{pr}_2^* B\right),
  \\
  \pi^* A \otimes_! l^* \pi^* B &=(\pi, \pi \circ l)^*
  \left(\mathrm{pr}_1^* A \otimes_! \mathrm{pr}_2^* B\right)[2(n+1)^2-2n]. 
\end{align*}
It is then enough to prove that
$\mathrm{pr}_1^* A \otimes \mathrm{pr}_2^* B$ and
$\mathrm{pr}_1^* A \otimes_! \mathrm{pr}_2^* B$ are isomorphic up to shift. This follows from the computation
\begin{align*}
\mathrm{pr}_1^* A \otimes_! \mathrm{pr}_2^* B&=\dual(\dual(\mathrm{pr}_1^* A ) \otimes \dual(\mathrm{pr}_2^* B)) \\
&=\dual(\mathrm{pr}_1^*\dual(A)[2n] \otimes \mathrm{pr}_2^*\dual(B)[2n]) \\
&=\mathrm{pr}_1^* \dual(\dual(A)[2n]) \otimes \mathrm{pr}_2^*\dual(\dual(B)[2n]) \\ &=(\mathrm{pr}_1^* A \otimes \mathrm{pr}_2^* B)[-4n], 
\end{align*} which uses the standard properties of Verdier duality and $\mathrm{pr}_i^!=\mathrm{pr}_i^\ast[2n]$. 
\end{proof}

\begin{lem}\label{exampleoftestsheaf}
  Let $n \geq 0$ be an integer. Let $H_1$ and $H_2$ be hyperplanes in
  $\Pp^n_k$ that intersect transversely, and consider the commutative
  diagram
 \[
  \begin{tikzcd}
    \Pp^n_k-H_1 \arrow{r}{a}
    & \Pp^n_k\\
    \arrow{u}{b}\Pp^n_k-(H_1\cup H_2) \arrow{r}{d} &\Pp^n_k-H_2. 
    \arrow{u}{c}
  \end{tikzcd}
 \]
  The object $a_* b_! \Ql[n]$ is isomorphic to $c_! d_* \Ql[n]$ and is a
  test sheaf.
\end{lem}

\begin{proof} By adjunction, there is a natural morphism
  $c_! d_* \Ql \to a_* b_! \Ql$ extending the identity on
  $\Pp^n_k-(H_1\cup H_2)$, and it suffices to prove that this map is
  an isomorphism on stalks at all geometric points of $H_1\cup
  H_2$. This is obvious save for points of the
  intersection~$H_1 \cup H_2$. For those, we argue as follows: since
  $H_1$ and $H_2$ intersect transversely, étale locally around each
  such point the diagram looks like
\[
  \begin{tikzcd}
    \Aa^{n-2}_k \times (\Aa^1_k - \{0\}) \times \Aa^1_k \arrow{r}{a}
    & \Aa^n_k\\
    \arrow{u}{b}\Aa^{n-2}_k \times (\Aa^1_k - \{0\}) \times (\Aa^1_k - \{0\}) \arrow{r}{d} &  \Aa^{n-2}_k \times \Aa^1 \times (\Aa^1_k - \{0\}).
    \arrow{u}{c}
  \end{tikzcd}
 \] Expressing the constant sheaf on the down left corner as an external product and letting~$j \colon \Aa^1-\{0\} \hookrightarrow \Aa^1$ denote the inclusion, the K\"unneth formula then implies that the map~$c_! d_* \Ql \to a_* b_! \Ql$ is locally given by the identity on $\Ql \boxtimes j_\ast \Ql \boxtimes j_!\Ql$. 

  We now proceed to the proof that
  $a_* b_! \Ql[n]\simeq c_! d_* \Ql[n]$ is a test sheaf. First of all,
  we check that $a_*b_!\Ql[n]$ is perverse. Indeed, all maps in the diagram are affine open immersions, and the direct image and exceptional direct image by those preserve perversity. Next, let~$B$ be a perverse sheaf on $\Pp^n_{k}$ and
  let~$i\in\Zz$. Using the projection formula, we get the equalities 
\begin{align*}
  \rmH^i ( \Pp^n_{k'}, c_! d_* \Ql[n] \otimes l_{\mtrx{g}}^* B) &
  = \rmH^{i+n} ( \Pp^n_{k'}, c_! (d_* \Ql \otimes c^*l_{\mtrx{g}}^* B)) \\
  &
  =\rmH^{i+n}_c (\Pp^n_{k'} - H_2, d_* \Ql \otimes c^* l_{\mtrx{g}}^* B) \\
  & = \rmH^{i+n}_c (\Pp^n_{k'} - H_2, c^* ( c_! d_* \Ql \otimes
  l_{\mtrx{g}}^* B) ).
\end{align*}
\par
The object $c_!d_* \Ql \otimes l_{\mtrx{g}}^* B$ is perverse by Lemma
\ref{genericperversity}. Since pullbacks by open immersions preserve
perversity, the object
$c^* ( c_! d_* \Ql \otimes c^* l_{\mtrx{g}}^* B)$ is a perverse sheaf
on the affine variety~$\Pp^n_{k'} - H_2$, and hence its compactly
supported cohomology vanishes when $i+n<0$ by Artin's vanishing
theorem.
\par
Dually, using Lemma \ref{shriektensor}, we get
\begin{align*}
  \rmH^i ( \Pp^n_{k'}, a_* b_! \Ql[n] \otimes l_{\mtrx{g}}^* B) &=
  \rmH^{i-n} ( \Pp^n_{k'}, a_* b_! \Ql \otimes_! l_{\mtrx{g}}^* B) \\
  &= \rmH^{i-n} ( \Pp^n_{k'}, a_* (b_! \Ql \otimes_! a^*
  l_{\mtrx{g}}^* B)) \\
  &= \rmH^{i-n} (\Pp^n_{k'} - H_1, b_!
  \Ql \otimes_! a^* l_{\mtrx{g}}^* B)\\
  &= \rmH^{i+n} (\Pp^n_{k'}- H_1, a^* (a_* b_! \Ql \otimes_!
  l_{\mtrx{g}}^* B)[-2n]).
\end{align*}
\par
Arguing as above using Lemma \ref{shriektensor}, we see that
$a^* (a_* b_! \Ql \otimes_!  l_{\mtrx{g}}^* B)[-2n]$ is a perverse sheaf on
the affine variety $\Pp^n_{k'}-H_1$, so that its cohomology vanishes
for $i+n>0$ by Artin's vanishing theorem. This completes the proof.
\end{proof}

\begin{lem}\label{testsheaftransferral} 
  Let $\mtrx{a} \in M^{n+1,m+1}_k(k)$. If $A$ is a test sheaf on
  $\Pp^m_k$, then $l_{\mtrx{a}*} A$ is a test sheaf on $\Pp^n_k$.
\end{lem}

\begin{proof} 
Let $B$ be a perverse sheaf on $\Pp^n_{k'}$ and let $\mtrx{g}$ be a
generic point in $\GL_{n+1,k'}$. By the projection formula, we have
\[
  \rmH^i(\Pp^n_{k'}, l_{\mtrx{a}*} A \otimes l_{\mtrx{g}}^* B) =
  \rmH^{i} (\Pp^m_{k'}, A \otimes l_{\mtrx{a}}^* l_{\mtrx{g}}^* B).
\]
\par
Let $\mtrx{h}$ be a generic point of $\GL_{m+1,k'(\mtrx{a},\mtrx{g})}$. Set
$\mtrx{b} =\mtrx{g} \mtrx{a} \mtrx{h}^{-1}$ and
$k'' = k'(\mtrx{a}, \mtrx{g}, \mtrx{h})$. By Lemma \ref{genlem2}, the
point $\mtrx{b}$ is a generic linear embedding and $\mtrx{h}$ is
generic over the field of definition of $k'(\mtrx{b})$, so that
\[ 
  \rmH^{i} (\Pp^m_{k''}, A \otimes l_{\mtrx{a}}^* l_{\mtrx{g}}^* B) =
  \rmH^{i} (\Pp^m_{k''}, A \otimes l_{\mtrx{h}}^* l_{\mtrx{b}}^* B)
  =\rmH^{i+n-m} (\Pp^m_{k''}, A \otimes l_{\mtrx{h}}^* l_{\mtrx{b}}^*
  B[m-n]).
\]
Because $\mtrx{b}$ is generic, the object $l_{\mtrx{b}}^* B[m-n]$ is
perverse by Lemma~\ref{generic-dual}, and therefore this cohomology
group vanishes for $i+n-m \neq -m$, i.e., for $i \neq -n$.
\end{proof}

\begin{defi}[Standard test sheaf]\label{def-test-am}
  For each integer $0 \leq m \leq n$, pick a point
  $\mtrx{a}_m \in M^{n+1,m+1}(k)$ and pick transversal
  hyperplanes~$H_1$ and~$H_2$ in $\Pp^m_k$. As in Lemma
  \ref{exampleoftestsheaf}, let $c \colon \Pp^m_k - H_2 \to \Pp^m_k$
  and $d \colon \Pp^n_k - (H_1 \cup H_2) \to \Pp^m_k - H_2 $ be the
  corresponding immersions. Define the $m$-th \emph{standard test
    sheaf}~$A_m$ as 
\[
A_m= l_{\mtrx{a}_m*} c_! d_* \Ql[m]. 
\]
\end{defi}

\begin{remark}
  In the case $m=0$, choosing $\mtrx{a}_0$ amounts to choosing a
  $k$-point of $\Pp^n_k$, the only possible choices for $H_1$ and
  $H_2$ are the empty hyperplanes in $\Pp^0_k=\Spec k$, and $A_0$ is a
  skyscraper sheaf supported at the chosen point.
\end{remark}

\begin{cor}\label{testsheafconstruction}
  For each $0 \leq m \leq n$, the object $A_m$ is a strong test sheaf
  on $\Pp^n_k$ of depth at most~$4$.
\end{cor}

\begin{proof} The fact that $A_m$ is a test sheaf follows from
  applying Lemma \ref{exampleoftestsheaf} and
  Lemma~\ref{testsheaftransferral}. Because $l_{\mtrx{a}*}$ preserves
  constant sheaves on linear subspaces, it preserves any filtration
  into constant sheaves on linear subspaces of the type required for a
  strong test sheaf, and thus it suffices to find such a filtration
  for~$c_!d_* \Ql$.
  
  We use the notation of Lemma~\ref{exampleoftestsheaf}. Let $h$ be the
  closed immersion of $H_2$ into~$\Pp^m_k$. Since~$c$ is an open
  immersion, there is a base change isomorphism $d_* \Ql = c^*a_* \Ql $. By adjunction, there is hence a morphism
  $c_!d_* \Ql = c_! c^*a_* \Ql \to a_* \Ql$, which is an isomorphism away
  from~$H_2$. Because $c_!d_* \Ql$ vanishes on $H_2$, the morphism
  $c_!d_* \Ql \to a_* \Ql$ has mapping cone~$h^* a_* \Ql$. The complex
  $a_* \Ql$ has a filtration whose associated quotients are the constant
  sheaf and a shift of the constant sheaf on $H_1$, so $h^*a_* \Ql$ has a
  filtration whose associated quotients are the constant sheaf on $H_2$ and
  a shift of the constant sheaf on the intersection~$H_1 \cap H_2$. The
  mapping cone triangle gives the desired filtration of~$c_!d_* \Ql$.
\end{proof}

\begin{remark}
  (1) One can also give a proof ``by pure thought'' of this proposition, based on the fact that $A_m$ is equivariant
  for the subgroup of~$\PGL_{m+1}$ acting on~$\Pp^m_k$ and preserving the hyperplanes $H_1$
  and $H_2$. Because this group action has finitely many orbits, and its
  stabilizers are connected, one can show that the only irreducible
  elements of the category of perverse sheaves invariant under this group
  are the intersection cohomology complexes of the closures of the orbits,
  which in this case are simply the constant sheaves on $\Pp^m$, $H_1$,
  $H_2,$ and $H_1 \cap H_2$. An analogous method is used in
  the theory of the geometric Satake isomorphism to classify perverse
  sheaves on the affine Grassmanian that are equivariant for the left action of the
  formal arc group (see, e.g.,~\cite[proof of Prop.\,1]{GaitsgoryNearbyCycles}).
  \par
  (2) The standard test sheaves depend on the choices in
  Definition~\ref{def-test-am}; whenever we use them, we assume
  implicitly these choices have been made for all~$m$, and that they
  are the same in the remainder of the arguments.
\end{remark}

\begin{lem}\label{testsheafpairing} Let $(A_j)_{0\leq j\leq n}$ be
  standard test sheaves. The following holds: 
  \begin{displaymath}
   \cc(A_{m_1}) \cdot \cc(A_{m_2})=\begin{cases} 0 & \text{ if  $m_1+m_2<n$,} \\ 1 & \text{ if $m_1 + m_2 = n$.}
   \end{cases} 
  \end{displaymath}
\end{lem}

\begin{proof}
  This is a straightforward consequence of Corollary~\ref{projectiveeulerformula}, which identifies the intersection number $ \cc(A_{m_1}) \cdot \cc(A_{m_2})$ with $(-1)^n \chi(\Pp^{n}_{k'}, A_{m_1} \otimes l_{\mtrx{g}}^* A_{m_2})$. Indeed, the support
  of~$A_m$ is of dimension~$m$. If $m_1+m_2<n$, it follows that the
  support of~$A_{m_1}$ does not intersect that of~$A_{m_2}$ after
  generic translation, and consequently their tensor product vanishes,
  from which $\cc(A_{m_1})\cdot \cc(A_{m_2})=0$ follows. Similarly, if $m_1+m_2=n$, the support of of~$A_{m_1}$ intersects
  that of~$A_{m_2}$ after generic translation at a single point. At
  this point, both sheaves have stalk~$\Ql$, so their tensor product
  is a skyscraper sheaf and has Euler--Poincar\'e characteristic~$1$.
\end{proof}

\begin{cor}\label{testsheafindependence}
  Let $(A_j)_{0\leq j\leq n}$ be standard test sheaves on~$\Pp^n_k$.
  Then the characteristic classes $(\cc(A_j))_{0\leq j\leq n}$ form a basis of
  $\CH(\Pp^n_k)\simeq \mathbf{Z}^{n+1}$.
\end{cor}

\begin{proof}
  Indeed, by Lemma \ref{testsheafpairing}, the matrix of intersection
  pairings between the characteristic classes~$\cc(A_i)$ is invertible.
\end{proof}


\section{Complexity of a tensor product}\label{sec-tensor}

In this section, we prove the crucial theorem showing that the
complexity controls the sum of Betti numbers of a tensor product of
complexes of $\ell$-adic sheaves, namely: 

\begin{theo}\label{main}
  Let $A$ and $B$ be objects of $\Der(\Pp^n_k)$. Then the estimate
  \[ 
  \sum_{i \in \Zz} h^i(\Pp^n_k, A \otimes B) \ll c(A) c(B)
  \]
  holds, with an implied constant that only depends on $n$.
\end{theo}

\begin{remark}
  In Theorem~\ref{th-main-effective}, we state and sketch a proof of a
  version of this result with an explicit constant.
\end{remark}

Before starting the proof, we show how this implies an important
corollary.

\begin{cor}\label{tensor1}
  Let $A$ and $B$ be objects of $\Der(\Pp^n_k)$. Then the estimate
  \[
    c( A \otimes B) \ll c(A) c(B)
  \]
  holds, with an implied constant that only depends on~$n$. 
\end{cor}

\begin{proof}
  Let~$\uple{g}$ be a geometric generic point of~$\GL_{n+1,k}$.  For
  $0\leq m\leq n$, let $\uple{a}_m$ be a geometric generic point of
  $M^{n+1,m+1}_k$. From the definition of complexity and Theorem~\ref{main}, we
  get the estimate
  \[ c(A \otimes B) = \sup_ {0 \leq m \leq n} \sum_{i \in \Zz} h^i(
    \Pp^n_k, A \otimes B \otimes l_ {\mtrx{a}_m*}\Ql)\]
  \[ \ll c(A)\ \sup_{0 \leq m \leq n} c ( B \otimes l_{\mtrx{a}_m *}
    \Ql).
  \]
  \par
  For $0\leq m'\leq n$, let $\mtrx{a}'_{m'}$ be a geometric generic point of
  $M^{n+1,m'+1}_{k(\mtrx{a}_m)}$.  Again by definition, the complexity above is given by
  \[c ( B \otimes l_{\mtrx{a}_m *} \Ql) = \sup_{0 \leq m' \leq n}
    \sum_{i \in \Zz} h^i(\Pp^n_k, B \otimes l_{\mtrx{a}_m*} \Ql
    \otimes l_{\mtrx{a}'_{m'}*} \Ql)
  \]
  for any~$m\leq n$.
  By Lemma \ref{generic-intersection}, the equality
  $l_{\mtrx{a}_m*} \Ql \otimes l_{\mtrx{a}'_{m'}*}
  \Ql=l_{\mtrx{b}*}\Ql$ holds for a geometric generic point $\mtrx{b}$ of
  $M^{n+1, m+m'-n+1}_k$. This implies the inequality 
  \[
    \sum_{i \in \Zz} h^i(\Pp^n_k, B \otimes l_{\mtrx{a}_m*} \Ql
    \otimes l_{\mtrx{a}'_{m'}*} \Ql) \leq c(B)
  \]
  for any $m$ and~$m'$, from which the result follows.
\end{proof}

The strategy for proving Theorem \ref{main} consists in first establishing a ``generic''
version of it, and then deducing from this the precise statement. In what follows, we denote by $(A_m)_{0\leq m\leq n}$ the family of standard test sheaves from Definition~\ref{def-test-am}.  For
$x \in \Rr^{n+1}= \CH_*(\Pp^n) \otimes \Rr$, we set
$$
\|x\| = \sum_{m=0}^n |x \cdot \cc(A_m)|
$$
Because $(\cc(A_m))_{0\leq m\leq n}$ is a basis of $\CH_*(\Pp^n)$, by
Corollary~\ref{testsheafindependence}, and the intersection form is
non-degenerate, this is simply the $\ell^1$-norm in the dual basis, and
in particular defines a norm on $\Rr^{n+1}$.

\begin{prop}\label{characteristicclassbound} For any object $B$ of
  $\Der(\Pp^n_k)$, the following inequality holds: 
  \[ \sum_{i \in \Zz} \| \cc(\PH^i(B))\| \leq 4 (n+1) c(B).\]
\end{prop}
 
\begin{proof} Let $0 \leq m \leq n$. Because $A_m$ is a strong test sheaf of depth at most $4$, Proposition~\ref{characteristicclasscontrol} yields the inequality 
  \[
\sum_{p \in \Zz} \left|\cc(\PH^p(B)) \cdot \cc(A_m)\right| \leq 4 c(B). 
\]
Thus, we have 
\[
\sum_{p \in \Zz} \|\cc(\PH^p(B))\| = \sum_{m=0}^n \ \sum_{p \in \Zz}
  \left| \cc(\PH^p(B) \cdot \cc(A_m)\right| \leq 4 (n+1)c(B). \qedhere
  \]
\end{proof}

The second part of the next corollary shows that the complexity of a complex of sheaves is, up
to constants, the sum of the norms of the characteristic cycles of its perverse cohomology sheaves. 

\begin{cor}\label{genericbettibound}
  Let~$A$ and~$B$ be objects of~$\Der(\Pp^n_k)$.
  \begin{enumerate}
  \item\label{genericbettibound:item1} For a generic point $\mtrx{g}$ of $\GL_{n+1,k}$ defined over
    $k'$, the estimate
    \[ \sum_{i \in \Zz} h^i(\Pp^n_{k'}, A \otimes l_{\mtrx{g}}^* B ) \ll
      c(A) c(B) \]
   holds, with an implied constant that only depends on $n$.
  \item\label{genericbettibound:item2} We have
    $$
    c(B) \asymp \sum_{i \in \Zz} \| \cc(\PH^i(B))\|, 
    $$
    where the implied constants depend only on $n$.
  \end{enumerate}
\end{cor}

\begin{proof}
  For the proof of \eqref{genericbettibound:item1}, we apply the spectral sequences associated to
  the perverse filtrations of $A$ and $B$ to get
  \[
    \sum_{i \in \Zz} h^i(\Pp^n_{k'}, A \otimes l_{\mtrx{g}}^* B ) \leq
    \sum_{p\in\Zz} \sum_{q\in\Zz} \sum_{i \in \Zz} h^i(\Pp^n_{k'},
    \PH^{p}(A) \otimes l_{\mtrx{g}}^* \PH^q(B) ).
  \]
  Applying Proposition \ref{bilinearbettibound}, and denoting by~$M_n$
  the norm of the bilinear form~$b_n$, so that the inequality
  $b_n(\alpha,\beta) \leq M_n \|\alpha\| \ \|\beta\| $ holds for
  any~$\alpha$ and~$\beta$, we get
  \[
    \sum_{i \in \Zz} h^i(\Pp^n_{k'}, A \otimes l_{\mtrx{g}}^* B )\leq
    M_n \sum_{p\in\Zz} \sum_{q\in\Zz} \| \cc(\PH^p(A))\|
    \ \|    \cc(\PH^q(B))\|,
  \]
  which is at most $16(n+1)^2M_n c(A)c(B)$ by Proposition
  \ref{characteristicclassbound}. 
  
  For~\eqref{genericbettibound:item2}, one bound follows from
  Proposition~\ref{characteristicclassbound}. In the other direction,
  we begin as above (with the same~$\uple{g}$) until we reach the
  bound
  \[
    \sum_{i \in \Zz} h^i(\Pp^n_{k'}, A \otimes l_{\mtrx{g}}^* B ) \ll
     \sum_{p \in \Zz} \| \cc(\PH^p(A))\| \  \sum_{q
      \in\Zz} \| \cc(\PH^q(B))\|.
  \]
  Taking~$A = l_{\mtrx{a_m}*} \Ql$, where $\mtrx{a}_m$ is
  an arbitrary element of $M^{n+1, m+1}(k)$, this gives
  \[\sum_{i \in \Zz} h^i(\Pp^m_{k'}, l_{\mtrx{a}_m}^* l_{\mtrx{g}}^* B )
    \ll \sum_{q \in \Zz} \| \cc(\PH^q(B))\|
  \]
  for $0\leq m\leq n$. This implies the desired bound, since
  $\mtrx{g} \mtrx{a}_m$ is a generic point of~$M^{n+1,m+1}_k$ by Lemma
  \ref{generic-mixing}.
\end{proof}

Our main theorem in this section is the same as
Corollary~\ref{genericbettibound}, but without the generic pullback
$l_{\mtrx{g}}^*$. To remove it, we will use a double induction
strategy. We view a generic point $\uple{g}$ of $\GL_{n+1}$ as a generic
diagonal matrix conjugated by a generic invertible matrix. We will
show that the theorem remains true if some of the entries of the
diagonal matrix are set to $1$, by induction on the number of
entries. Once all the diagonal entries are set to~$1$, any conjugate
is the identity matrix, and so we obtain our desired statement.

At each step, we have a family of cohomology groups parameterized by
$\Gg_m$, i.e. a complex of sheaves on $\Gg_m$, and we want to bound
the stalk at the identity using the stalk at the generic point. In
general, the stalk at the generic point could be very small and the
stalk at the identity could still be very large. However, we will now
show that this cannot happen as long as we also control the
cohomology. In our case, it turns out that controling the cohomology
corresponds to a lower-dimensional version of the problem, and we may
use induction on the dimension to achieve our goal.

\begin{lem}\label{gmformula}
  For each object $A$ of $\Der(\Gg_{m,k})$, the inequality 
  \begin{equation}\label{gminequality}
    \sum_{i \in \Zz} \dim \mathcal
    H^i(A)_1 \leq \sum_{i \in \Zz} \dim \mathcal H^i(A)_\eta + \sum_{i
      \in \Zz} h^i_c(\Gg_{m,k},A)
  \end{equation} holds, where $\eta$ denotes the generic point of $\Gg_{m, k}$.  
\end{lem}

\begin{proof} 
We first reduce the proof to the case when $A$ is perverse. On the one hand, the spectral sequence associated to the perverse filtration of $A$ gives the inequality
\[ 
\sum_{i \in \Zz} \dim \mathcal H^i(A)_1 \leq \sum_{j \in \Zz} \sum_{i \in \Zz} \dim \mathcal H^i(\PH^j(A))_1.
\]
On the other hand, because perverse sheaves are supported in a single degree at the generic point, we have
\[ \sum_{i \in \Zz} \dim \mathcal H^i(A)_\eta= \sum_{j \in \Zz} \sum_{i \in \Zz} \dim \mathcal H^i(\PH^j(A))_\eta.\] 
Because the compactly supported cohomology of an affine curve with
coefficients in a perverse sheaf is concentrated in degrees $0$ and $1$,
the spectral sequence for the perverse filtration of $A$ degenerates
and
\[ \sum_{i \in \Zz} h^i_c(\Gg_{m,k},A) = \sum_{j \in \Zz} \sum_{i \in
    \Zz}
  h^i_c(\Gg_{m,k},\PH^j(A)).\]
Therefore, it is sufficient to prove the
inequality~(\ref{gminequality}) when~$A$ is a perverse sheaf.

If~$A$ is perverse, most of the terms in the sums
in~(\ref{gminequality}) vanish. Removing all the terms that are known
to vanish, the desired inequality can be stated as
\[
  \dim \mathcal H^0(A)_1 + \dim \mathcal H^{-1}(A)_1  \leq
  \dim \mathcal H^{-1}(A)_\eta +  h^0_c (\Gg_{m,k},A) + h^1_c (
  \Gg_{m,k},A).
\]
Since $\dim \mathcal H^{-1}(A)_1 \leq \dim \mathcal{H}^{-1}(A_\eta)$, it
suffices to show the inequality 
\[ 
\dim \mathcal H^0(A)_1    \leq   h^0_c (\Gg_{m,k},A) +
  h^{1}_c ( \Gg_{m,k},A), 
  \]
or even
\[ 
\chi_c(\Gg_{m,k},A) \geq \dim \mathcal H^0(A)_1.
\]
This follows from the Euler--Poincaré characteristic formula for perverse
sheaves on smooth curves, where the local term at the point $1$ is
at least $ \dim \mathcal H^0(A)_1$, all other local terms are
non-negative, and the global term is a multiple of the Euler--Poincaré
characteristic of $\Gg_m$ and hence vanishes.
\end{proof}

Let~$n\geq 1$ be an integer. For each point $\mtrx{g}$ of $GL_{n+1, k}$ valued in an algebraically closed field, the fixed points of the linear map
$l_{\mtrx{g}}$ are the projectivizations of the eigenspaces of
$\mtrx{g}$.  In particular, for any $\lambda\not=1$, the fixed points
of~$l_{\Diag(\lambda,1,\dots,1)}$ on $\Pp^n$ consist of the isolated
point~$x_0=[1:0:\cdots:0]$ and the hyperplane $H_0=\{[0:a_1:\cdots:a_n]\}$. We identify $H_0$ with $\Pp^{n-1}$.

In the next two lemmas, we will denote by $U \subseteq \Pp^n$ the
complement of $\{x_{0}\}\cup H_{0}$, by~$j\colon U \to \Pp^n$ the corresponding open
immersion, and by $\pi\colon U \to H_0\simeq \Pp^{n-1}$ the projection
map from~$U$ to the hyperplane~$H_{0}$.  We also denote by
$p\colon \{x_{0}\} \to \Pp^n$ the closed immersion of the isolated
fixed point and by $h\colon H_{0}=\Pp^{n-1} \to \Pp^n$ the closed
immersion of the hyperplane.

\begin{lem}\label{gmdegenerationbound}
  Let~$n\geq 1$ be an integer. Let $A$ and $B$ be objects of $\Der(\Pp^n_k)$. Let
  $\lambda$ be a geometric generic point of $\Gg_{m,k}$ defined over
  an algebraically closed field $k'$. With notation as above, the following inequality holds: 
  \begin{align*}
    \sum_{i \in \Zz} h^i(\Pp^n_k, A \otimes B) \leq &\sum_{i \in \Zz}
    h^i(\Pp^n_{k'}, A \otimes l_{\Diag(\lambda,1,\dots,1)}^*B) \\
    &+2 \sum_{i \in \Zz} h^i(\Spec(k), p^*A \otimes p^* B)+ 2 \sum_{i
      \in \Zz}h^i(\Pp^{n-1}_k, h^* A \otimes h^* B)
    \\
    &\qquad \qquad + \sum_{i \in \Zz} h^i(\Pp^{n-1}_k, \pi_! j^* A \otimes \pi_!  j^*
    B).
  \end{align*}
\end{lem}

\begin{proof}
  Let $K$ be the pushforward of
  $ A \otimes l_{\Diag(\mu,1,\dots,1)}^*B$ from $\Pp^n \times \Gg_m$
  to $\Gg_m$, where $\Gg_m$ has coordinate $\mu$. Proper base
  change yields the equalities 
  \begin{gather*}
    \sum_{i \in \Zz} \dim \mathcal H^i(K)_1= \sum_{i \in \Zz}
    h^i(\Pp^n_k, A \otimes B),
    \\
   \sum_{i \in \Zz} \mathcal H^i(K)_\eta = \sum_{i \in \Zz}
   h^i(\Pp^n_{k'}, A \otimes l_{\Diag(\lambda,1,\dots,1)}^*B).
 \end{gather*}
Hence, by Lemma \ref{gmformula}, it suffices to prove the inequality
  \begin{multline*}
    \sum_{i \in \Zz} h^i_c(\Gg_{m,k},K)\leq
    2 \sum_{i \in \Zz} h^i(\Spec(k), p^*A \otimes p^* B)
    \\+ 2 \sum_{i
      \in \Zz}h^i(\Pp^{n-1}_k, h^* A \otimes h^* B) +
    \sum_{i \in \Zz}
    h^i(\Pp^{n-1}_k, (\pi_! j^* A) \otimes (\pi_!  j^* B)).
  \end{multline*}
  By the Leray spectral sequence, the equality
  \begin{equation}\label{eq-leray}
    \sum_{i \in \Zz} h^i_c(\Gg_{m,k},K)= \sum_{i \in \Zz} h^i_c\left(
      \Pp^n \times \Gg_m, A \otimes
      l_{\Diag(\mu,1,\dots,1)}^*B\right).
  \end{equation} holds. To compute the right-hand side, we partition
  $\Pp^n \times \Gg_m=X_0\cup X_1\cup X_2$ as
  \begin{displaymath}
    X_0=\{x_0\} \times \Gg_m, \quad 
    X_1=H_0\times \Gg_m, \quad
    X_2=U\times \Gg_m. 
  \end{displaymath}
  Note that~$X_0$ and~$X_1$ are closed, and~$X_2$ is open.  By
  excision, the right-hand side of~(\ref{eq-leray}) is bounded by
  $$
  \sum_{j=0}^2\sum_{i\in\Zz} h^i_c(X_j,A \otimes
  l_{\Diag(\mu,1,\dots,1)}^*B).
  $$
  Since $X_0$ and $X_1$ are fixed points of
  $l_{\Diag(\lambda,1,\dots,1)}^*$ for all~$\lambda$, the restriction of the complex $ A \otimes l_{\Diag(\lambda,1,\dots,1)}^*B$ to either of them is the same as that of $A \otimes B$. The Künneth
  formula gives
  \begin{gather*}
    \rmH_c^*(X_0,A\otimes B)\simeq \rmH^*(\{x_0\},A\otimes B)\otimes
    \rmH_c^*(\Gg_m,\Ql) \\
    \rmH_c^*(X_1,A\otimes B)\simeq \rmH^*(H_0,A\otimes B)\otimes
    \rmH_c^*(\Gg_m,\Ql),
  \end{gather*}
 and hence (since the sum of Betti numbers of~$\Gg_m$ is equal to~$2$) we
  get
 \begin{align*}
  \sum_{i\in\Zz} h^i_c(X_0,A\otimes
  B)&=2\sum_{i\in\Zz}h^i(\{x_0\},A\otimes B)= 2 \sum_{i \in \Zz}
  h^i(\Spec(k), p^*A \otimes p^* B) \\
    \sum_{i\in\Zz} h^i_c(X_1,A\otimes B)&=2\sum_{i\in\Zz}h^i(H_0,A\otimes
  B)= 2 \sum_{i \in \Zz} h^i(\Pp^{n-1}, h^*A \otimes h^*B).
  \end{align*}
  Finally, note that the $\Gg_m$-action on $U$ is free, so that there is an isomorphism between $X_2$ and the space of pairs of points~$(x,y)$ in $U\times U$ such that $x$ and $y$ lie in the same orbit of the $\Gg_m$-action, which identifies the projection to $U$ and the pullback along $l_{\Diag(\lambda,1,\dots,1)}$ with the first and the second projections. The space of such pairs is
  isomorphic to the fiber product of~$U$ with itself over the quotient~$U/ \Gg_m$. There is an isomorphism
  $U / \Gg_m \simeq H_0\simeq \Pp^{n-1}$ induced by $\pi$, hence the equality
  \[
    \sum_{i\in\Zz} h^i_c( X_2, A \otimes
    l_{\Diag(\lambda,1,\dots,1)}^*B) =\sum_{i\in\Zz}h^i_c(
    U\times_{\Pp^{n-1}} U, j^*A \boxtimes j^*B). 
  \]
By the K\"{u}nneth formula (or the projection formula applied twice), the right-hand side is also equal to
  \[ \sum_{i\in\Zz} h^i_c( U, A \otimes \pi_! j^* B)= \sum_{i\in\Zz}
    h^i(\Pp^{n-1}_k, \pi_! j^* A \otimes \pi_!  j^*B ).
  \]
  Gathering all these computations, we obtain the result.
\end{proof}

\begin{lem}\label{complexityreductionbound}
  Let~$n\geq 1$.  Let $A$ be an object of $\Der(\Pp^n_k)$ and let
  $\mtrx{g}$ be a geometric generic point of~$\GL_{n+1, k}$ defined
  over an algebraically closed field $k'$. Then the following inequalities hold: 
  \begin{gather*}
    c(h^* l_{\mtrx{g}}^*A) \leq c(A), \\
    c( \pi_! j^* l_{\mtrx{g}}^*A) \leq 3 c(A),\\
    \sum_{i \in \Zz} h^i(\Spec(k'), p^* l_{\mtrx{g}}^*A) \leq c(A).
  \end{gather*}
\end{lem}

\begin{proof} 
  For each $0\leq m\leq n-1$, let $\mtrx{a}_m$ be a geometric generic point of
  $M^{n,m+1}_{k(\mtrx{g})}$.  By
  definition, 
  $$
  c(h^* l_{\mtrx{g}}^*A) = \sup_m \sum_{i \in \Zz} h^i(\Pp^m,
  l_{\mtrx{a}_m}^* h^* l_{\mtrx{g}}^* A).
  $$
  The composition $l_{\mtrx{g}} \circ h \circ l_{\mtrx{a}_m}$ is a
  generic linear embedding from $\Pp^m$ to $\Pp^n$, and hence this
  quantity is~$\leq c(A)$ by definition.

  For any $i\in\Zz$, by proper base change, there is an isomorphism
  \begin{equation}\label{eq-betti}
    \rmH^i(\Pp^m, l_{\mtrx{a}_m}^* \pi_! j^* l_{\mtrx{g}}^* A)=\rmH^i(\pi^{-1}l_{\uple{a}_m}(\Pp^m),l_{\uple{g}}^*A).
  \end{equation}
The inverse image $\pi^{-1}l_{\uple{a}_m}(\Pp^m)$ is of the form
  $$
  L-\{x_0\} -(H_0\cap L), 
  $$
  where $L$ is a general $(m+1)$-dimensional subspace of $\Pp^n$ that
  contains~$x_0$. By excision, the sum of the Betti
  numbers of~(\ref{eq-betti}) is at most
  $$
  \sum_{i\in\Zz} h^i(L,l_{\uple{g}}^*A) + \sum_{i\in\Zz}
  h^i(\{x_0\},l_{\uple{g}}^*A)+ \sum_{i\in\Zz} h^i(H_0\cap
  L,l_{\uple{g}}^*A).
  $$
  Each of these is bounded by~$c(A)$ (since $\mtrx{a}_m$ and
  $l_{\mtrx{g}}$ are generic, see~Lemmas \ref{transcendenceswapping}
  and~\ref{generic-mixing}), hence
  $c(\pi_!j^*l_{\uple{g}}^*A)\leq 3c(A)$.
\par
Finally, $l_{\mtrx{g}} \circ p$ is the inclusion map of a generic
point, so the inequality 
\[
  \sum_{i \in \Zz} h^i(\Spec(k'), p^* l_{\mtrx{g}}^*A) \leq c(A)
\]
holds by the definition of $c(A)$. 
\end{proof}

We are finally ready to prove the main theorem of this section.

\begin{proof}[Proof of Theorem \ref{main}] We prove the theorem by
  induction on $n\geq 0$. The theorem holds for $n=0$, because $A$ and
  $B$ are then simply complexes of vector spaces so
  $$
  c(A)=\sum_{i\in\Zz} h^i(\Pp^0,A), 
  $$
 and hence
  $$
  \sum_{i \in \Zz} h^i(\Pp^0,A \otimes B)=\sum_{i\in\Zz}h^i(\Pp^0,A)
  \sum_{i\in\Zz}h^i(\Pp^0,B).
  $$

  Now assume that~$n\geq 1$ and that Theorem~\ref{main} holds for
  $\Pp^{n-1}$.

  Let $\mtrx{g}$ be a generic point of $\GL_{n+1,k}$ and let
  $\lambda_1,\dots,\lambda_{n+1}$ be independent transcendentals over
  $k(\mtrx{g})$. Let $k'$ be an algebraically closed extension of
  $k(\mtrx{g})$ containing $\lambda_1,\dots,\lambda_{n+1}$.
  For~$1\leq m\leq n+1$, we denote
  $$
  \delta_m=\Diag(1,\ldots,1,\lambda_{m},\ldots,\lambda_{n+1}).
  $$ Then
  \[
    \sum_{i \in \Zz} h^i(\Pp^n_k, A \otimes B)  = \sum_{i \in \Zz}
    h^i(\Pp^n_{k'}, l_{\mtrx{g}}^* A \otimes l_{\mtrx{g}}^* B) =
    \sum_{i \in \Zz} h^i(\Pp^n_{k'},
    l_{\Diag(1,\dots,1)}^*l_{\mtrx{g}}^* A \otimes l_{\mtrx{g}}^* B).
  \]

On the other hand,  we have
\begin{align*}
  \sum_{i \in \Zz} h^i(\Pp^n_{k'},
  l_{\delta_{n+1}}^*l_{\mtrx{g}}^* A \otimes
  l_{\mtrx{g}}^* B)
  &= \sum_{i \in \Zz} h^i(\Pp^n_{k'},A \otimes  l_{\mtrx{g}
    \delta_{n+1}\mtrx{g}^{-1}}^* B ) \\
  &\ll c(A) c(B)
\end{align*}
by Corollary~\ref{genericbettibound}\,(1) and Lemma~\ref{final
  generic}\,(1).

Therefore, using a telescoping sum, we get
\begin{multline*}
  \sum_{i \in \Zz} h^i(\Pp^n_k, A \otimes B) \\= \sum_{m=1}^{n+1}
  \Bigl( \sum_{i \in \Zz} h^i(\Pp^n_{k'}, l_{\mtrx{g}}^* A \otimes
  l_{\delta_{m+1}}^* l_{\mtrx{g}}^* B) - \sum_{i \in \Zz}
  h^i(\Pp^n_{k'}, l_{\mtrx{g}}^* A \otimes
  l_{\delta_m}^*
  l_{\mtrx{g}}^* B) \Bigr)
  \\
  + O( c(A) c(B)).
\end{multline*}

By Lemma \ref{gmdegenerationbound}, for a fixed~$m$, the term between
parentheses is
\begin{multline*}
  \leq 2 \sum_{i \in \Zz} h^i(\Spec k', p^*l_{\mtrx{g}}^*A \otimes p^*
  l_{\delta_{m+1}}^*
  l_{\mtrx{g}}^*B) \\
  + 2 \sum_{i \in \Zz}h^i(\Pp^{n-1}_{k'}, h^* l_{\mtrx{g}}^*A \otimes
  h^* l_{\delta_{m+1}}^*l_{\mtrx{g}}^*B)
  \\
  +\sum_{i \in \Zz} h^i(\Pp^{n-1}_{k'}, (\pi_! j^* l_{\mtrx{g}}^*A)
  \otimes (\pi_! j^*
  l_{\delta_{m+1}}^* l_{\mtrx{g}}^*B)).
\end{multline*}
Applying the induction hypothesis, this is
$$
\ll 2 c( p^*l_{\mtrx{g}}^*A) c( p^*
l_{\delta_{m+1}}^*l_{\mtrx{g}}^*B) +2 c(h^*l_{\mtrx{g}}^*A) c(h^*
l_{\delta_{m+1}}^* l_{\mtrx{g}}^*B) + c( \pi_!  j^*l_{\mtrx{g}}^*A)
c( \pi_! j^* l_{\delta_{m+1}}^*l_{\mtrx{g}}^*B)
$$
and finally, applying Lemma \ref{complexityreductionbound} and Lemma
\ref{final generic}\,(2), this is
\[
  \leq 2 c(A) c(B) + 2 c(A) c(B) + 9c(A) c(B) =13 c(A)c(B).
\]
This concludes the proof.
\end{proof}


\section{Quantitative sheaf theory on quasi-projective varieties}
\label{sec-6ops}

The goal of this section is to define the complexity of objects of the
derived category for any quasi-projective algebraic variety over a
field, or rather the pair consisting of such a variety and a given
quasi-projective embedding.

\subsection{Definition of complexity and continuity of the six
  operations}

It turns out that it is not really possible to define a complexity
invariant of sheaves on an algebraic variety, with its expected
properties, that only depends on the algebraic variety. We give two
examples, related to two desirable properties of a complexity function
(the second was suggested by an anonymous referee).

\begin{example}\label{ex-counter}
  (1) A hypothetical canonical complexity invariant $c(A)$ for
  objects~$A$ of $\Der(X)$ should of course be invariant under
  automorphisms of~$X$. Another basic requirement for applications of
  a complexity invariant should be that, if~$\Ff$ is a finite field
  and~$k$ an algebraic closure of~$\Ff$, then there should be, up to
  isomorphism over~$k$, only finitely many irreducible perverse
  sheaves of bounded complexity (see
  Corollary~\ref{cor-finiteness}). However, there are examples of
  algebraic varieties~$X$ over~$\Ff$ with infinitely many
  $\Ff$-automorphisms and perverse sheaves~$A$ on~$X$ such that there
  are infinitely many pairwise non-isomorphic perverse sheaves among
  the $\sigma^*A$ for such automorphisms~$\sigma$ (e.g., let~$X$ be
  the affine plane, and let~$A$ be a suitable shift of the constant sheaf~$\Ql$ on an
  irreducible plane curve).
  \par
  (2) Consider $X=\Aa^2$, with coordinates $(x,y)$. For any polynomial
  $f\in k[x,y]$, the map $u_f\colon (x,y)\mapsto (x,y-f(x))$ is an
  automorphism of~$X$. Let $i\colon \Aa^1\to X$ be the closed immersion $x\mapsto (x,0)$. For any complexity function~$c$ on~$\Der(X)$
 that satisfies the ``bilinearity'' property of Corollary~\ref{tensor1} and is invariant under automorphisms, the following~estimate would hold: 
  $$
  c(i_*\Ql[1]\otimes u_f^*(i_*\Ql[1]))\ll 1. 
  $$
  This is impossible since the support of
  $i_*\Ql[1]\otimes u_f^*(i_*\Ql[1])$ is the intersection of the
  line~$y=0$ and its image under~$u_f$, which can be an arbitrarily
  large finite set. A similar projective example arises from
  $X=E\times E$ for an elliptic curve~$E$, with the immersion
  $E\simeq E\times\{0\}\to E\times E$ and the automorphisms
  $(x,y)\mapsto (x,y+nx)$ for $n\in\Zz$.
\end{example}

We will work instead with pairs $(X,u)$ consisting of an algebraic variety $X$ over~$k$ and a locally closed
immersion~$u\colon X\to \Pp^n_k$ for some integer~$n\geq 0$, called the
\emph{embedding dimension} of~$(X,u)$. 

\begin{remark}
  (1) Recall that a locally closed immersion~$u$ is a morphism that can be
  factored as~$u=i\circ j$, where $j$ is an open immersion and $i$ is a
  closed immersion.
  \par
  (2) We will often simply denote by $X$ the pair $(X,u)$, when no
  confusion is possible, and call this simply a \emph{quasi-projective
    variety}. We will also sometimes do this over a more general
  base~$S$ than the spectrum of a field.
\end{remark}

\begin{defi}[Complexity of a complex of sheaves on a quasi-projective
  variety]\label{def-ca}
  Let~$(X,u)$ be a quasi-projective variety over an algebraically closed
  field $k$. The \emph{complexity relative to~$u$} of an object $A$ of
  $\Der(X)$ is the non-negative integer
  \[
    c_u(A) = c( u_{!}  A).
  \]
  When the embedding is clearly understood, we will simply speak of the complexity of~$A$. 
\end{defi}

\begin{remark}
  The complexity $c_u(A)$ vanishes if and only if~$A=0$ (Proposition~\ref{rm-zero}). 
\end{remark}

The first main objective is to prove that the complexity is under
control when performing all usual operations, starting with
Grothendieck's six functors (i.e., in the language used in the
introduction, these functors are ``continuous''). For this purpose, we
will also define a complexity invariant for morphisms.

\begin{defi}[Complexity of a locally closed immersion]
  Let $(X,u)$ be a quasi-projective variety over $k$. We define the
  complexity of~$u$ to be 
  $$
  c(u)=c_u(\Ql)=c (u_! \Ql). 
  $$
\end{defi}

\begin{defi}[Complexity of a morphism]\label{def-cuv}
  Let $f\colon (X,u)\to (Y,v)$ be a morphism of quasi-projective
  varieties over~$k$. Let $m$ and $n$ be the embedding dimension
  of~$X$ and $Y$ respectively. For integers $0\leq p\leq m$ and
  $0\leq q\leq n$, let $\mtrx{a}_{p}$ and $\mtrx{b}_{q}$ be geometric
  generic points of $M^{m+1,p+1}_k$ and $M^{n+1,q+1}_k$ respectively,
  all defined over a common algebraically closed field $k'$. 
  
  The
  \emph{complexity of the morphism~$f$ relative to $(u,v)$} is defined
  as
  \[
    c_{u,v}(f)= \max_{0\leq p\leq m } \max_{0\leq q\leq n} \sum_{i \in
      \Zz} h^i_c(X_{k'}, u^* l_{\mtrx{a}_{p} *}\Ql \otimes f^* v^*
    l_{\mtrx{b}_{q}*} \Ql ) .
  \]
\end{defi}

In the vein of Lemma \ref{lem-complexity-open-set}, the definition of
$c_{u, v}(f)$ can be phrased in terms of closed points over a dense open
subset $W \subset M^{m+1,p+1}_k\times M^{n+1,q+1}_k$ rather than
geometric generic points. Namely, letting
$\Gamma \subset \Pp^m \times \Pp^n$ denote the graph of $f$ relative to
the locally closed immersions $u$ and $v$, there exist such $W$ with the
property that the above Betti numbers are equal to
\[
h^j_c(\Gamma \cap (\ker(l_{\mtrx{a}_{p}}) \times \ker( l_{\mtrx{b}_{q}})), \Ql) 
\] for all points $(\mtrx{a}_{p},\mtrx{b}_{q}) \in W(\bar k)$.

\begin{remark}
\label{rem-complex-immersion}
These definitions are compatible in the sense that, for a locally closed
immersion $u\colon X \to \Pp^n$, using Lemma \ref{generic-intersection}, we
have $c_{u,\mathrm{Id}}(u)=c(u)$.
  
In a similar vein, if $i \colon X\to Y$ is an immersion of quasi-projective
varieties $(X,u)$ and $(Y,v)$, such that $u=v\circ i$, then
$c_{u,v}(i)=c(u)$, which can be seen readily from the definition, the
projection formula and Lemma \ref{generic-intersection}.
\end{remark}

The fundamental result is the following:

\begin{theo}[Continuity of the six operations]\label{pushpull}
  Let $(X,u)$ and $(Y,v)$ be quasi-projective varieties over~$k$, and
  let $f\colon X\to Y$ be a morphism. For objects~$A$ and~$B$ of~$\Der(X)$ and~$C$ of~$\Der(Y)$, the following estimates hold: 
  \begin{align}
    c_{u}(\dual(A))&\ll c(u) c_{u}(A)
                 \label{eq-6dual}
    \\
    c_{u}(A\otimes B)&\ll c_{u}(A)c_{u}(B)
                         \label{eq-6tensor}
    \\
    c_u(A\otimes_! B)&\ll c(u)^3c_u(A)c_u(B)
                       \label{eq-6tensorshriek}
    \\
    c_{u}(\rhom(A,B))
               &\ll c(u)c_{u}(A)c_{u}(B)
                   \label{eq-6rhom}
    \\
    c_{u} ( f^* C)& \ll c_{u,v}(f) c_{v}(C)
                      \label{eq-6pullback}
    \\
    c_{v} (f_! A) &\ll c_{u,v}(f) c_{u}(A)
                      \label{eq-6extzero}
    \\
    c_{u} ( f^! C) &\ll c(u) c(v) c_{u,v}(f) c_{v}(C)
                       \label{eq-6pullbackshriek}
    \\
    c_{v} (f_* A) &\ll c(u) c(v) c_{u,v}(f)c_{u}(A).
                      \label{eq-6pushforward}
  \end{align}
  In all these estimates, the implied constants only depend on the
  embedding dimensions of~$(X,u)$ and $(Y,v)$.
\end{theo}

In the next subsection, we will prove Theorem~\ref{pushpull}.  In
later subsections, we will then handle similarly a number of
additional operations, e.g., decomposing a complex into irreducible
perverse sheaves.

\begin{remark}
  Let~$k_0$ be a field of characteristic coprime to~$\ell$.  For any
  quasi-projective variety $(X_0,u_0)$ over $k_0$ and any object~$A_0$
  of $\Der(X_0)$, we set 
  $$
  c_{u_0}(A_0)=c_{u}(A),
  $$ where~$(X,u)$ is the base change of $(X_0, u_0)$ to an algebraically closed field extension~$k$
  of~$k_0$ and $A$ is the base change of~$A_0$. Similarly, we define the complexity of a
  morphism $f_0\colon (X_0,u_0)\to (Y_0,v_0)$ as
  $$
  c_{u_0,v_0}(f_0)=c_{u,v}(f), 
  $$ where $f$ is the base change of~$f_0$ to an algebraically closed field. Since the complexity is independent of the chosen algebraically closed field, Theorem \ref{pushpull} also holds for the complexity over non-algebraically
  closed fields.
\end{remark}

\begin{remark}
  Theorem~\ref{th-sample} from the introduction follows from
  Theorem~\ref{pushpull} applied to $X=\Aa^n$ and~$Y=\Aa^m$ with~$u$
  and~$v$ the standard open immersions to~$\Pp^n$ and~$\Pp^m$
  respectively.
\end{remark}

\subsection{Proof of the continuity of the six functors}\label{six}

We begin with a lemma concerning complexes on projective space.

\begin{lem}\label{dual1}
  For each object $A$ of $\Der(\Pp^n_k)$, the equality $c(A)= c(\dual(A))$ holds.
\end{lem}

\begin{proof}
  Let $\mtrx{a} $ be a geometric generic point of of $M^{n+1,m+1}_k$. Using
  Verdier duality and Lemma~\ref{generic-dual}, the computation
  \[
    \sum_{i \in \Zz} h^i(\Pp^m, l_{\mtrx{a}}^* A) = \sum_{i \in \Zz}
    h^i(\Pp^m, \dual (l_{\mtrx{a}}^* A) ) = \sum_{i \in \Zz} h^i(\Pp^m,
    l_{\mtrx{a}}^* \dual(A))
  \]
  yields $c(A) = c(\dual(A))$ by definition of complexity. 
\end{proof}

\begin{proof}[Proof of~\eqref{eq-6dual}]
 The equalities 
  $$
  c_u(\dual(A)) = c(u_!\dual(A))= c( u_* \dual(A) \otimes u_! \Ql)
  $$ hold by definition, and hence we get the estimate
  $$
  c_u(\dual(A))= c( \dual(u_! A) \otimes u_! \Ql) \ll c(\dual(u_! A))
  c(u_!  \Ql) = c(u_! A) c(u_! \Ql)= c_u(A) c(u)
  $$
  by combining Corollary \ref{tensor1} and Lemma \ref{dual1}.
\end{proof}

\begin{proof}[Proof of~\eqref{eq-6tensor}
  and~\eqref{eq-6tensorshriek}] By definition, we have
  $$
  c_u(A\otimes B)=c(u_!(A\otimes B)).
  $$ Then by Corollary~\ref{tensor1}, we obtain
  \[
    c_u(A \otimes B) = c(u_!(A \otimes B))= c(u_! A \otimes u_! B) \ll
    c(u_! A) c(u_! B) =c_u(A) c_u(B).
  \]
  Combining this with~(\ref{eq-6dual}), we
  obtain~(\ref{eq-6tensorshriek}).
\end{proof}

\begin{proof}[Proof of~\eqref{eq-6rhom}]
  We have by definition of internal Hom 
  \[
    c_u(\rhom(A,B))= c_u (\dual(A) \otimes B) \ll c_u(\dual(A))c_u(B) \ll c(u)
    c_u (A) c_u(B),
  \]
  where the estimates follow by applying~(\ref{eq-6tensor}) and~(\ref{eq-6dual}).
\end{proof}

\begin{proof}[Proof of~\eqref{eq-6pullback} and~\eqref{eq-6extzero}]
  Let $m$ be the embedding dimension of~$X$ and~$n$ that of~$Y$. By
  definition, we have
  \[
    c_{u}( f^* C) = c( u_{!} f^* C) = \max_{0\leq p\leq m} \sum_{i \in
      \Zz} h^i (\Pp^{p}_{k'}, l_{\uple{a}_p}^* u_{!} f^* C).
  \]
  Applying three times the projection formula, we have for any $p$ and
  $i\in\Zz$
  $$
  h^i (\Pp^{p}_{k'}, l_{\uple{a}_p}^* u_{!} f^* C)=
  h^i_c(Y_{k'}, f_! u^* l_{\mtrx{a}_{p}!} \Ql \otimes C).
  $$
  Applying the Leray spectral sequence for étale cohomology with
  compact support to the morphism $v\colon Y\to \Pp^n$, we derive
  $$
  \sum_{i \in \Zz} h^i_c(Y_{k'}, f_! u^* l_{\mtrx{a}_{p}!} \Ql \otimes
  C)= \sum_{i\in\Zz} h^i(\Pp^{n}_{k'}, v_! f_! u^*
  l_{\mtrx{a}_{p}!} \Ql \otimes v_! C).
  $$
  Then, by Theorem~\ref{main} and the definition, we get
  $$
  \sum_{i\in\Zz} h^i(\Pp^{n}_{k'}, v_! f_! u^* l_{\mtrx{a}_{p}!} \Ql
  \otimes v_! C) \ll c_v( f_! u^* l_{\mtrx{a}_{p}!} \Ql ) c_{v}(C).
  $$
  Thus, 
  $$
  c_{u}( f^* C)\ll c_v(C)\max_{0\leq p\leq
    m}c_v(f_!u^*l_{\uple{a}_p!}\Ql).
  $$
  Applying again the definition and the projection formula three
  times, we have
  \begin{align*}
    c_v(f_!u^*l_{\uple{a}_p!}\Ql)
    &=
      \max_{0\leq q\leq n}
      \sum_{i\in\Zz}
      h^i(\Pp^q,l^*_{\uple{b}_q}v_!f_!u^*l_{\uple{a}_p!}\Ql)
    \\
    &=\max_{0\leq q\leq n}\sum_{i\in\Zz}
      h^i_c(X_{k'}, f^*v^*l_{\uple{b}_q!}\Ql\otimes
      u^*l_{\uple{a}_p!}\Ql)
  \end{align*}
  for any~$p$.  This concludes the proof of~(\ref{eq-6pullback}) by definition
  of~$c_{u,v}(f)$, and the proof of~(\ref{eq-6extzero}) is similar.
\end{proof}

\begin{proof}[Proof of~\eqref{eq-6pullbackshriek} and~\eqref{eq-6pushforward}]
  From $f^!C=\dual(f^*\dual(C))$ we get the estimates
  $$
  c_u(f^!C)\ll c(u)c_u(f^*\dual(C))\ll c(u)c_{u,v}(f)c_v(\dual(C))
  \ll c(u)c(v) c_{u,v}(f)c_v(C)
  $$
  by~(\ref{eq-6dual}) and~(\ref{eq-6pullback}). This
  establishes~(\ref{eq-6pullbackshriek}), and~(\ref{eq-6pushforward})
  is proved similarly using~(\ref{eq-6extzero}) instead of~(\ref{eq-6pullback}). 
\end{proof}

We can also bound the complexity, with respect to the Segre embedding,
of the external product of complexes. 

\begin{prop}\label{pr-external}
  Let~$(X,u)$ and~$(Y,v)$ be quasi-projective algebraic varieties
  over~$k$ with embedding dimensions~$n$
  and~$m$ respectively, and consider their product $(X\times Y,u\boxtimes v)$, where~$u\boxtimes v=s\circ (u\times v)$ is the composition of $u \times v$ with the Segre embedding
  $$
  s\colon \Pp^n\times \Pp^m\to\Pp^{(n+1)(m+1)-1}. 
  $$ Let~$p_1\colon X\times Y\to X$ and~$p_2\colon X\times Y\to Y$ denote the projections. 
Then the estimate
  $$
  c_{u\boxtimes v}(A\boxtimes B)\ll c_{u\boxtimes
    v,u}(p_1)c_{u\boxtimes v,v}(p_2)c_u(A)c_v(B)
  $$
 holds for all objects $A$ of~$\Der(X)$ and~$B$ of~$\Der(Y)$, with an  implied constant that only depends on~$(n,m)$.
\end{prop}

\begin{proof}
  The estimate
  $$
  c_{u\boxtimes v}(A\boxtimes B)\ll c_{u\boxtimes v}(p_1^*A)
  c_{u\boxtimes v}(p_2^*B)
  $$
 holds by~(\ref{eq-6tensor}). Since the first factor on the right-hand side satisfies
  $$
  c_{u\boxtimes v}(p_1^*A)\ll c_{u\boxtimes v,u}(p_1)c_u(A)
  $$
  by~(\ref{eq-6pullback}), and similarly for
  $c_{u\boxtimes v}(p_2^*B)$, the result follows.
\end{proof}  

\begin{remark}
 It should be possible to estimate the complexities $c_{u\boxtimes v,u}(p_1)$ and
  $c_{u\boxtimes v,v}(p_2)$ in terms of~$c(u)$ and~$c(v)$. Indeed, this
  amounts to estimating the cohomology of $u_! \Ql \boxtimes v_! \Ql$ on~$\Pp^n \times \Pp^m$ restricted to the intersection of a general
  linear subspace with the inverse image of a general linear subspace
  under the Segre embedding. One can bound this~by
  $$
  c_s ( u_!  \Ql \boxtimes v_! \Ql) c_s (\mathcal{F}),
  $$
  where $\mathcal{F}$ is the constant sheaf on the intersection of a
  general linear subspace with the inverse image of a general linear
  subspace under the Segre embedding, and then attempt to estimate the two factors
  separately.
\end{remark}

\subsection{Linear operations}\label{lin}

\begin{prop}\label{easylin}
  Let $(X,u)$ be a quasi-projective variety over~$k$.

  \begin{enumerate}
  \item\label{easylin:item1} For any distinguished triangle $A_1 \to A_2 \to A_3$ in
    $\Der(X)$, the following holds: 
    \begin{gather*}
      c_u(A_2) \leq c_u(A_1) + c_u(A_3),\\
      c_u(A_1) \leq c_u(A_2) + c_u(A_3),\\
      c_u(A_3) \leq c_u(A_1) + c_u(A_2).
    \end{gather*}
    
  \item\label{easylin:item2} For any objects $A$ and $B$ of $\Der(X)$, the following equality holds: 
    $$
    c_u(A\oplus B) = c_u(A) + c_u(B).
    $$

  \item\label{easylin:item3} Let $A$ be an object of $\Der(X)$. For any $h\in\Zz$, the following equality holds: 
      $$
      c_u(A[h])=c_u(A).
      $$
  \end{enumerate}
\end{prop}

\begin{proof}
  The first bound in \eqref{easylin:item1}, as well as the equalities in \eqref{easylin:item2} and \eqref{easylin:item3}, follow immediately
  from the definition~$c_u(A)=c(u_!A)$ and the expression of the
  complexity in terms of Betti numbers. The second and third
  inequalities are then deduced from the first (and from~\eqref{easylin:item3}) using the
  distinguished triangles $A_2\to A_3\to A_1[1]$ and $A_3[-1]\to A_1\to A_2.$
\end{proof}

In general, it is not obvious how sharp the first inequality
$c(A_2) \leq c(A_1) + c(A_3)$ is. However, in the important special
case of the decomposition of a perverse sheaf into its 
irreducible constituents, there is a more satisfactory estimate.

\begin{theo}\label{hardlin}
  Let $(X,u)$ be a quasi-projective variety over~$k$.  Let $A$ be an
  object of $\Der(X)$. For any $i\in\Zz$, let $n_i$ be the length of
  $\PH^i(A)$ in the abelian category of perverse sheaves and let
  $(A_{i,j})_{1\leq j\leq n_i}$ be the family of its Jordan-H\"{o}lder
  factors, repeated with multiplicity. Then the estimates
  \[
    c_u(A)\leq \sum_{i \in \Zz} \sum_{j=1}^{n_i} c_u(A_{i,j}) \ll c(u)
    c_u(A)
  \]
  hold, with an implied constant that only depends on the embedding dimension
  of $(X,u)$.
\end{theo}

\begin{proof}
  The first inequality follows immediately from the previous
  proposition. We prove the second one.

  For $i\in\Zz$, let $m_i$ be the length of $\PH^i(u_!A)$ in the
  category of perverse sheaves on the projective space target of~$u$,
  and let $(B_{i,j})_{1\leq j\leq m_i}$ be the Jordan-H\"{o}lder
  factors (repeated with multiplicity) of these perverse sheaves.
  \par
  Let~$\overline{X}$ be the closure of the image of~$u$. The $B_{i,j}$ are
  irreducible perverse sheaves with support contained in~$\overline{X}$,
  and hence each $u^* B_{i,j}$ is either zero or is an irreducible
  perverse sheaf on~$X$.

  Because $u_! A$ is supported on~$\overline{X}$, its perverse filtration
  is stable under pullback to~$\overline{X}$. The perverse filtration is
  always stable under open immersions; since $u$ is a locally closed
  immersion, it follows that it is stable under $u^*$.

  Therefore, we have $\PH^i(A) = \PH^i (u^*u_! A)=u^* \PH^i(u_!A)
  $. This perverse sheaf is an iterated extension of $u^* B_{i,j}, \dots, u^* B_{i,m_i}$; by the uniqueness of the Jordan-H\"{o}lder factors, the perverse sheaves $u^* B_{i,1}, \dots, u^* B_{i,m_i}$
  that do not vanish coincide with the $A_{i,1}, \dots, A_{i,n_i}$, including multiplicity. We may therefore assume that the equality $u^* B_{i,j}= A_{i,j}$ holds for~$1\leq j\leq n_i$ and $u^* B_{i,j}=0$ for
  $j>n_i$.

  Then we have
  \begin{align*}
    \sum_{i\in\Zz}\sum_{j=1}^{n_i}c_u(A_{i,j})=\sum_{i\in\Zz}\sum_{j=1}^{n_i}c_u(u^*B_{i,j})
    &\ll c(u)      \sum_{i\in\Zz}\sum_{j=1}^{m_i}c(B_{i,j})  \\
    &\ll 
      c(u)      \sum_{i\in\Zz}\sum_{j=1}^{m_i}\|\cc(B_{i,j})\|,
  \end{align*}
  by using~(\ref{eq-6pullback}) and Corollary~\ref{genericbettibound}\,(2).

  On the other hand, we have
  $$
  \cc(\PH^i(u_!A)) = \sum_{j=1}^{m_j} \cc(B_{i,j}).
  $$
  Since the $B_{i,j}$ are perverse sheaves, their characteristic
  cycles are effective. The norm on~$\Rr^{n+1}$ is additive on the
  cone of effective cycles, hence the equality
  $$
  \|\cc(\PH^i(u_!A))\| = \sum_{j=1}^{m_i}\| \cc(B_{i,j})\|. 
  $$
 The estimate 
  $$
  \sum_{i \in \Zz} \sum_{j=1}^{n_i} c_u(A_{i,j})
  \ll c(u)
  \sum_{i \in \Zz} \|\cc(\PH^i(u_!A))\|
  \ll c(u) c(u_! A) =c(u) c_u(A)
  $$
  then follows by Proposition~\ref{characteristicclassbound}.
\end{proof}

\begin{cor}\label{cor-middle}
  Let $(X,u)$ be a quasi-projective variety over~$k$ and
  $w\colon W\to X$ the embedding of a locally closed subvariety. For
  each perverse sheaf~$A$ on~$W$, the middle extension perverse
  sheaf~$w_{!*}A$ on~$X$ satisfies
  \begin{align*}
  c_{u\circ w}(A)&\ll c(u\circ w)c_u(w_{!*}A), \\
  c_u(w_{!*}A)&\ll (c(u)c(u\circ w))^2c_{u\circ w}(A),
  \end{align*}
  where the implied constants only depend on the embedding
  dimension of~$(X, u)$.  
\end{cor}

\begin{proof}
  We recall that $w_{!*}A$ is the image of the canonical morphism
  $$
  \PH^0(w_!A)\to \PH^0(w_*A)
  $$
  of perverse sheaves. It satisfies $A=w^*w_{!*}A$, and hence the estimates
  $$
  c_{u\circ w}(A)=c_{u\circ w}(w^*w_{!*}A)\ll c_{u\circ
    w,u}(w)c_u(w_{!*}A)=c(u\circ w)c_{u}(w_{!*}A)
  $$
  follow from~(\ref{eq-6pullback}) and Remark~\ref{rem-complex-immersion}. Conversely, Proposition~\ref{easylin} gives the bound
  $$
  c_u(w_{!*}A)\leq c_u(\PH^0(w_!A))+c_u(\PH^0(w_*A)).
  $$
  Then Theorem~\ref{hardlin}, combined with~(\ref{eq-6extzero})
  and~(\ref{eq-6pushforward}) and Remark~\ref{rem-complex-immersion} again,
 gives the estimate
  $$
  c_u(w_{!*}A)\ll c(u)(c_u(w_!A)+c_u(w_*A))\ll c(u)(c(u\circ w)+c(u\circ
  w)^2c(u)) c_{u\circ w}(A), 
  $$ as we wanted to show. 
\end{proof}
  
\subsection{Nearby and vanishing cycles}

In this section, we prove the continuity of the functors of nearby and
vanishing cycles. Let $S$ be the spectrum of a strictly Henselian discrete valuation ring,
with special point~$\sigma$ and generic point~$\eta$. Let $\bar \eta$ be
a geometric generic point above $\eta$. Let $f \colon X\to S$ be a quasi-projective morphism and $u \colon X\to \Pp^n_S$ a locally closed embedding. We denote by $X_\sigma$ and $X_\eta$ the fibers of $f$ over $\sigma$ and $\eta$ respectively. We also consider the induced embeddings $u_\sigma \colon X_\sigma\to \Pp^n$ and $u_\eta \colon X_\eta\to \Pp^n$. For an object~$A$ of~$\Der(X)$, we denote by $A_\sigma$ and
$A_\eta$ the restrictions of $A$ to~$X_\sigma$ and $X_\eta$ respectively. We denote by $\Psi_f$ (\resp $\Phi_f$) the functor of nearby cycles
(\resp of vanishing cycles) from $\Der(X)$ to
$\Der(X_\sigma)$.

\begin{theo}
\label{theo-vancycles}
Let~$X$, $S$ and~$f$ be as above. For each object $A\in \Der(X)$, the following estimates hold: 
\begin{align*}
  c_{u_\sigma}(\Psi_f(A))&\ll c(u_\sigma)c_{u_\eta}(A_\eta),\\
  c_{u_\sigma}(\Phi_f(A))&\ll c(u_\sigma)c_{u_\eta}(A_\eta)+c_{u_\sigma}(A_\sigma).
\end{align*}
\end{theo}

Before starting the proof, recall the following compatibilities of the nearby cycle functor with pushforwards and pullbacks. Let $h : X'\to X$ be an $S$-morphism, and $f'= f\circ h$. Recall from
\cite[Exp.\,XIII (2.1.7.1)]{SGA7} that, if $h$ is proper, there is a canonical isomorphism
\begin{equation}
\label{eqn-vancyles-prop}
\Psi_{f} h_* \to h_* \Psi_{f'}
\end{equation}
by proper base change.  Moreover, if~$h$ is smooth then
by~\cite[Exp.\,XIII (2.1.7.2)]{SGA7}, there is a canonical
isomorphism
\begin{equation}
\label{eqn-vancyles-smooth}
h^* \Psi_{f}\to \Psi_{f'} h^*.
\end{equation}

\begin{proof}
  By definition of the vanishing cycles, there is a distinguished
  triangle
\[
A_\sigma \to \Psi_f(A)\to \Phi_f(A),
\]
and hence the inequality $c_{u_\sigma}(\Phi_f(A))\leq c_{u_\sigma}(A_\sigma)+c_{u_\sigma}(\Psi_f(A))$ holds by Proposition~\ref{easylin}. It is thus enough to prove the first inequality.

We first assume that $f$ is proper, so that $u$ is a closed immersion. Using (\ref{eqn-vancyles-prop}) with $h=u$, we can replace $X$ by $\Pp^m_S$  and $A$ by $u_!(A)$. 

From (\ref{eqn-vancyles-prop}), for every~$i\in\Zz$ we get the equality
\begin{equation}
\label{eqn-vancyles}
h^i(X_{\bar \eta}, A_{\bar \eta})=h^i(X_\sigma, \Psi_f(A)).
\end{equation}

By definition, we have
\[ c_{u_\sigma}(\Psi_f(A))=\max_{0\leq m\leq n} \sum_{i \in \Zz} h^i(\Pp^n_{k'},
  l_{\mtrx{a}_m}^*(A_\eta)),
\]
where $\mtrx{a}_m$ is a geometric generic point of $M^{n+1,m+1}_{k(\sigma)}$
defined over an algebraically closed field $k'$, and similarly for
$c_{u_\eta}(A_\eta)$.  As in the proof of Lemma \ref{generic-dual}, we
view $l_{\mtrx{a}_m}$ as the composition of the map
$s\colon x \mapsto (x,\mtrx{a}_m)$ from $\Pp^m$ to $\Pp^m \times M^{m+1,n+1}$
with the matrix multiplication map
$p\colon \Pp^m \times M^{m+1,n+1} \to \Pp^n$. We still denote by $p$ the smooth map \hbox{$\Pp^m_S \times_S M^{m+1,n+1}_S \to \Pp^n_S$}. By (\ref{eqn-vancyles-smooth}), there is an isomorphism 
\begin{equation}
\label{eqn-ppsi}
p^*\Psi_{f}\to \Psi_{f} p^*, 
\end{equation}
where we still denote by $f$ the natural morphism $\Pp^m_S \times_S M^{m+1,n+1}_S\to S$.

Let $S'$ denote the strict localization of $M^{m+1,n+1}_S$ at the point
$\mtrx{a}_m$ and let \hbox{$s\colon S'\to M^{m+1,n+1}_S$} be the
localization map. The scheme~$S'$ is the spectrum of a strictly
Henselian discrete valuation ring, with special point~$\sigma'$ and
generic point~$\eta'$, which are mapped to $\sigma$ and $\eta$
respectively by the canonical morphism $s'\colon S'\to S$. The situation
is similar to the discussion preceding \cite[Lem.\,3.3]{sga4h}. Set
$\tilde k=k(\bar \eta)\times_{k(\eta)}k(\eta')$, which is the fraction
field of the strict localization of $M^{m+1,n+1}_{\bar S}$, where
$\bar S$ is the normalization of $S$ in $\bar \eta$. Then
$\Gal(\bar \eta/\eta)=\Gal(\tilde k/\eta')$ holds. Let $\bar \eta'$ the
spectrum of an algebraic closure of $\tilde k$ and
$P=\Gal(\bar\eta'/\tilde k)$, which is a pro-$p$ group for $p$ the
characteristic exponent of $k$.

Let $f'\colon \Pp^n_{S'}\to S'$ be the strict localization at $\mtrx{a}_m$ of the morphism 
\[
\Pp^m_S \times_S M^{m+1,n+1}_S\longrightarrow M^{m+1,n+1}_S. 
\] We have the following cartesian diagram:
 \[
  \begin{tikzcd}
    \Pp^n_{S'} \arrow[r,"f'"] \arrow[d,"s"]
    & S'\arrow[d,"s"]\arrow[dr,"s'"]&\\
    \Pp^n_{S}\times_S M^{m+1,n+1}_S \arrow[r]\arrow[rr,bend right,"f" description] &M^{m+1,n+1}_S\arrow[r]&S.
  \end{tikzcd}
 \]
As recalled in (\ref{eqn-vancyles-smooth}), there is an isomorphism
\begin{equation}
\label{eqn-locam}
s^*\Psi_{f}\to \Psi_{s'\circ f'} s^*.
\end{equation}  

By \cite[Th.\,Finitude, Lem.\,3.3]{sga4h}, for every object $B$ of $\Der(\Pp^n_{S'})$ and $i\in \Zz$, the equality 
\begin{equation}
\label{eqn-Pinv}
\Psi_{s'\circ f'} (B)=\Psi_{f'} (B)^P
\end{equation} holds. Since $P$ is a pro-$p$ group, taking the $P$-fixed part is an exact functor, hence the inequality 
\begin{equation}
\label{eqn-ineqP}
h^i(\Pp^n_{k'},\Psi_{f'} (B)^P)\leq h^i(\Pp^n_{k'},\Psi_{f'} (B)).
\end{equation}
Combining equations (\ref{eqn-ppsi}), (\ref{eqn-locam}) and (\ref{eqn-Pinv}), we find
\begin{equation}
\label{eqn-loc}
l_{\mtrx{a}_m}^*\Psi_f(A)=s^*p^*\Psi_f(A)=s^*\Psi_f(p^*A)=\Psi_{s'\circ f'}(s^*p^*A)=\Psi_{f'}(l_{\mtrx{a}_m}^*A)^P.
\end{equation}
Hence, by (\ref{eqn-ineqP}) applied to $B=s^*p^*A$, we get 
\begin{equation}
\label{eqn-ineqPA}
h^i(\Pp^n_{k'},l_{\mtrx{a}_m}^*\Psi_f(A))\leq h^i(\Pp^n_{k'},\Psi_{f'}(l_{\mtrx{a}_m}^*A)).
\end{equation}
By (\ref{eqn-vancyles}) applied to $f'$, we have
\begin{equation}
\label{eqn-vancyclesfp}
h^i(\Pp^n_{k'},\Psi_{f'}(l_{\mtrx{a}_m}^*A))=h^i(\Pp^n_{\bar\eta'},(l_{\mtrx{a}_m}^*A)_\eta)=h^i(\Pp^n_{\bar\eta'},l_{\mtrx{a}_m}^*(A_{\bar \eta'})).
\end{equation}
Combining (\ref{eqn-ineqPA}) and  (\ref{eqn-vancyclesfp}), we find the sought-after inequality 
\[
c_{u_\sigma}(\Psi_f(A))\leq c_{u_\eta}(A_\eta). 
\]

We now come back to the general case, i.e. we do not assume that $f$ is proper. We factor the immersion
$u \colon X\to \Pp^n_S$ as $i\circ j$, where $i$ is a closed immersion
and $j$ is an open immersion.

By the first step, applied to $j_!A$ and the morphism $\bar f$ given
by the composition of $i$ and the projection $\Pp^n_S\to S$, we get
$c_{i_\eta}(j_!(A)_\eta)\leq c_{i_\sigma}(\Psi_{\bar f}j_!A)$. By
(\ref{eqn-vancyles-smooth}), the object $j_\sigma^*\Psi_{\bar f}j_!A$ is
isomorphic to $\Psi_f(A)$, and hence we get 
\begin{align*}
  c_{u_\sigma}(\Psi_f(A))& \ll c_{{u_\sigma},{i_\sigma}}(j_\sigma)c_{i_\sigma}(\Psi_{\bar
    f}j_!A)\\
  &\leq c_{{u_\sigma},{i_\sigma}}(j_\sigma)c_{i_\eta}((j_!A)_\eta)
  \\
  &=c_{{u_\sigma},{i_\sigma}}(j_\sigma)c_{u_\eta}(A_\eta)=c(u_\sigma)c_{u_\eta}(A_\eta)
\end{align*} by~(\ref{eq-6pullback}) and Remark
\ref{rem-complex-immersion}. This finishes the proof. 
\end{proof}

\begin{cor}
\label{cor-vancycles}
Let $k$ be an algebraically closed field, $f\colon (X,u)\to (S,v)$ a flat
morphism of quasi-projective varieties over $k$, with $S$ smooth and
irreducible of dimension $1$. Let $\sigma$ be a closed point of
$S$. Let $S_{(\sigma)}$ be the strict localization of $S$ at $\sigma$.
Denote by $\Psi_f$ and~$\Phi_f$ the nearby and vanishing cycles
functors of the base change of $f$ to $S_{(\sigma)}$, which we view as
functors from $\Der(X)$ to $\Der(X_\sigma)$. For any $A\in \Der(X)$, the following estimates hold:
\begin{align*}
  c_{u_\sigma}(\Psi_f(A))&\ll c_{u,v}(f)^2c_u(A),
  \\
  c_{u_\sigma}(\Phi_f(A)) &\ll c_{u,v}(f)^2c_u(A).
\end{align*}
\end{cor}

\begin{proof}
  Let $\Gamma \subset X\times_k S$ be the graph of $f$. We denote by
  $h \colon X\to \Gamma$ the canonical $S$\nobreakdash-morphism, and by
  $f'\colon \Gamma\to S$ the projection. Abusively, we will also write $u \colon \Gamma \to \Pp^n_S$ for the canonical embedding induced by
  $u\colon X\to \Pp^n_k$.

 Since $h$ is an isomorphism, the base change morphism
  \hbox{$\Psi_{f'} h_{\eta*}\to h_{\sigma*} \Psi_f$} is an
  isomorphism as well. Moreover, $h_\eta$ and $h_\sigma$ are isomorphisms and the equality $c_{u_\sigma}(B)=c_{u_\sigma}(h_{\sigma*}B)$ holds for any object $B$ of $\Der(X_\sigma)$.  By
  Theorem \ref{theo-vancycles} applied to the strict localization at~$\sigma$ of~$f'\colon \Gamma \to S$ and $h_*A$, the following estimates hold: 
\begin{align*}
  c_{u_\sigma}(\Psi_{f'}(h_*A))&\ll c(u_\sigma)c_{u_\eta}(h_*A_\eta),
  \\
  c_{u_\sigma}(\Phi_{f'}(h_*A))&\ll
  c(u_\sigma)c_{u_\eta}(h_*A_\eta)+c_{u_\sigma}(h_*A_\sigma).
\end{align*}

Using the base change isomorphism previously quoted, this implies that
\begin{align*}
  c_{u_\sigma}(\Psi_{f}(A))&\ll c(u_\sigma)c_{u_\eta}(A_\eta),
  \\
  c_{u_\sigma}(\Phi_{f}(A))&\ll
  c(u_\sigma)c_{u_\eta}(A_\eta)+c_{u_\sigma}(A_\sigma).
\end{align*}

We will now prove the estimates
\begin{align*}
c_{u_\sigma}(A_\sigma)&\ll c_{u,v}(f)c_u(A),\\
c(u_\sigma)&\ll c_{u,v}(f),\\
c_{u_\eta}(A_\eta)&\ll c_{u,v}(f)c_u(A),
\end{align*}
which will conclude the proof.

First, let $\delta_\sigma\in \Der(S)$ be the rank-one skyscraper sheaf
supported at $\sigma\in S$. We have $c_v(\delta_\sigma)=1$ and
$f^*\delta_\sigma$ is the constant sheaf supported on $X_\sigma$, so
that by Theorem \ref{pushpull}, we obtain
$c_u(f^*\delta_\sigma)\ll c_{u,v}(f)c_v(\delta_\sigma)$. Hence, the estimate
$$
c(u_\sigma)=c_u(f^*\delta_\sigma)\ll c_{u,v}(f)
$$ holds. This proves the
second inequality.

Moreover, we have $A_\sigma=A\otimes \Ql_{\vert X_\sigma}$, and hence the estimate
$$
c_{u_\sigma}(A_\sigma)\ll c_u(A)c(u_\sigma)
$$ holds by Theorem
\ref{pushpull}. Combined with the second inequality, this yields the
first one.

For the third inequality, let $\mtrx{b}$ be a geometric generic point of
$M^{n_S,n_S+1}$ such that the intersection of $v(S)$ and the image of
$l_\mtrx{b}$ in $\Pp^{n_S}_{k'}$ is finite. This intersection consists of finitely many geometric generic points $\bar \eta_1,\dots,\bar \eta_n$ of $S$. The complex $ f^*v^*l_{\mtrx{b}*}\Ql$ is then the constant sheaf supported on $f^{-1}(\{\bar \eta_1,\dots,\bar \eta_n\})$, hence the estimate 
\[
c_{u_\eta}(A_\eta)\leq c_{u}(A\otimes f^*v^*l_{\mtrx{b}*}\Ql).
\]
From Theorem \ref{pushpull} and the definition of $c_{u,v}(f)$, we deduce
\[
  c_{u}(A\otimes f^*v^*l_{\mtrx{b}*}\Ql) \ll
  c_u(A)c_{u}(f^*v^*l_{\mtrx{b}*}\Ql)\leq c_u(A)c_{u,v}(f),
\]
which ends the proof.
\end{proof}

\subsection{Uniformity in families}

In Theorem~\ref{pushpull}, the factors $c(u)$ or $c_{u,v}(f)$
appear. The next proposition ensures that these are bounded in
algebraic families; for all practical
purposes, they will thus behave as constants in applications.

\begin{prop}\label{uniformity}
  Let $S$ be a noetherian scheme in which $\ell$ is invertible. Let
  $X$ be a scheme of finite type over $S$, and let $u\colon X \to \Pp^n_S$
  be a locally closed immersion of schemes over $S$, so that for each geometric points $s$ of $S$, the pair $(X_s,u_s)$ is a quasi-projective variety over the algebraically closed field~$k(s)$. Then: 
  \begin{enumerate}
  \item\label{uniformity:item1} There exists a constant $M_1$, depending only on $(X,S,u)$, such that the inequality~$c(u_s)\leq M_1$ holds for all~$s$. 
  \item\label{uniformity:item2} For any object $A$ of $\Der(X)$, there exists a constant~$M_2$, depending only on $(X,S,u,A)$, such that the inequality $c_{u_s}(A_s)\leq M_2$ holds for
    all~$s$.
  \item\label{uniformity:item3} Let $Y$ be a scheme over $S$ and let $v\colon Y \to \Pp^{n'}_S$ be a
    locally closed immersion of schemes over $S$. Let $f\colon X \to Y$ be
    a morphism of schemes over $S$. Then there exists a constant
    $M_3$, depending only on $(X,S,u,Y,v,f)$ such that the inequality~$c_{u_s,v_s}(f_s)\leq M_3$ holds for all~$s$.
  \end{enumerate}
\end{prop}

\begin{proof}
  By definition of $c(u_s)$, part~\eqref{uniformity:item1} will follow from~\eqref{uniformity:item2} applied to~$A=\Ql$. 
 
  We now prove~\eqref{uniformity:item2}.  Let $0\leq m\leq n$ be an integer. Let
  $r\colon M_{S}^{n+1,m+1}\times_S\Pp^m_S \to \Pp^n_S$ be the morphism defined by
  $r(\mtrx{a},x) = l_{\mtrx{a}}(x)$.  Consider the diagram
  \[
  \xymatrix{
  X \ar[r]^{u} &  \Pp^n_S & M^{n+1,m+1}_S \times_S  \Pp^m_S \ar[l]_-{r} \ar[d]^{\mathrm{pr}_1} \\
  & & M^{n+1,m+1}_S. 
 }
  \]
 
 Let~$s$ be a geometric point of~$S$. Let $\mtrx{a}_m$ be a geometric
  generic point of the fiber of~$M^{n+1,m+1}_S$ over~$s$. Let~$i\in\Zz$.
  By the proper base change theorem, the equality 
  $$
  \sum_{i \in \Zz} h^i(\Pp^m_{k'}, l_{\mtrx{a}_m}^* u_{s!} A_s)=
  \sum_{i\in\Zz} \dim \mathcal H^i( pr_{1*} r^* u_! A)_{\mtrx{a}_m}
  $$
holds. Since $u$ is a locally closed immersion and~$r$ and~$\mathrm{pr}_1$ are morphisms of
  finite type, the complex $\mathrm{pr}_{1*}r^*u_! A$ is an object of $\Der(M^{n+1,m+1}_S)$, so by constructibility, the sum above is
  bounded as~$s$ varies.
  
    For \eqref{uniformity:item3}, we argue in a similar way with
  $M^{n+1,m+1}_S \times_S M^{n'+1,m'+1}_S$.
\end{proof}

\begin{remark}\label{rem:uniformpolynomials}
For instance, if $f\colon \Aa^n\to \Aa^m$ is a morphism given by~$m$
polynomials in $n$ variables of degree at most~$d$, then using the
universal family of such polynomials it follows that there exists a constant
$b(n,m,d)$ such that the inequality~$c_{u,v}(f)\leq b(n,m,d)$ holds for all~$f$ (with the
standard embeddings $u\colon \Aa^n\to \Pp^n$ and
$v\colon \Aa^m\to \Pp^m$). We can in fact make some cases fully explicit by appealing to results
of Katz~\cite{katzbetti}. Since this can be useful in some
applications, we spell out one such case in the next proposition. 
\end{remark}

For non-negative integers $N$, $r$ and $d$, set
$$
B(N,r,d)= 6 \cdot 2^r \cdot (3 + rd)^{N+1}.
$$

\begin{prop}\label{pr-explicit-uniform}
  Let~$r$, $d$, $n_1$ and~$n_2$ be non-negative integers. Let $X$ be a
  closed subvariety of $\Aa^{n_1}$ over~$k$ defined by $r$ equations
  of degree at most $d$. Let $Y$ be any subvariety of
  $\Aa^{n_2}_k$. Let $u$ and $v$ be the natural immersions
  of~$\Aa^{n_1}$ and~$\Aa^{n_2}$ in $\Pp^{n_1}$ and~$\Pp^{n_2}$
  respectively.  For any morphism $f\colon X \to Y$ defined by $n_2$
  polynomials of degree at most~$d$, the following inequality holds:
  $$
  c_{u,v}(f) \leq B(n_1,n_2+r,d).
  $$
\end{prop}

\begin{proof}
  By definition, the complexity $c_{u,v}(f)$ is the maximum over all 
  integers \hbox{$m_1\leq n_1$} and~$m_2\leq n_2$ of the sum of the compactly
  supported Betti numbers (with $\Ql$\nobreakdash-coefficients) of the
  intersection of $X$ with a generic linear subspace of dimension~$m_1$ in~$\Pp^{n_1}$ and with the pullback by~$f$ of a generic
  linear subspace of dimension $m_2$ in~$\Pp^{n_2}$. The intersection of the first linear subspace with $\Aa^{n_1}$ is an
  affine space of dimension $n_1$. Its intersection with~$X$ is
  defined by $r$ equations of degree at most~$d$, and the pullback of
  a linear subspace is defined by $n_2-m_2$ linear combinations of the
  polynomials defining~$f$, which are hence also of degree at most~$d$.
 Thanks to~\cite[Cor.\,of Th.\,1, p.\,34]{katzbetti}, we deduce the inequalities
  $$
  \sum_{i\in\Zz}h^i_c(X, u^* l_{\mtrx{a}_{m_1} }^*\Ql \otimes f^* v^*
  l_{\mtrx{b}_{m_2}}^* \Ql )\leq B(m_1,n_2-m_2+r,d)\leq B(n_1,n_2+r,d)
  $$
  (with notation as in Definition~\ref{def-cuv}) for all $m_1$ and $m_2$,
  and the result follows.
\end{proof}

\begin{remark}
  In some cases, Katz's remark in~\cite[p.\,43]{katzbetti} 
  leads to better estimates.
\end{remark}

\subsection{The open set of lissity of a complex}

Another invariant governed by the complexity is the degree of the
complement of a dense open set where an object of the derived category is lisse
(i.e., where all its cohomology sheaves are lisse). Below, by \emph{degree} of a subvariety $Z$ of projective space, we mean the sum of the degrees of its irreducible components (which may have different dimensions).

\begin{theo}\label{singloc}
  Let $(X,u)$ be an irreducible quasi-projective variety over~$k$. Let $A$ be an object of $\Der(X)$. Let~$Z$ be the complement of the maximal
  open subset where~$X$ is smooth and~$A$ is lisse.
  Then the estimate
  \[
    \deg (u(Z)) \ll (3 + s)c(u) c_u(A)
  \]
 holds, where the degrees are computed in the projective space target of~$u$, and~$s$
  is the degree of the codimension $1$ part of the singular locus of $X$.
\end{theo}

\begin{proof}
  We first assume that~$A$ is an irreducible perverse object. Let~$n$ be the
  embedding dimension of~$(X,u)$. Let~$S$ be the support of~$A$.

  If $S$ has codimension $m \geq 1$, then $A$ is lisse (being zero)
  outside~$S$; we will show that the inequality $c_u(A) \geq \deg(u(S))$ holds. Indeed, let $\mtrx{a}_m$ be a
  geometric generic point of~$M^{m+1,n+1}_k$, so that
  $l_{\mtrx{a}_m}\colon \Pp^m \to \Pp^n$ is a generic linear embedding that
  intersects~$S$ in $\deg(u(S))$ general points. The complex
  $l_{\mtrx{a}_m}^* u_!  A$ is then supported on $\deg(u(S))$ points and 
  non-zero on each of them, so the sum of its Betti numbers
  is~$\geq \deg(u(S))$, which gives the stated bound.

  Now assume that~$S$ is equal to $X$. In case $A$ is lisse on the non-singular locus of~$X$, the subset $Z$ is the singular locus of~$X$, and the statement trivially holds since the left-hand side of the estimate is independent of $A$. Otherwise, $Z$ has codimension $1$ in $X$. Indeed, were its codimension bigger, $A$ would be lisse on the non-singular locus of $X$ by purity of the branching locus \cite[Exp.\,X, Cor.\,3.3]{SGA1}. Let~$d$ be the degree of~$u(Z)$. Set \hbox{$m=1+\mathrm{codim}(u(X))$,} and let $\mtrx{a}_m$ be a geometric generic point of~$M^{m+1,n+1}_k$, so that \hbox{$l_{\uple{a}_m}\colon \Pp^m \to \Pp^n$} is a generic linear embedding that
  intersects the support of $A$ in a curve $C$, of which $d$ points lie in the
  singular locus~$Z$. Let~$j$ be the embedding of~$C$ in~$X$.  Moreover, let
  $\tilde{C}$ be the normalization of~$C$ and
  let~\hbox{$\pi\colon \tilde{C} \to C$} be the canonical morphism.

  We can view the pullback~$j^*A$ of~$A$ to~$C$ as obtained first by pullback
  along a smooth morphism $M^{m+1,n+1} \times \Pp^m \to \Pp^n$, followed by
  pullback to the geometric generic fiber.  Thus any property of~$A$ that is
  preserved by smooth pullback and by restriction to the geometric generic fiber
  will be preserved by pullback to $C$.  In particular, by Lemma
  \ref{generic-dual} a shift of $j^*A$ is perverse on~$C$, and because an
  irreducible perverse sheaf on a curve is a shift of a middle-extension sheaf,
  the pullback~$j^*A$ is a shift of a middle extension sheaf~$\mathcal F$. In
  addition, $\mathcal F$ is not lisse on at least $d$ points, since neither
  smooth pullback nor restriction to the geometric generic fiber can make a
  singular point disappear.

  Let $\tilde{\mathcal{F}}$ be the middle extension
  of~$\pi^*\mathcal{F}$ from the maximal open set where it is lisse;
  then $\mathcal F $ is canonically isomorphic to
  $\pi_* \tilde{\mathcal F}$, since both are middle-extension sheaves
  on $C$ that are isomorphic on a dense open set. Let $d'$ be the number of points where $\mathcal F$ is not lisse and that are not contained in the singular support of $C$. We have $d\leq d'+s$ and $\tilde{\mathcal{F}}$ is not lisse in at least $d'$ points. 

 From the equalities
  $$
  \sum_{i\in \Zz} h^i(\Pp^n,l_{\uple{a}_m}^*u_!A)=
  \sum_{i\in\Zz}h^i_c(C,j^*A)=\sum_{i\in\Zz} h^i_c(C,\mathcal{F})
  =\sum_{i\in\Zz}h^i_c(\tilde{C},\tilde{\mathcal{F}}),
  $$
we get
  $$
  -\chi_c(\tilde{C},\tilde{\mathcal{F}})\leq \sum_{i\in\Zz}
  h^i_c(\Pp^n,l_{\uple{a}_m}^*u_!A)\leq c_u(A).
  $$

Let $r$ be the ``generic rank of $A$" in the sense of Proposition~\ref{rm-zero}. By Proposition~\ref{rm-zero}, the inequality $r\leq c_u(A)$ holds. Since $l_{\mtrx{a}_m}$ is a generic linear embedding, $r$ is also the ``generic rank of $j^*A$" in the sense of Proposition~\ref{rm-zero}. Since $j^*A$ is a shifted middle-extension sheaf $\mathcal{F}$, $r$ is also the generic rank of $\mathcal{F}$, \emph{i.e.} the rank of a lisse sheaf of which $\mathcal{F}$ is the middle extension. Since $\tilde{\mathcal{F}}$ be the middle extension
  of~$\pi^*\mathcal{F}$ from the maximal open set where it is lisse, $r$ is also the generic rank of~$\tilde{\mathcal{F}}$. Using the Grothen\-dieck-Ogg-Shafarevich formula,
  we get $ \chi_c(\tilde{C},\tilde{\mathcal{F}})\leq d'-2r$ . From this we
  derive the inequalities
  $$
  d'\leq 2r+c_u(A)\leq 3c_u(A).
  $$
  We finally get
  $$
  d\leq d'+s\leq (3+s)c_u(A),
  $$
since $c_u(A)\geq 1$. 

  Now consider the general case.  For~$i\in\Zz$, let $n_i$ be the
  length of $\PH^u(A)$ and let~$(A_{i,j})_{1\leq j\leq n_i}$ be the
  Jordan--Hölder factors of $\PH^i(A)$, repeated with multiplicity.
  The object~$A$ is lisse on the intersection over $i$ and $j$ of the
  maximal open sets~$U_{i,j}$ where $A_{i,j}$ is lisse.  (Note that
  all the irreducible perverse factors of a lisse sheaf are lisse, so
  that if~$A_{i,j}$ is not lisse at a point, then neither is $A$).

  By the case of perverse sheaves, the complement of~$U_{i,j}$ is the
  union of subvarieties with total degree $\leq
  (3+s)c_u(A_{i,j})$. Thus, we obtain subvarieties $Z_i$ with
  $$
  \sum_{i=1}^m \deg (Z_i) \leq (3 + s)\sum_{i\in\Zz}\sum_{j=1}^{n_i}
  c_u(A_{i,j})\ll (3+s)c(u) c_u(A)
  $$
  by Theorem \ref{hardlin}.
\end{proof}

\subsection{Complexity of the cohomology sheaves of a complex}

In this section, we consider an analogue of Theorem~\ref{hardlin}
where the Jordan--Hölder components of the perverse cohomology sheaves
are replaced by the ordinary cohomology sheaves. This is a case where
we will only be able to prove ``continuity'' abstractly, without
explicit estimates. Precisely, we have:

\begin{prop}\label{reductiontosheaf}
  The exist a function $N\colon \Rr^+\times \Rr^+\to \Rr^+$ with the
  following property: for any quasi-projective algebraic variety
  $(X,u)$ over~$k$ with embedding dimension~$n$, and for any
  object~$A$ of $\Der(X)$, the following inequality holds: 
  \[
    \sum_{i\in\Zz} c_u( \mathcal H^i (A) ) \leq N(n,c_u(A)).
  \]
\end{prop}

We begin with two lemmas.

\begin{lem}\label{chow}
  For any non-negative integers $d$ and $n$, there exists
  $C(d,n)\geq 0$ such that, for any algebraically closed field $k$,
  for any prime $\ell$ invertible in $k$, and for any closed immersion
  $u\colon X \to \Pp^n$ whose image is a union of irreducible
  subvarieties of total degree~$d$, the inequalities
  \begin{gather*}
    c(u)\leq C(d,n)\\
    \deg(X_s)\leq C(d,n)
  \end{gather*}
  hold, where $X_s$ is the singular locus of $X$.
\end{lem}

\begin{proof}
  The theory of Chow varieties (see, e.g.,~\cite{kollar}) provides a
  quasi-projective scheme~$\mathrm{Chow}_{d,n}$ over~$k$ whose points
  ``are'' the closed immersions $i\colon X \to \Pp^n$ whose image is a
  union of irreducible subvarieties of total degree $d$. The first
  inequality then follows from Proposition
  \ref{uniformity}\,\eqref{uniformity:item1}, and the second is a
  consequence of the constructibility of the function that maps a point
  of $\mathrm{Chow}_{d,n}$ to the degree of the singular locus of the
  corresponding quasi-projective variety.
\end{proof}

\begin{lem}\label{monster}
  For any non-negative integers $d$ and $n$, there exists
  $C(d,n)\geq 0$ with the following property: for any algebraically closed field~$k$, for any prime~$\ell$
  invertible in~$k$, for any quasi-projective variety $(X,u)$ over~$k$
  with embedding dimension~$n$, and for any object $A$ of~$\Der(X)$
  with $c_u(A)\leq d$, there exists a stratification
  $$
  X_{n+1} \subseteq X_n \subseteq X_{n-1} \subseteq \dots X_1
  \subseteq X
  $$
  of $X$ such that
  \begin{itemize}
  \item The subvariety $X_1$ is the support of~$A$.
  \item For all~$i$, the subvariety $X_i - X_{i-1}$ is smooth.
  \item For all~$i$, the object $A$ is lisse on
    $X_i - X_{i-1}$.
  \item For all $i$, the subvariety $X_i$ is a union of at most
    $C(d,n)$ subvarieties of degree~\hbox{$\leq C(d,n)$}.
  \item For all~$i$, the inequality $c(u_i) \leq C(d,n)$ holds, where
    $u_i\colon X_i - X_{i+1} \to \Pp^n$ is the natural immersion.
  \end{itemize}
\end{lem}

\begin{proof}
The equality $c_u(A)=c(u_!A)$ holds by definition, so we may replace $X$ with
  the projective space~$\Pp^n$ (and~$u$ with the identity), and~$A$ with
  $u_!A$, provided we describe the subvarieties~$X_i$ as forming a
  stratification of the support of~$A$.
  
  Indeed, we define $X_1$ to be the support of $A$, and then
  inductively $X_{i+1}$ to be the complement in~$X_i$ of the maximal
  smooth open subset of $X_i$ on which $A$ is lisse. On noting the inequality
  $\dim(X_{i+1})<\dim(X_i)$, this provides a stratification of the
  support of~$A$ with at most $n+1$ non-empty subvarieties. We denote
  by~$v_i$ the immersion $X_i\to\Pp^n$.

  We now prove by induction on~$i$ with $1\leq i\leq n+1$ that~$X_i$
  is a union of varieties whose total degree is bounded only in terms of
  $(d,n)$, and that $c(v_i)$ is bounded only in terms of $(d,n)$.
  
  Since $X_1$ is the support of $A$, either $X_1$ is equal to $\Pp^n$
  or it is the complement of the maximum open subset on which $A$ is
  lisse (in fact, zero). In the second case, Lemma~\ref{singloc} shows
  that that $X_1$ is a union of varieties of degree at most
  $3 c_u(A)$. In the first case, the same inequality holds
  trivially. Then by the first inequality in Lemma~\ref{chow}, the
  complexity $c(v_1)$ has a bound in terms of $c_u(A)$ and $n$, and hence
  in terms of $(d,n)$. This establishes the base case of the
  induction.

  Assume that the induction assumption holds for some~$i$ with
  $1\leq i\leq n$. Then, by~(\ref{eq-6pullback}), we obtain
  $c_{v_i}(v_i^*A)\leq c_{v_i,\mathrm{Id}}(v_i)c(A)=c(v_i)c(A)$. By
  the definition of~$X_{i+1}$, applying Lemma~\ref{singloc} and the
  second inequality of Lemma \ref{chow}, we deduce that~$X_{i+1}$ is a
  union of varieties of total degree bounded only in terms of $(d,n)$. 
  The first inequality of Lemma~\ref{chow} applied once more shows
  that $c(v_{i+1})$ is bounded only in terms of $(d,n)$. This
  completes the induction.

  It only remains to bound the complexity of the immersions
  $u_i$. The excision triangle gives
  $$
  u_{i!}\Ql\longrightarrow v_{i!}\Ql\longrightarrow v_{(i+1)!}\Ql
  $$
  in $\Der(\Pp^n)$, and hence the inequality
  $$
  c(u_i) = c( u_{i!} \Ql ) \leq c(v_{i!} \Ql) + c(v_{(i+1)!}
  \Ql)=c(v_i)+c(v_{i+1})
  $$
  follows by Proposition~\ref{easylin}\,\eqref{easylin:item1}. 
\end{proof}

\begin{proof}[Proof of Proposition~\ref{reductiontosheaf}]
  We apply the previous lemma with $d=c_u(A)$, and we let $C(d,n)\geq 0$ denote the corresponding number and 
  $$
  X_{n+1} \subseteq X_n \subseteq X_{n-1} \subseteq \dots X_1
  \subseteq X
  $$
 a stratification of $X$ with the properties of the lemma.
 
  Applying excision repeatedly and Proposition~\ref{easylin}\,\eqref{easylin:item1}, we get
  \begin{align*}
    \sum_{i\in\Zz} c_u( \mathcal H^j (A) )
    &
      = \sum_{i\in\Zz} c( u_!\mathcal H^j (A))
    \\
    &\leq \sum_{i\in\Zz}
      \sum_{j=1}^{n+1} c( u_{j!} u_j^* u_!\mathcal H^i(A) )
      =\sum_{j=1}^{n+1} \sum_{i\in\Zz}
      c( u_{j!} u_j^* u_!\mathcal H^i(A) ).
  \end{align*}
  Applying~(\ref{eq-6extzero}) to~$f=u_j$ viewed as a morphism from
  $(X_j-X_{j+1},u_j)$ to $(\Pp^n,\mathrm{Id})$, we have
  $$
  c( u_{j!} u_j^* u_!\mathcal H^i(A) ) \ll
  c_{u_j,\mathrm{Id}}(u_j)c_{u_j}(u_j^* u_!\mathcal H^i(A) )=
  c(u_j)c_{u_j}(u_j^* u_!\mathcal H^i(A)) 
  $$
  for all~$i$ and~$j$. By construction, $u_j^*u_{j!}\mathcal{H}^i(A)$ is lisse on the smooth scheme \hbox{$X_j-X_{j+1}$}; since the perverse homology sheaves and
  the usual homology sheaves of a lisse sheaf on a smooth scheme agree, the estimates
  $$
  \sum_{i\in\Zz} c( u_{j!} u_j^* u_!\mathcal H^i(A) ) \ll c(u_j)
  \sum_{i\in\Zz} c_{u_j}( \PH^i(u_j^* u_!A) ) \ll
  c(u_j)^2c_{u_j}(u_j^*u_{!}A)
  $$ hold by Theorem~\ref{hardlin}. Applying~(\ref{eq-6pullback}), we get
  $$
  c_{u_j}(u_j^*u_{!}A)\ll c(u_j)c(u_!A)=c(u_j)c_u(A),
  $$
  and therefore
  $$
  \sum_{i\in\Zz} c_u( \mathcal H^j (A) ) \ll
  c_u(A)\sum_{j=1}^{n+1}c(u_j)^3
  $$
  is bounded only in terms of~$(n,d)$, since $c(u_j)\leq C(d,n)$ for
  all~$j$.
  \end{proof}

\subsection{Generic base change}

We now prove an effective version of Deligne's generic base change
theorem~\cite[Th.\,finitude, 1.9]{sga4h}. This argument is
  due to A. Forey.

\begin{theo}
\label{thm-gen-bc-conductors}
Let $(X,u)$, $(Y,v)$ and $(S,w)$ be quasi-projective algebraic
varieties over~$k$. Let~$f\colon X\to Y$ and $g\colon Y\to S$ be
morphisms.
\par
For any object $A$ of~$\Der(X)$, there exists an integer $N\geq 0$,
depending only on $c_u(A)$ and $(f,g,u,v,w)$, and a dense open set
$U\subset S$ such that
\par
\emph{(i)} The image of the complement of $U$ has degree~$\leq N$.
\par
\emph{(ii)} The object $f_*A$ is of formation compatible with any base
change $S'\to U\subset S$.
\end{theo}

\begin{remark}
  The original generic base change theorem is stated for a
  constructible sheaf of $R$-modules, where $R$ is a Noetherian ring
  satisfying $nR=0$ for some integer~$n$ that is invertible in $S$, and $S$ is
  not supposed to be defined over a field $k$. An additional statement
  of the theorem is the constructibility of $f_*A$ on $Y_U$. In the
  setting of the above statement, it is already known that $f_*A$ is
  constructible, precisely by applying \hbox{\cite[Th.\,finitude, 1.9]{sga4h}}
  to $f\colon X\to Y$ over $k$. However, in order to understand the
  complexity of the complement of the open set $U$ over which $f_*A$
  is of formation compatible with base change, we would need to redo
  the whole proof, following closely Deligne's argument (see also
  \cite[Th.\,9.3.1]{fu_etale_co}).
\end{remark}

In this section, we will often write simply $c(X)$ for the complexity $c(u)$, where $(X,u)$ is a quasi-projective variety. Recall also from Remark \ref{rem-complex-immersion} that if $i \colon Z\to X$ is an immersion, then $c_{u\circ i,u}(i)=c(Z)$. 

Before starting the proof, we state a useful lemma.

\begin{lem}
\label{lem-affine-open-cond}
Let $(X,u)$ be a quasi-projective variety over~$k$ of embedding dimension~$n$. There exists a
finite open cover~$(U_i)$ of $X$ into at most $\dim(X)+1$ affine subvarieties with open
immersions~\hbox{$u_i\colon U_i\to X$} such that~$c_{u\circ u_i,u}(u_i)=c(U_i)$ is bounded in
terms of~\hbox{$c(u)=c(X)$} and $n$.
\end{lem}

\begin{proof}
  The image of $u$ in~$\Pp^n$ can be written as $Z-W$, where $Z$ and $W$ are
  closed subvarieties of $\Pp^n$. As in the proof of Lemma \ref{monster}, by
  Theorem \ref{singloc} the degrees of $Z$ and~$W$ are bounded in terms of
  $c(u)$ and $n$. Let $H$ be a hypersurface of $\Pp^n$ such that $W\subset H$
  but $Z\notin W$. One can choose $W$ of degree at most the degree of~$W$. Then
  set $U_1=Z-H$, which is an affine open subset of $Z-W$, of complexity bounded
  by the degrees of $Z$ and $H$. Since $\dim(Z\cap H)<\dim(Z)$, one can conclude
  by induction: the affine open subsets of $Z\cap H$ that are obtained are restrictions
  of affine open subsets of $\Pp^n$, and hence their restrictions to $Z$ produce affine
  open subsets of $Z-W$, which, together with $U_1$, cover $Z-W$.
\end{proof}

During the proof of Theorem~\ref{thm-gen-bc-conductors}, we will
repeatedly consider subschemes of $X$, $Y$ and $S$. These will always
implicitly be considered with the locally closed embedding inherited
from $X$, $Y$ or $S$. All complexities will be computed with respect
to those implicit embeddings, and hence we will simplify the notation
by dropping the embeding from the complexity. Moreover, whenever we
say that we can shrink~$S$, we mean that we can replace~$S$ by a dense
open subset whose complement has bounded degree in terms of~$c(A)$.

\begin{proof}[Proof of Theorem \ref{thm-gen-bc-conductors}]
  Working successively with each irreducible component of~$S$, there
  is no loss of generality in assuming, as we do from now on, that $S$
  is irreducible.

  \textbf{Step 1.} We first consider the particular case where~$S=Y$, the variety $X$
  is smooth of pure relative dimension~$d$ over~$S$, and~$A$ is a
  lisse sheaf such that~$R^if_!A'$ is lisse for each~$i\in\Zz$. Here,
  $A'$ denotes the ``naive'' dual $A'=\rhom(A,\Ql)$, i.e.\,non-derived sheaf homomorphisms. Then Deligne~\cite[Th.\,finitude, 2.1]{sga4h} proves
  the result with~$U=S$ (this is essentially a consequence of the proper
  base change theorem and Poincaré duality).
 
  \textbf{Step 2.} We now assume that $S=Y$, $X$ is smooth over $S$ of
  pure relative dimension~$d$, and that $A$ is a lisse sheaf. Let
  again $A'=\rhom(A,\Ql)$; by~(\ref{eq-6extzero}) and
  Theorem~\ref{singloc}, there exists an open set $U\subset S$ whose
  complement has degree bounded in terms of $c(A')$, and hence in terms of
  $c(A)$, such that $f_!A'$ is lisse on~$U$. Over the dense open set
  $U$, we are in the situation of Step~1, and hence the result holds
  in this case too.

  \par
  \textbf{Step 3.} We now perform some reductions for the proof of the
  general case. We first observe that the problem is local on $Y$, so
  that we may assume that $Y$ affine. Using a finite affine cover of
  $X$ and excision, we may also assume that $X$ is affine. Note that
  the complexity of the restriction of $A$ to those affine open
  subsets is bounded in terms of~$c(A)$ by~(\ref{eq-6pullback}).


  Up to shrinking $X$ again, we can now factor $f$ into an open immersion followed by a proper morphism $g$ (this is a form of
  Nagata's compactification theorem).

  By proper base change, the result holds for the morphism $g$ with
  $U=S$, and hence it is enough to prove the result when~$f$ is an
  open immersion with dense image. We will then prove the result by
  induction on the relative dimension~$n\geq 0$ of $X$ over $S$.

  For $n=0$, since $X$ and $Y$ are of dimension $0$ over $S$, up to
  shrinking $S$ to a dense open subset (whose complement is of degree
  bounded in terms of $X$ and $Y$), we have $X=Y$. Hence the result
  holds.

  We now consider $n\geq 1$, and assume that the result holds for
  relative dimension up to~$n-1$.

  We will prove below the following sublemma:

  \begin{lem}\label{lm-sub}
    With notation and assumption as before, up to shrinking $S$, there exists a
    dense open subset $Y'\subset Y$ such that the result holds over $Y'$ and
    such that the complement of $Y'$ in $Y$ is finite of bounded degree over
    $S$.
  \end{lem}
  
  \textbf{Step 4.} Under the conditions of Step~3, we further assume
  that $X$ is smooth over $S$ and that $A$ is a lisse sheaf. We can
  replace $Y$ by its closure in the image of~$v$, and hence assume
  that $Y$ is projective over $S$. Using~ Lemma~\ref{lm-sub}, we can
  assume that there exists $Y'\subset Y$ such that the result holds
  for $Y'$ and that $Y-Y'$ is finite over $S$ of bounded degree.  We
  then have the following commutative diagram:
  $$
  \begin{tikzcd}
    & Y'\arrow{d}{j} &
    \\
    X \arrow{r}{f}\arrow{dr}{a} & Y\arrow{d}{b} & Y-Y' \arrow{l}{i}
    \arrow{dl}{c}
    \\
    & S &
  \end{tikzcd}
  $$

  By Step~2, up to shrinking $S$, the object $a_*A$ is of formation
  compatible with any base change. By the choice of~$Y'$ (after again
  shrinking $S$) the object $j^*f_*A$ is of formation compatible with
  any base change. By the proper base change theorem, the same holds for
  $b_* j_!  j^*f_*A$. Now we have a distinguished triangle
\[
  j_! j^*f_*A\longrightarrow f_*A\longrightarrow
  i_*i^*f_*A
\]
Applying $b_*$ to it, we get the distinguished triangle
\[
  b_*j_! j^*f_*A\longrightarrow a_*A\longrightarrow c_*i^*f_*A
\]
Since the first two complexes in this triangle are of formation
compatible with any base change, the same holds for the third one,
namely $c_*i^*f_*A$. Since $c$ is finite, this implies that $i^*f_*A$
also has the same property. Hence in the first triangle, the first and
third complexes are of formation compatible with any base change, and
hence the middle one, namely $f_*A$, also has this property. This
finishes the proof in this case.

\textbf{Step 5.} We now show how to reduce the situation (after the
reduction in Step~3) to that of Step 4. We will also prove below the
following additional sublemma:

\begin{lem}\label{lm-subsub}
  Up to shrinking $S$, and performing a base change along a finite
  surjective radicial morphism and reducing $X$, in an effective way, we
  can find an open dense subset~$V$ of $X$ that is smooth over $S$ and
 has complement of bounded degree.
\end{lem}

By Proposition~\ref{reductiontosheaf}, the complexity of the
cohomology sheaves of $A$ are bounded in terms of $c(A)$; hence, up to
replacing $A$ by each of its cohomology sheaves in turn, we reduce to
the case where~$A$ is a sheaf. If the support of~$A$ is not dense
in~$X$, then the support has relative dimension~$\leq n-1$, and we are
done by induction. If the support of~$A$ is dense in~$X$, then by
Theorem~\ref{singloc}, up to shrinking $V$, we may assume that the
restriction of~$A$ to~$V$ is a lisse sheaf. Applying induction to the
restriction of~$A$ to $X-V$ and excision, we can assume that $A$ is a
lisse sheaf supported on $V$.

Let~$j$ be the open immersion $V\to X$. By Step~4, the result holds
for~$j$, and hence up to shrinking~$S$, we may assume that $j_*j^*A$
is of formation compatible with any base change. Choose a cone~$C$ of
the canonical adjunction morphism $A\to j_*j^*A$. By definition, we
have a distinguished triangle
\[
  A\longrightarrow j_*j^*A\longrightarrow C
\]
and the cohomology of $C$ is supported on $Z=X- V$. Since the first
two complexes are of formation compatible with any base change, the
same is true for $C$. By Theorem~\ref{pushpull} and
Proposition~\ref{easylin}, the complexity of the restriction of~$C$
to~$Z$ is bounded in terms of the complexity of~$A$. Since the
relative dimension of~$Z$ is~$<n$, the induction hypothesis applies to
the closed immersion $Z\to X$ and to the restriction of~$C$
to~$Z$. Hence, up to shrinking~$S$, we can assume that $f_*C$ is of
formation compatible with any base change.

We then apply~$f_*$ to the previous distinguished triangle and
obtain
\[
  f_*A\to (fj)_*j^*A\to f_*C
\]
By Step~4 again, the result holds for the open immersion
$fj \colon V\to Y$, so that after shrinking~$S$, the objet
$(fj)_*j^*A$ is of formation compatible with any base change. Since we
have seen that this also holds for $f_*C$, we deduce that $f_*A$ is of
formation compatible with any base change, as desired.
\end{proof}

We now prove the sublemmas above.

\begin{proof}[Proof of Lemma~\ref{lm-sub}]
  The problem is local on $Y$, so by Lemma~\ref{lem-affine-open-cond},
  we can also assume that $Y$ is affine over $S$. We hence assume that
  $Y\subset \Aa^r_S$, where the choice of coordinates is induced
  by~$v$.

  For $1\leq i\leq r$, consider the $i$-th coordinate projection
  $p_i\colon Y\to \Aa^1_S$. We view $X$ and~$Y$ as $\Aa^1_S$-schemes
  using $p_i$ as structure morphism. The generic relative dimensions
  of~$X$ and~$Y$ are then~$\leq n-1$, and hence we can apply the
  induction hypothesis to this situation. We find a dense open subset
  $U_i\subset \Aa^1_S$ such that the complement of $U_i$ in $\Aa^1_S$ is
  of degree bounded in terms of $c(A)$, and such that $f_*A$ is of
  formation compatible with any base change $S'\to U_i\subset \Aa^1_S$.
Define
  $$
  Y'=\bigcup_{1\leq i\leq r}p_r^{-1}(U_i).
  $$
Since the complexity of the $U_i$ is bounded in terms of $c(A)$, the
  same holds for~$Y'$. From the definition of $U_i$, the result holds
  for the restriction of~$f$ to~$X'=f^{-1}(Y')$. Since the morphism
\[
  Y-Y'= (\Aa^1_S- U_1)\times_S\cdots\times_S (\Aa^1_S-
  U_r)\longrightarrow S
\]
is generically finite, there exists a dense open $S'\subset S$ over
which this morphism is finite. The degree of $S-S'$ is
bounded in terms of the degrees of $\Aa^1_S-U_i$,
for \hbox{$i=1, \ldots, r$}, and hence is itself bounded in terms of $c(A)$,
which ends the proof of the sublemma.
\end{proof}

\begin{proof}[Proof of Lemma~\ref{lm-subsub}]
  Recall that $X$ is assumed to be affine, with a fixed closed embedding
  into $\Aa^r_S$ for some integer $r\geq 0$. By the effective version of
  generic flatness, see, e.g., \cite[Th.\,2.45]{vasco}, we find a
  dense open subset~$S'$ of~$S$ over which~$X$ is flat and such that the complement of $S'$ in $S$ is of bounded degree. We may therefore
  assume that $X$ is flat over $S$. To obtain smoothness, we use the
  Jacobian criterion. We want to define $U$ to be the complement in $X$
  of the vanishing locus of the Jacobian ideal, which is indeed of
  bounded degree. However, for this to define a dense open subset of
  $X$, we need to perform first a base change along a finite surjective
  radicial morphism $S'\to S$ (which does not change étale cohomology),
  then replace $X$ and $Y$ by their reductions. The degree of the finite
  surjective radicial morphism is determined by the degrees (according
  to the affine embedding of $S$) of the coefficients of the polynomials
  defining~$X$ in $\Aa^r_S$, and hence is bounded. Once the result is
  known for $X_{S'}$ over~$S'$, we conclude the proof by considering the
  image~$U$ of the dense open subset of $S'$ by the morphism $S'\to S$,
  since the degree of the complement of~$U$ is still bounded.
\end{proof}

\subsection{Tannakian operations}\label{sec-tannaka}

Let $(X,u)$ be a connected quasi-projective algebraic variety
over~$k$, and let~$\mathcal{F}$ be a lisse $\ell$-adic sheaf
on~$X$. We shall view~$\mathcal{F}$ as a continuous representation~$\pi$ of the étale fundamental group of~$X$ on a finite-dimensional~$\Ql$\nobreakdash-vector
space. The \emph{monodromy group} of~$\mathcal{F}$ is then defined as the Zariski closure of its image, which is an algebraic group $G$ over $\Ql$ equipped with a distinguished faithful representation. Given a further algebraic representation~$\rho\colon G\to \GL_N(\Ql)$,
we denote by~$\rho(\mathcal{F})$ the lisse $\ell$-adic sheaf
corresponding to the representation $\rho\circ \pi$.

The following lemma is very useful in deriving properties of $\rho(\mathcal{F})$ from those of~$\mathcal{F}$. By a reductive group over a field of characteristic zero, we mean a group $G$ all whose finite-dimensional representations are completely reducible (that is, $G$ is not assumed to be connected).

\begin{lem}\label{lm-tannaka}
  Let~$G\subset \GL_N$ be a reductive group over an algebraically
  closed field of characteristic~zero and denote by~$\mathrm{Std}$ its
  tautological $N$-dimensional representation. For any algebraic irreducible
  representation~$\rho\colon G\to \GL(V)$, there exist non-negative
  integers~$(a,b)$ such that $\rho$ is a subrepresentation
  of~$\mathrm{Std}^{\otimes a}\otimes \dual(\mathrm{Std})^{\otimes
    b}$. If the determinant of $\mathrm{Std}$ has finite order, there
  exists such a pair with~$b=0$.
\end{lem}

\begin{proof}
  For the first part of the statement, see,
  e.g.,~\cite[Th.\,III.4.4]{compact_representations} in the case of
  compact Lie groups, and use the equivalence of categories between
  reductive groups and compact Lie groups.
  \par
  When the determinant has finite order, it suffices to prove that the
  contragredient of the tautological representation is a
  subrepresentation of a tensor power of the standard representation.
  But, using brackets to denote multiplicity, we have
  $$
  \langle\dual(\mathrm{Std}),\mathrm{Std}^{\otimes a}\rangle=\langle
  \mathbf{1},\mathrm{Std}^{\otimes (a+1)}\rangle
  $$
 for all $a\geq 0$, and if $\det(\mathrm{Std})^{m}=\mathbf{1}$ for some~$m\geq 1$, we have
  $\langle \mathbf{1},\mathrm{Std}^{\otimes mN}\rangle\geq 1$. Indeed, as the determinant is the highest exterior power, there is an inclusion $\det \subset \mathrm{Std}^{\otimes N}$, and this induces $\mathbf{1}=\det^{\otimes m} \subset \mathrm{Std}^{\otimes mN}$.
\end{proof}

\begin{prop}\label{pr-tannaka}
  With notation as above, assume that the group~$G$ is reductive with
  finite center. There exists an integer~$a\geq 0$, depending only
  on~$\rho$, such that
  $$
  c_u(\rho(\mathcal{F}))\ll c_u(\mathcal{F})^a, 
  $$
  where the implied constant depends only on the embedding dimension
  of~$u$.
\end{prop}

\begin{proof}
Let~$\mathrm{Std}$ denote the standard faithful representation of~$G$ corresponding to the sheaf~$\mathcal{F}$. The identity component $G^0$ of $G$ is reductive, so the restriction of the one-dimensional character $\det(\mathrm{Std})$ of $G$ to $G^0$ is non-trivial if and only if its restriction to the central torus of $G^0$ is non-trivial. The central torus of $G^0$ is a torus with an action of the (finite) component group of $G$ by conjugation which does not fix any non-trivial cocharacter. Hence, it does not fix any non-trivial character, so the restriction of $\det(\mathrm{Std})$ to $G^0$ is trivial, and thus $\det(\mathrm{Std})$ has finite order. By Lemma~\ref{lm-tannaka}, there exists a non-negative integer~$a$
  such that $\rho$ is a direct summand of $\mathrm{Std}^{\otimes a}$. Thus, the estimates
  \[
  c_{u}(\rho(\mathcal{F}))\leq c_{u}( \mathcal{F}^{\otimes a})
  \ll c_{u}(\mathcal{F})^{a}
  \]
  hold by Proposition~\ref{easylin}\,\eqref{easylin:item2} and parts \eqref{eq-6dual} 
  and~\eqref{eq-6tensor} of Theorem \ref{pushpull}. 
\end{proof}

\begin{remark}
  If $X$ is a curve, the estimate can be strengthened considerably
  to
  $$
  c_u(\rho(\mathcal{F}))\ll \dim(\rho)c_u(\mathcal{F})
  $$
  by using the ``local'' formula of Theorem~\ref{curveconstruction}
  below (see~\cite[Ch.\,3]{gksm} for arguments of this kind). It would
  be very interesting to know whether such a strong bound holds in
  higher dimension. 
\end{remark}

In some applications, it is natural to try to improve on this bound by
estimating~$a$ effectively, or in other words, to make
Lemma~\ref{lm-tannaka} effective. We record here one such estimate,
since this might be of interest for other applications.

Let~$N\geq 1$ be an integer and let~$G\subset \GL_{N}$ be a reductive
algebraic group over an algebraically closed field of
characteristic~zero. As above, we denote by~$\mathrm{Std}$ the tautological faithful representation of~$G$ in~$\GL_N$.

\begin{defi}\label{def-invariant}
  For any irreducible algebraic representation $\rho$ of~$G$, we
  define
  \[
    w(\rho) = \min \left\{ a+ b \mid \rho\text{ is a subrepresentation
        of } \mathrm{Std}^{\otimes a} \otimes
      \dual(\mathrm{Std})^{\otimes b}\right\}.
    \]
\end{defi}

Lemma~\ref{lm-tannaka} means that $w(\rho)$ is finite for all
irreducible algebraic representations~$\rho$ of~$G$.

\begin{prop}\label{reptheory}
  Let~$\Lambda$ be the weight lattice of the identity component~$G_0$
  of~$G$. Fix a norm~$\|\cdot\|$ on~$\Lambda$. There
  exist real numbers~$\alpha$, $\beta\geq 0$ such that, for any
  irreducible algebraic representation $\rho$ of~$G$, the inequality
  $$
  w(\rho)\leq \alpha \max_i\|\lambda_i\|+\beta
  $$
  holds, where the $\lambda_i$ are the highest weights of the
  irreducible components of the restriction of~$\rho$ to~$G_0$.
\end{prop}

\begin{proof}
  In this proof, we will say that a pair $(a,b)$ is a \emph{spot} for
  a representation~$\rho$ of~$G$ if~$\rho$ appears as a
  subrepresentation of $\mathrm{Std}^{\otimes a}\otimes
  \dual(\mathrm{Std})^{\otimes b}$.
  \par
  Let~$n$ be the semi-simple rank of~$G$. For a positive dominant
  weight~$\lambda$ in~$\Lambda$, we denote by~$\pi_{\lambda}$ the
  irreducible representation of~$G^0$ with highest weight~$\lambda$.
  \par
  We first assume that~$G$ is a connected semisimple group. In this
  case, the representation~$\rho$ is associated to a single
  highest weight~$\lambda$. Moreover, we can choose a basis $(e_i)$ of
  the root lattice tensored with~$\Rr$ such that a Weyl chamber can be
  identified with the cone of vectors with coordinates~$\geq 0$.
  \par
  For integers~$i$ with~$1\leq i\leq n$, we denote by $\lambda_i$ the
  weight in the line spanned by~$e_i$ with minimal $i$-th coordinate,
  and we write~$\lambda_i=x_ie_i$ for some integer~$x_i\geq 1$. Further,
  for each~$i$, we fix spots~$(a_i,b_i)$ of~$\pi_{\lambda_i}$ (which
  exist by Lemma~\ref{lm-tannaka}).
  
  We denote by $A$ the finite set of positive dominant weights $\mu$
  such that each coordinate of~$\mu$ is~$\leq x_i$ for all~$i$. For
  each~$\mu\in A$, we fix spots~$(a_{\mu},b_{\mu})$ of~$\pi_{\mu}$.
 
   By subtracting from a weight $\lambda$ of $G$ suitable multiples
  of~$\lambda_1$, \ldots, $\lambda_n$, until all coordinates
  are~$\leq x_i$, we see that we may write
  $$
  \lambda=\mu+\sum_{i=1}^r n_i\lambda_i
  $$
  where~$\mu\in A$ and $0\leq n_i\leq \ell_i(\lambda)$ for some
  linear maps~$\ell_i\colon \Lambda\to\Zz$.
 
  From highest weight theory, the representation~$\rho$ with highest
  weight~$\lambda$ is a summand of
  $$
  \pi_{\mu}\otimes \bigotimes_{1\leq i\leq r}\pi_{\lambda_i}^{\otimes n_i}. 
  $$
 Hence, by definition, the pair~$(a,b)$ given by
  $$
  a=a_{\mu}+\sum_{i=1}^r n_ia_i,\quad b=b_{\mu}+\sum_{i=1}^r n_ib_i,
  $$
  is a spot of~$\rho$.  We have then
  $$
  w(\rho)\leq a+b\leq \alpha+\beta\|\lambda|
  $$
  where
  $$
  \alpha=\max_{\mu\in A}(a_{\mu}+b_{\mu}),\quad \beta=\max_{1\leq
    i\leq r} \|\ell_i\| \times \sum_{i=1}^r (a_i+b_i)
  $$
  (with the usual norm for linear maps~$\Lambda\to\Rr$ with respect to
  the given norm on~$\Lambda$).
  
  Assume now that~$G$ is a connected reductive group.
  We construct the  weights $\lambda_i$ and~$A$  of the derived group
  of~$G$ as before, then lift them (keeping the notation) to the whole
  group. Let further~$(\gamma_j)_{1\leq j\leq k}$ be a basis of the
  subspace of the weights whose associated representations are trivial
  on the derived group.

  For any highest weight $\lambda$ of~$G=G^0$, the
  representation~$\pi_{\lambda}$ can then be expressed as a summand of
  a representation of the form
  $$
  \pi_{\mu}\otimes \bigotimes_{1\leq i\leq r}\pi_{\lambda_i}^{\otimes
    n_i}\otimes \bigotimes_{1\leq j\leq k}\pi_{\gamma_j}^{\otimes
    m_j}
  $$
  where, as before, the~$n_i$ and~$m_j$ are bounded by linear
  functions of~$\lambda$. This leads to a bound of the desired type as
  in the semisimple case.

  Finally, we consider the general case. Let~$\mathrm{Std}_0$ be the
  tautological representation of~$G^0$. Note that for any non-negative
  integers~$(a,b)$, the representation
  \begin{equation}\label{eq-wab}
    \mathrm{Ind}_{G^0}^G(\mathrm{Std}_0^{\otimes
      a}\otimes\dual(\mathrm{Std}_0)^{\otimes b})
  \end{equation}
  contains
  $\mathrm{Std}^{\otimes a}\otimes \dual(\mathrm{Std})^{\otimes b}$ as
  a subrepresentation. Moreover, by Frobenius reciprocity, the
  multiplicity of~$\rho$ in~(\ref{eq-wab}) is the dimension of the
  space of $G^0$-morphisms from the restriction~$\rho_0$ of~$\rho$
  to~$G^0$ to
  $\mathrm{Std}_0^{\otimes a}\otimes\dual(\mathrm{Std}_0)^{\otimes
    b}$. For an irreducible component~$\rho'$ of~$\rho_0$, this space
  is non-zero for some~$(a,b)$ with $a+b=w(\rho')$. The bound then
  follows from the case of connected groups established in the first part. 
\end{proof}

\subsection{Independence of $\ell$}
\label{sec-indep}

In this section and the next one, we work over a finite field~$\Ff$ and~$k$ is an
algebraic closure of~$\Ff$. We use the notation for trace functions
from the introduction. We first prove that the complexities of a compatible system
of $\ell$-adic sheaves are largely independent of $\ell$, following
ideas from Katz~\cite{katzbetti}. 

We fix a quasi-projective variety $(X,u)$ over $\Ff$ and a field~$K$
of characteristic zero. In order to vary the prime~$\ell$, we make the following definitions.
Let $\Lambda$ be a non-empty set and let $\mathcal S$ be a family
$(\ell_{\lambda},\iota_{\lambda})_{\lambda\in\Lambda}$ of pairs where
each $\ell_\lambda$ is a prime number invertible in~$\Ff$ and
$\iota_\lambda$ is a field embedding of the given field $K$ in
$\bQ_{\ell_\lambda}$.

\begin{defi}
  An $\mathcal S$-compatible system of complexes on~$X$ is a family
  $(A_\lambda)_{\lambda\in\Lambda}$ consisting of objects $A_{\lambda}$ of
  $\Der(X,\bQ_{\ell_\lambda})$ such that the following conditions hold:
  \begin{itemize}
  \item for any $\lambda\in\Lambda$, any finite extension $\Ff_n$ of $\Ff$, and any
    $x\in X(\Ff_n)$, the trace $t_{A_{\lambda}}(x;\Ff_n)$ takes values in the
    subfield $\iota_{\lambda}(K)$ of $\bQ_{\ell_\lambda}$;
  \item for any $\lambda$, $\mu$ in~$\Lambda$, any finite
    extension~$\Ff_n$ of~$\Ff$, and any $x\in X(\Ff_n)$, we have
    $$
    \iota_{\lambda}^{-1}t_{A_{\lambda}}(x;\Ff_n)=
    \iota_{\mu}^{-1}t_{A_{\mu}}(x;\Ff_n).
    $$
  \end{itemize}
\end{defi}

\begin{lem}\label{lm-euler-indep}
  Let $(A_\lambda)_{\lambda\in\Lambda}$ be an $\mathcal{S}$-compatible system on~$X$. The Euler--Poincaré characteristic
  $\chi_c(X\times k,A_{\lambda})$ of~$A_{\lambda}$ is independent
  of~$\lambda$.
\end{lem}

\begin{proof}
  For any~$\lambda$, it follows from the definition of $\mathcal{S}$-compatible system that the
  $L$-function of $A_{\lambda}$ is independent of~$\lambda$.  From the
  expression for this $L$-function given by the
  Grothendieck--Lefschetz formula, we know
  that~$\chi_c(X\times k,A_{\lambda})$ is the degree of this rational
  function (degree of the numerator minus degree of the denominator),
  and hence is independent of~$\lambda$.
\end{proof}

We then have the following result:

\begin{theo}\label{independence}
  Let~$\mathcal S$ be as above.
  \par
  \emph{(1)} Let~$n\geq 0$ be an integer and let $(A_\lambda)_\lambda$
  be an $\mathcal S$-compatible system of perverse sheaves on
  $\Pp^n_\Ff$. Then for  $\lambda$ and~$\mu$ in~$\Lambda$, we have
  \[
    c(A_\lambda)\asymp c(A_{\mu}),
  \]
  where the implied constants only depend on $n$.
  \par
  \emph{(2)} Let $(X,u)$ be a smooth quasi-projective variety over
  $\Ff$ and let $(\mathcal F_\lambda)_\lambda$ be an
  $\mathcal S$\nobreakdash-compatible system of lisse sheaves on~$X$. Then for all
  $\lambda$ and~$\mu$ in~$\Lambda$, we have
  \[
    c_u(\mathcal F_\lambda)\asymp c_u(\mathcal F_{\mu}),
  \]
  where the implied constants only depend on the embedding dimension
  of $(X,u)$.
\end{theo}

\begin{proof}
  (1) We argue by induction on $n$, following the strategy of
  Katz~\cite[Th.\,7]{katzbetti}. It is enough to prove the estimate
   $$
  c(A_1)\ll c(A_2)
  $$ 
  for two compatible perverse sheaves $A_1$ in $\Der(\Pp^n,\bQ_{\ell_1})$ and
  $A_2$ in $\Der(\Pp^n,\bQ_{\ell_2})$, with an implied constant that only depends on~$n$.

  If $n=0$, then $A_1$ and $A_2$ are vector spaces of the same
  dimension and their complexity $c(A_1)=c(A_2)$ is this common dimension.

  We now fix $n\geq 1$ and assume that the result holds on $\Pp^m$
  for all $0\leq m\leq n-1$. For each $0\leq m\leq n-1$, we pick a geometric generic point
  \hbox{$\mtrx{a}_m\in M^{n+1,m+1}(k')$} for some algebraically closed
  field~$k'$.

  Let~$0\leq m\leq n-1$. By applying Lemma~\ref{lem-indepl} to $A_1$
  and $A_2$, we find open dense subsets~$U_{m,1}$ and $U_{m,2}$ of
  $M^{n+1,m+1}_k$ satisfying the properties of
  Lemma~\ref{lem-indepl}. Up to replacing $\Ff$ by a finite extension,
  which we may do, we can assume that $(U_{m,1}\cap U_{m,2})(\Ff)$ is
  non-empty for all~$m$, and we then pick
  $\mtrx{b}_m\in (U_{m,1}\cap U_{m,2})(\Ff)$, again for all~$m$.  The
  complexes $l_{\mtrx{b}_m}^* A_1[m-n]$ and
  $l_{\mtrx{b}_m}^* A_2[m-n]$ on $\Pp^{m}$ over~$\Ff$ are perverse and
  form a compatible system; hence, by the induction hypothesis, the equivalence 
  \begin{equation}\label{eq-k1}
    c(l_{\mtrx{b}_m}^* A_1) \asymp c(l_{\mtrx{b}_m}^* A_2)
  \end{equation}
  holds, with implied constants that only depend on $n$.
  We have
  \[
    \max_{0\leq m\leq n-1} \sum_{i \in \Zz} h^i(\Pp^m_{k'},
    l_{\mtrx{a}_m}^* A_1)=\max_{0\leq m\leq n-1} \sum_{i \in \Zz}
    h^i(\Pp^m_{k'}, l_{\mtrx{b}_m}^* A_1),
  \]
  again by Lemma~\ref{lem-indepl}. Hence
  \begin{equation}\label{eq-k2}
    \max_{0\leq m\leq n-1} \sum_{i \in \Zz} h^i(\Pp^m_{k'},
    l_{\mtrx{a}_m}^* A_1)=c(l_{\mtrx{b}_m}^* A_1),
  \end{equation}
  by the definition of the complexity and Lemma~\ref{genlem2}.
  
  Together with~(\ref{eq-k1}), it follows that
  \begin{align*}
    c(A_1)&= \max_{0\leq m\leq n} \sum_{i \in \Zz} h^i(\Pp^m_{k'},
    l_{\mtrx{a}_m}^* A_1) \\ & \leq \max_{0\leq m\leq n-1} \sum_{i \in
      \Zz} h^i(\Pp^m_{k'}, l_{\mtrx{a}_m}^* A_1) +\sum_{i \in \Zz}
    h^i(\Pp^n_{k'}, l_{\mtrx{a}_n}^* A_1)\\
    &\ll c(l_{\uple{b}_m}^*A_2)+ \sum_{i \in \Zz} h^i(\Pp^n_{k'},
    l_{\mtrx{a}_n}^* A_1)\\
    &\ll c(A_2)+ \sum_{i \in \Zz} h^i(\Pp^n_{k'}, l_{\mtrx{a}_n}^*
    A_1).
  \end{align*}
  
  Since the geometric generic point $\mtrx{a}_n\in M^{n+1,n+1}_{k'}$
  induces an isomorphism $l_\mtrx{a_n}$ of $\Pp^n_{k'}$, it is enough
  to prove the estimate
  \[
    \sigma_1=\sum_{i\in \Zz} h^i(\Pp^n_{k'}, A_1)\ll c(A_2)
  \]
  
  Arguing as in the proof of Proposition \ref{bilinearbettibound}, we
  have
  \[
    h^0(\Pp^n_{k'}, A_1)=\chi(\Pp^n_{k'}, A_1)+ \sum_{i<0}(-1)^{i+1}
    h^i(\Pp^n_{k'}, A_1)+\sum_{i>0}(-1)^{i+1} h^i(\Pp^n_{k'}, A_1),
  \]
  and therefore
  \begin{align*}
    \sigma_1&=\sum_{i<0} h^i(\Pp^n_{k'}, A_1)+ h^0(\Pp^n_{k'}, A_1)+\sum_{i>0} h^i(\Pp^n_{k'}, A_1)\\
    &\leq 2\sum_{i<0} h^i(\Pp^n_{k'}, A_1)+\chi(\Pp^n_{k'},
    A_1)+2\sum_{i>0} h^i(\Pp^n_{k'}, A_1).
  \end{align*}
  
  By excision and Artin's vanishing theorem for affine varieties, and
  the fact that~$A_1$ is perverse, the canonical map
  \[
    \rmH^{i}(\Pp^n_{k'},A_1)\longrightarrow \rmH^i(\Pp^{n-1}_{k'},l^*_{\mtrx{a}_{n-1}}A_1)
  \]
  is an isomorphism for $i<0$. Similarly, because the dual of a
  perverse sheaf is perverse and duality exchanges
  $l^*_{\mtrx{a}_{n-1}}$ and $l^!_{\mtrx{a}_{n-1}}$, the canonical map
  \[
    \rmH^i(\Pp^{n-1}_{k'},l^!_{\mtrx{a}_{n-1}}A_1)\longrightarrow
    \rmH^{i}(\Pp^n_{k'},A_1)
  \]
  is an isomorphism for $i>0$. Since $l_{\mtrx{a}_{n-1}}$ is a closed
  immersion, the functors $l^*_{\mtrx{a}_{n-1}}$ and~$l^!_{\mtrx{a}_{n-1}}$
  are equal up to a shift and a Tate twist. Hence, 
  \begin{align*}
    \sigma_1&=2\sum_{i<0} h^i(\Pp^n_{k'}, l^*_{\mtrx{a}_{n-1}}A_1)+\chi(\Pp^n_{k'}, A_1)+2\sum_{i>0} h^i(\Pp^n_{k'}, l^!_{\mtrx{a}_{n-1}}A_1)\\
    &\leq 4 \sum_{i\in \Zz} h^i(\Pp^{n-1}, l^*_{\mtrx{a}_{n-1}}A_1)+
    \chi(\Pp^n_{k'}, A_2)
  \end{align*}
  using Lemma~\ref{lm-euler-indep}. Hence, 
  $$
  \sigma_1\ll c(A_2)+\chi(\Pp^n_{k'},A_2) \ll c(A_2),
  $$
  by~(\ref{eq-k1}) and~(\ref{eq-k2}), and the elementary
  Proposition~\ref{pr-betti} below.
  
  (2) The proof is similar, using the obvious adaptation of
  Lemma~\ref{lem-indepl} to lisse sheaves, and the fact that the dual of a
  lisse sheaf on a smooth scheme is lisse. Alternatively, but at the cost
  of adding a dependency of the implied constants on $c(u)$, we may
  use~\cite[Th.\,3]{fujiwara} to reduce to~(1); indeed, this shows that
  the middle extensions $u_{!*}A_{\lambda}$ of the components of a
  compatible system of lisse sheaves is a compatible system of perverse
  sheaves on the projective space target of~$u$, and one can apply
  Corollary~\ref{cor-middle} to bound the complexity of the
  middle extensions.
\end{proof}

\subsection{Complex conjugation}

We keep working over a finite field~$\Ff$ with algebraic closure~$k$. In applications to analytic number theory, taking the complex conjugate
of the trace function~$t_A$ is a natural operation, but this has no canonical counterpart
at the level of~$\Der(X)$.  Nevertheless, we have the following result:

\begin{prop}\label{pr-conjugation}
  Let~$(X,u)$ be a quasi-projective algebraic variety
  over~$\Ff$. Given any mixed object~$A$ of~$\Der(X)$, there exists a
  mixed object~$A'$ of $\Der(X)$ such that
  \[
    t_{A'}(\cdot;\Ff_n)=\overline{t_A(\cdot;\Ff_n)}\quad \text{for any~$n$}
    \quad \text{and}\quad c_u(A')\ll c(u)c_u(A),
  \]
where the implied constant depends only on the embedding dimension
of~$(X, u)$. 
\end{prop}

\begin{proof}
  For each $n \geq 1$, the trace function $t_A(\cdot;\Ff_n)$ is the
  restriction to the set $X(\Ff_n)$ of~$\Ff_n$-rational points of $X$
  of the trace function of the complex~$u_!A$ on~$\Pp^n$, which is
  again mixed if so is $A$. Assume that we have found a mixed
  object~$\widetilde{A}$ of~$\Der(\Pp^n)$ that ``works'' for $u_!A$,
  then $A'=u^*\widetilde{A}$ works for~$A$, and
  $$
  c_u(A')\ll c(u)c(\widetilde{A})\ll c(u)c(u_!A)=c(u)c_u(A),
  $$
  where the first bound is part \eqref{eq-6pullback} of Theorem~\ref{pushpull}.  Thus, we may assume $X=\Pp^n$ (with the identity embedding).
  \par
  We first assume that~$A$ is a perverse sheaf. If~$A$ is pure of
  weight zero, then $A$ is geometrically semisimple by \cite[5.3.8]{BBDG}. As explained in~\cite[Lem.\,1.8.1.\,(1)]{mmp}, a result of Gabber on the independence of $\ell$ of traces of intermediate extensions \hbox{\cite[Th.\,3]{fujiwara}} then implies that we can take~$A'=\dual(A)$. The perverse sheaf $A'$ is also pure of weight zero and has complexity $c(A')=c(A)$ by
  Lemma~\ref{dual1}. If~$A$ is pure of weight~$w$, then we can simply
  take $A'=\dual(A(w))(-w)$, which is pure of weight~$w$.

  Suppose now that~$A$ is a mixed perverse sheaf. Let $(F_j)$ be the
  canonical weight filtration on~$A$, with quotients~$F_{j+1}/F_j$
  pure of weight~$w_j$ (see~\cite[Th.\,5.3.5]{BBDG}). We first claim 
  that the estimate
  \[
  \sum_{j}c(F_{j+1}/F_j)\ll c(A)
  \] holds. Indeed, since the quotients are perverse sheaves, we have
  $$
  \sum_{j}c(F_{j+1}/F_j)\ll \sum_j \|\cc(F_{j+1}/F_j)\|= \Bigl\|
  \sum_j \cc(F_{j+1}/F_j) \Bigr\|=\|\cc(A)\|\ll c(A)
  $$
  by Corollary~\ref{genericbettibound} and the fact that the
  characteristic cycles of perverse sheaves are effective (compare
  with the end of the proof of Theorem~\ref{hardlin}).
  In view of the equality
  $$
  t_{A}(\cdot;\Ff_n)=\sum_{j} t_{F_{j+1}/F_j}(\cdot;\Ff_n)
  $$
  for all~$n\geq 1$, we can take
  $$
  A'=\bigoplus_{j} \dual((F_{j+1}/F_j)(w_i))(-w_i),
  $$
  which is a mixed perverse sheaf with complexity
  $$
  c(A')=\sum_j c(\dual(F_{j+1}/F_j))= \sum_{j} c(F_{j+1}/F_j)\ll c(A),
  $$
  by Lemma~\ref{dual1} and the above.
    
  Finally, in the general case, if~$A$ is a mixed complex of
  weights~$\leq w$, then its perverse cohomology sheaves~$\PH^i(A)$
  are mixed of weights~$\leq w+i$ for all~$\in\Zz$
  by~\cite[Th.\,5.4.1]{BBDG} and we have
  $$
  t_A(\cdot;\Ff_n)=\sum_{i\in\Zz}(-1)^it_{\PH^i(A)}(\cdot;\Ff_n)
  $$
  for all~$n\geq 1$. For all~$i$, let~$A'_i$ be a mixed perverse sheaf satisfying the
  conditions of the proposition for~$\PH^i(A)$. We can define
  $$
  A'=\bigoplus_{i\in\Zz}A'_i,
  $$
  which satisfies
  $$
  c(A')=\sum_{i\in\Zz}c(A'_i)\ll \sum_{i\in\Zz} c(\PH^i(A)) \ll c(A)
  $$
  by the previous case.
\end{proof}

\begin{remark}
  In practice, the conjugate of the trace functions that occur in
  concrete applications are often clearly trace functions with
  essentially the same complexity (e.g., the conjugate of the trace
  function of an Artin--Schreier sheaf $\mathcal{L}_{\psi(f)}$ is
  simply $\mathcal{L}_{\psi(-f)}$). Nevertheless,
  Proposition~\ref{pr-conjugation} might be useful for certain
  theoretical arguments.
\end{remark}

\begin{remark}
  In fact, the assumption that $A$ is a mixed object is always
  satisfied thanks to the deep theorem of L.\,Lafforgue that an
  irreducible lisse sheaf with determinant of finite order on a normal
  variety $X$ is pure of weight zero (see
  \cite[Prop.\,VII.7\,(i)]{Lafforgue}). We briefly sketch how to
  deduce from this that $A$ is mixed. By definition, it suffices to
  treat the case where $A$ is a constructible sheaf and, by induction
  on the dimension of~$X$, we can even suppose that $A$ is lisse. The
  successive quotients of a Jordan--Hölder filtration of $A$ are then
  irreducible lisse sheaves; up to a twist by some $\Ql(w)$, their
  determinant is of finite order by geometric class field theory (see
  \cite[(1.3.6)]{weilII}), and hence they are pure thanks to
  Lafforgue's theorem.
\end{remark}


\section{Examples and applications}\label{creation}

\subsection{Sums of Betti numbers}

One of the most useful properties of the complexity is that it
controls Betti numbers, as the following proposition shows.

\begin{prop}\label{pr-betti}
  Let $(X,u)$ be a quasi-projective variety over~$k$. For any~$A$ in
  $\Der(X)$, the following holds: 
  \begin{align*}
    |\chi_c(X,A)|&\leq \sum_{i \in \Zz} h^i_c(X,A) \leq c_u(A),
    \\
    |\chi(X,A)|&\leq \sum_{i \in \Zz} h^i(X,A) \ll c(u) c_u(A).
  \end{align*}
  Moreover, the implied constant in the second estimate only depends on the
  embedding dimension of~$(X,u)$.
\end{prop}

\begin{proof}
  Let $n$ be the embedding dimension of~$(X,u)$. For each object $A$ of $\Der(X)$, the equality 
  $h^i_c(X, A)=h^i(\Pp^n_{\bar k}, u_{!}A)$ holds. Since étale cohomology is
  invariant under extension of scalars to an algebraically closed
  field $k'$ and~$l_{\mtrx{a}_n} \colon \Pp^n_{k'} \to \Pp^n_{k'}$ is an isomorphism, the first bound is a straightforward consequence of the definition of the complexity $c_u(A)$. On noting that Verdier duality implies the equality of sums of Betti numbers $\sum_{i \in \Zz} h^i(X, A)=\sum_{i \in \Zz} h^i_c(X, \dual(A))$, the second estimate then follows from the first one and the estimate $c_u(\dual(A)) \ll c(u) c_u(A)$ from \eqref{eq-6dual}, with an implied constant that only depends on $n$. 
  \end{proof}

\begin{remark}
  The converse estimate
  \[
  c_u(A)\ll \sum_{i \in \Zz} h^i_c(X,A)
  \]
  does not hold in general, since the right-hand side is independent
  of~$u$, whereas we have seen in Example~\ref{ex-counter} that there is no reasonable intrinsic notion of complexity of an object of $\Der(X)$. This can also be seen concretely, e.g. in the
  case of curves. 
\end{remark}

\subsection{Complexity of sheaves on curves}

In the case of curves, we can write down an explicit
formula for the complexity.

\begin{theo}\label{curveconstruction}
  Let~$(C, u)$ be a smooth connected quasi-projective curve over~$k$. Let $\overline{C}$ be the smooth compactification of $C$, and denote by~$g$ the genus of~$\overline{C}$,
  by~$n$ the number of points of~$\overline{C}-C$, and by $d$ the degree of $u(C)$. 
  \begin{enumerate}
  \item\label{curveconstruction:item1} The complexity of an object~$A$ of~$\Der(C)$ is given by
  \begin{equation}\label{eq-curve}
    c_u(A) = \max \Bigl( d \sum_{i \in \Zz} \dim \mathcal H^i(A)_\eta,
    \sum_{i \in \Zz} h^i_c(C,A)\Bigr),
  \end{equation}
  where~$\eta$ stands for the generic point of~$C$.
  
  \item\label{curveconstruction:item2} If $A$ is a perverse sheaf on $C$, the following inequalities hold: 
  $$
  \max(d, 2g+n-2) \rank(A)+\loc(A)\leq c_u(A) \leq \max(d, 2g+n+2)
  \rank(A) +\loc(A).
  $$ 
    
  \item\label{curveconstruction:item3} If $\mathcal F$ is a middle-extension sheaf on
  $C$, the following inequalities hold:
  $$
  \max(d, 2g+n-2) \rank(\mathcal F) +\loc(\mathcal{F})\leq
  c_u(\mathcal F) \leq \max(d, 2g+n+2) \rank(\mathcal F)
  +\loc(\mathcal{F}).
  $$
   \end{enumerate}
\end{theo}

\begin{proof}
  Part \eqref{curveconstruction:item1} is simply a translation of the definition of~$c_u(A)$.
  \par
  To deduce \eqref{curveconstruction:item2}, we first note that a perverse sheaf $A$ satisfies
  $$
  \sum_{i \in \Zz} \dim \mathcal H^i(A)_\eta
  =\dim\mathcal{H}^{-1}(A)=\rank(A). 
  $$
 Writing the sum of Betti numbers as 
  $$
  \sum_{i \in \Zz} h^i_c(C,A)=-\chi_c(C,A)+2h^0(C,A)+2h^{-2}(C,A),$$
we then have
  $$
  -\chi_c(C,A) \leq \sum_{i \in \Zz} h^i_c(C,A) \leq
  -\chi_c(C,A)+4\rank(A).
  $$
  Taking the equality $\chi_c(C,\Ql[1])=2g-2+n$ into account, the statement then follows from the Grothendieck--Ogg--Shafarevitch formula.
  \par
  Finally, \eqref{curveconstruction:item3} is a special case of \eqref{curveconstruction:item2}, for the perverse sheaf~$A=\mathcal{F}[1]$. 
\end{proof}

As a corollary, we can now compare the complexity
defined in this paper with the ``analytic conductor'' of Fouvry,
Kowalski and Michel for the affine line over the algebraic closure
of a finite field.  Precisely, for a middle-extension
sheaf~$\mathcal{F}$ on~$\Aa^1$ over such a field, the latter is
defined in~\cite[Def.\,1.13]{fkm1} as
\begin{align*}
c_{\FKM}(\mathcal{F})&=\rank(\mathcal{F})+ \loc_{\FKM}(\mathcal{F}) \\
\loc_{\FKM}(\mathcal{F})&=(\text{number of singular points of $\mathcal{F}$
  in~$\Pp^1$})+\sum_{x\in\Pp^1}\swan_x(\mathcal{F}).
\end{align*}

\begin{cor}
 For a middle-extension sheaf $\mathcal{F}$ on~$\Aa^1$ over a
  finite field, the inequalities
  $$
  c_{\FKM}(\mathcal{F})\leq c_u(\mathcal F) \leq 3c_{\FKM}^2(\mathcal
  F)
  $$
  hold, with~$u$ the obvious embedding $\Aa^1\to \Pp^1$ of degree~$1$.
\end{cor}

\begin{proof}
  We apply Theorem~\ref{curveconstruction} with $g=0$, $n=1$,
  and~$d=1$.  On noting the inequality~$\drp_x(\mathcal{F})\leq \rank(\mathcal{F})$, the
  upper-bound in~\eqref{curveconstruction:item3} leads to
  \begin{align*}
    c_u(\mathcal{F})
    \leq 3\rank(\mathcal{F})+ \loc(\mathcal{F})
    &
      \leq
      3\rank(\mathcal{F})+\rank(\mathcal{F}) \loc_{\FKM}(\mathcal{F})
    \\
    &\leq 3\rank(\mathcal{F})c_{\FKM}(\mathcal{F})
      \leq 3c_{\FKM}(\mathcal{F})^2.
  \end{align*}
  In addition, the inequality $\loc_{\FKM}(\mathcal{F})\leq \loc(\mathcal{F})$ holds, hence the inequalities 
  $$
  c_{\FKM}(\mathcal{F})\leq \rank(\mathcal{F})+\loc(\mathcal{F})
  \leq c_u(\mathcal{F})
  $$
  by the lower-bound in \loccit
\end{proof}

\subsection{Artin--Schreier and Kummer sheaves}
\label{sec-ask}

In this section, $k$ is the algebraic closure of a finite field~$\Ff$. Given an $\ell$-adic additive character~$\psi\colon\Ff\to \Ql^\times$ (\resp an $\ell$-adic multiplicative character~$\chi\colon \Ff^{\times} \to \Ql^\times$), we
denote by~$\mcL_{\psi}$ (\resp $\mcL_{\chi}$) the corresponding
\emph{Artin--Schreier sheaf} on~$\Aa^1$ over~$\Ff$ (\resp the \emph{Kummer sheaf}
on~$\Gg_m$ over~$\Ff$), as defined, e.g., in~\cite[Sommes
exp.]{sga4h}. These are lisse sheaves of rank one; the Artin--Schreier sheaf is
wildly ramified at~$\infty$ if~$\psi$ is non-trivial, whereas the Kummer sheaf
is tamely ramified at~$0$ and~$\infty$. We also denote by~$\mcL_{\psi}$ and~$\mcL_{\chi}$ the middle extensions of these sheaves to~$\Pp^1$.

\begin{prop}\label{pr-as} Let $C = \Pp^1$ over~$k$ and
  $u=\mathrm{Id}$.
  \par
  \emph{(1)} For any non-trivial additive character $\psi$ of $\Ff$, we have
  $c_u(\mcL_\psi) =1$.
  \par
  \emph{(2)} For any non-trivial multiplicative character $\chi$ of
  $\Ff^{\times}$, we have $c_u(\mcL_\chi) =1$.
\end{prop}

\begin{proof}
  These follow from Theorem \ref{curveconstruction}\,\eqref{curveconstruction:item1} by the
  standard knowledge of the relevant Betti numbers.
\end{proof}



We can now easily estimate the complexity of more general
Artin--Schreier and Kummer sheaves, which are building blocks of many
of the sheaves used in applications to analytic number theory.

\begin{defi}
  Let~$X$ be a scheme over~$\Ff$ and $U$ a dense open subset
  of~$X$. Let~$j$ be the open immersion $U\to X$. For each morphism $f\colon U\to\Aa^1$, we define the \emph{Artin--Schreier} sheaf $\mcL_{\psi(f)}$ on~$X$ as~$\mcL_{\psi(f)}=j_!f^*\mcL_{\psi}$. For each morphism~$g\colon U\to\Gg_m$, we define the \emph{Kummer sheaf} $\mcL_{\chi(g)}$ on~$X$ as \hbox{$\mcL_{\chi(g)}=j_!g^*\mcL_{\chi}$}.
\end{defi}

Let $(X,u)$ be a quasi-projective variety over~$k$, and let
$v\colon \Aa^1\to \Pp^1$ be the obvious embedding (or its restriction to~$\Gg_m$).  Combined with~\eqref{eq-6pullback}
and~\eqref{eq-6extzero}, Proposition \ref{pr-as} gives the estimates 
\begin{gather*}
  c_u(\mcL_{\psi(f)})\ll c_{u\circ j,u}(j)c_{u\circ j,v}(f)
  \\
  c_u(\mcL_{\chi(g)})\ll c_{u\circ j,u}(j)c_{u\circ j,v}(g),
\end{gather*}
(with the same notation as above, and when~$\psi$ and~$\chi$ are
non-trivial), where the implied constants only depend on the embedding
dimension of~$u$. According to Remark~\ref{rem:uniformpolynomials}, these quantities are uniformly bounded
if~$f$ or $g$ varies among morphisms with ``bounded degree''.  One can
sometimes be more explicit, using tools like
Proposition~\ref{pr-explicit-uniform}. Here is a simple example:

\begin{cor}\label{cor-as}
  Let~$f_1$ and~$f_2$ be polynomials in~$n$ variables with coefficients in~$k$, with
  $f_2$ non-zero. Let~$U\subset \Aa^n$ be the open set where~$f_2$ does not vanish. Set $f=f_1/f_2\colon U\to \Aa^1$. Then the following estimate holds: 
  $$
  c_{u}(\mcL_{\psi(f)})\ll (4+\deg(f_1)+\deg(f_2))^{n+2}
  $$
  where $u \colon \Aa^n\to\Pp^n$  is the standard embedding and the implied constant only depends on~$n$.
\end{cor}

\begin{proof}
  Apply the bound above combined with
  Proposition~\ref{pr-explicit-uniform}, viewing $U$ as the subscheme
  of $\Aa^{n+1}$ defined by the equation $f_2(x)y=1$, and $f\colon U\to\Aa^1$ as the morphism $(x,y)\mapsto f_1(x)y$.
\end{proof}

\subsection{Further examples}

We collect here a few more examples of ``standard'' sheaves whose complexity is easily estimated. Again,~$k$ is the algebraic
closure of a finite field~$\Ff$. 

\begin{prop}\mbox{}
  \begin{enumerate}
  \item Let~$\mathcal{K}$ be any hypergeometric sheaf on~$\Gg_m$ in the
    sense of Katz\footnote{See~\cite[\S\,8.2]{esde} for the definition and properties of hypergeometric sheaves.}. We have
    $$
    c_u(\mathcal{K})\ll \rank(\mathcal{K})
    $$
    where the implied constant is absolute and~$u\colon \Gg_m\to \Pp^1$ is the standard embedding.
    
  \item Let $(E,u)$ be an elliptic curve over $\Ff$ embedded
    in~$\Pp^2$ as a cubic curve. Let $\theta$ be a non-trivial character
    $E(\Ff) \to \Ql^\times$, and let $\mcL_\theta$ be the associated
    character sheaf defined by the Lang isogeny. Then
    $c_u(\mcL_\theta)=3$. 
    
  \item Let $\{x_1,\ldots,x_4\}$ be four distinct points in
    $\Pp^1(\Ff)$. Let $\mathcal F$ be an irreducible middle\nobreakdash-extension sheaf of rank $2$ on $\Pp^1$ over~$\Ff$ with unipotent local monodromy at those four
    points.\footnote{\ Deligne and Flicker~\cite[Prop.\,7.1]{deligne-flicker} have
      proved that there are $|\Ff|$ such sheaves up to isomorphism and
      $\Ff$-twisting.}  Then $c(\mathcal F) = 2$.
  \end{enumerate}
\end{prop}

\begin{proof} All these statements follow straightforwardly from calculating
  Betti numbers and Theorem \ref{curveconstruction}\,(1).

  In the case of a hypergeometric sheaf, the cohomology is one-dimensional. 

  In case (2), we apply~\eqref{eq-curve} with $(n,g,d)=(0,1,3)$. The
  rank of $\mathcal L_\theta$ is $1$, so the first term in the maximum
  is equal to~$3$. The second term is $0$, since $h^i(E, \mcL_\theta)=0$ for all $i$ by the assumption that $\theta$ is non-trivial. 

  In case (3), we apply the formula with $(n,g,d)=(0,0,1)$.  The
  first term in the maximum is~$2$ since the rank of~$\mathcal{F}$
  is~$2$.  In the second term, we have
  $$
  h^0(\Pp^1, \mathcal F)= h^2(\Pp^1, \mathcal F) =0,
  $$
  since~$\mathcal{F}$ is an irreducible middle-extension sheaf of
  rank~$2$.  Thus, the sum of Betti numbers is equal to
  $- \chi(\Pp^1,\mathcal F)$. The Euler--Poincar\'{e} characteristic formula
  for $\mathcal F$ gives
  $$
  \chi(\Pp^1,\mathcal{F})=2\chi(\Pp^1,\Ql)-4\times 1=0
  $$
  Hence, the sum of the Betti numbers vanishes and $c(\mathcal{F})=2$.
\end{proof}

\subsection{Tame ramification}
  
When a sheaf naturally appears in higher dimension without being obtained from simpler sheaves by means of Grothendieck's six functors, one may
still be able to estimate its complexity, using either the direct
definition of complexity, or sometimes the uniformity statement of
Proposition~\ref{uniformity}. We illustrate this here with some
results involving tame ramification.

\begin{prop}\label{tame} Let $X$ be a smooth scheme over $k$, and let 
  $j\colon X \to \overline{X}$ be an embedding into a smooth scheme
  $\overline{X}$ such that the complement $D= \overline{X}\setminus X$
  is a divisor with normal crossings. Let
  $i\colon \overline{X} \to \Pp^n_k$ be a closed immersion, and let
  $\mathcal F$ be a lisse sheaf on $X$ that is tamely ramified
  along~$D$.  Then the following holds:
  $$
  c_{i\circ j}(\mathcal F) \ll \rank(\mathcal{F}) c(i\circ j).
  $$ 
  \end{prop}

\begin{proof} Set $d=\dim(X)$.  The object $\mathcal{F}[d]$ is
  perverse since $\mathcal F$ is a lisse sheaf on a smooth scheme of
  dimension $d$. Since~$j$ is an affine open immersion, the object
  $j_! \mathcal F[d]$ is perverse on~$\overline{X}$, and since $i$ is
  a closed immersion, the object $A=i_!j_! \mathcal F[d]$ is also
  perverse (see~\cite[Cor.\,4.1.3]{BBDG}). Since~$A$ is perverse, Corollary~\ref{genericbettibound}\,(2) and
  Lemma~\ref{constantsheafbasis} yield
  $$
  c_{i \circ j}(\mathcal F) = c(A) \ll \| \CC(A)\| \ll \sup_{0\leq
    m\leq n} |\CC(A) \cdot \CC(K_m)|,
  $$
  where $K_m$ is the constant sheaf on a generic $m$-dimensional
  subspace of~$\Pp^n$.

  By Theorem \ref{geneufor}, $|\CC( A) \cdot \CC(K_m)|$ is the
  absolute value of the Euler--Poincaré characteristic of the pullback
  of $A$ to a generic $m$-dimensional subspace, or equivalently the
  Euler--Poincaré characteristic of the restriction of $\mathcal F$ to
  the intersection of a generic $m$\nobreakdash-dimensional subspace with $X$.
 
  For a generic $m$-dimensional subspace~$H$, the intersection
  $H\cap \overline{X}$ is smooth, the intersection with $\overline{D}$
  has normal crossings, and the restriction of $\mathcal F$ has tame
  ramification around the intersection with $D$. This implies the equality
  $$
  \chi(X\cap H,\mathcal{F})=\rank(\mathcal{F})\chi(X\cap H),
  $$ and
  since the linear subspace~$H$ is generic, the Euler--Poincaré
  characteristic of $X\cap H$ is bounded by $c(i \circ j)$.
\end{proof}

\begin{cor}\label{bwr} Let $(X,u)$ and $(Y,v)$ be quasi-projective algebraic
  varieties over~$k$, and let $\pi\colon Y \to X$ be a finite
  \'{e}tale covering. Assume that~$v$ is the composition of an
  embedding into a smooth normal crossings compactification
  $\overline{Y}$ with a closed immersion. Let $\mathcal F$ be a lisse sheaf on $X$ such that
  $\pi^* \mathcal F$ is tamely ramified along $\overline{Y}-Y$. Then
 the estimate
  \[
    c_u(\mathcal F) \ll \rank(\mathcal{F}) c_{v,u}(\pi) c(v)
  \]
  holds, with an implied constant that only depends on the embedding dimensions of~$(X, u)$ and~$(Y, v)$.
\end{cor}

\begin{proof}
  The composition of the adjunction maps
  \[\mathcal F \to \pi_* \pi^* \mathcal F = \pi_! \pi^* \mathcal F \to
    \mathcal F
  \]
is the multiplication by the degree of
  $\pi$. The sheaf $\mathcal F$ is hence a direct factor of
  $\pi_! \pi^* \mathcal F$, and the estimates
  \[
    c_{u} (\mathcal F) \leq c_{u} ( \pi_! \pi^* \mathcal F) \ll
    c_{v,u}(\pi) c_{v}( \pi^* \mathcal F) \ll \rank(\mathcal{F})
    c_{v,u}(\pi) c(v),
  \]
  hold by Proposition~\ref{easylin}, part \eqref{eq-6extzero} of Theorem~\ref{pushpull}, and
 Proposition \ref{tame} applied to~$\pi^*\mathcal{F}$ on~$Y$.
\end{proof}

\subsection{The Riemann Hypothesis}\label{sec-rh}

In this section, we record the general ``quasi-ortho\-gonality'' version of the
Riemann Hypothesis over finite fields arising from Deligne's work and
the theory of complexity developed here. 

We work over a finite field~$\Ff$, and assume that~$k$ is an algebraic
closure of~$\Ff$. We fix an isomorphism $\iota$ from~$\Ql$ to~$\Cc$ to
define weights, as explained in the notation section. The notation for
trace functions is also recalled there.

We first explain the type of bounds for values of trace functions of
perverse sheaves that can be obtained using the complexity.

\begin{prop}\label{pr-bound}
  Let $(X,u)$ be a quasi-projective variety over $\Ff$ of pure dimension
  $d$. Let $M$ be a perverse sheaf on~$X$ that is pure of weight
  zero.
  \par
  \begin{enumerate}
  \item\label{pr-bound:item1} For any $x\in X(\Ff)$, the estimate $t_M(x)\ll c_u(M)$ holds.
  \item If $d\geq 1$ and $M$ is geometrically irreducible with support
    equal to~$X$, then the following estimate holds:
    \[
    t_M(x) \ll c_u(M) |\Ff |^{-1/2}.
    \]
    
  \item There exists a stratification of $X$, defined over~$\Ff$, into
    locally closed strata $S_j$ that are irreducible and smooth of pure
    dimension $d_j$ such that the degree of $u(S_j)$ is bounded in terms
    of $c_u(M)$ and the estimate
        $$
    t_M(x) \ll c_u(M) |\Ff|^{\frac{\max\{-d,-d_j-1\} }{2}}
    $$
    holds for all~$j$ and all~$x\in S_j(\Ff)$.
  \end{enumerate}
 In all these estimates, the implied constants only depend on the
  embedding dimension~of~$(X, u)$.
\end{prop}

\begin{proof}
  (1) The perversity of $M$ implies the vanishing of the cohomology
  sheaves $\mathcal{H}^i(M)$ in all degrees $i>0$ and $i<-d$. By the
  definition of weights for perverse sheaves, if~$M$ is pure of weight
  zero, then $\mathcal{H}^i(M)$ is pointwise mixed of weight $\leq i$,
  which means that the eigenvalues of Frobenius acting on the stalks
  $\mathcal{H}^i(M)_x$ all have modulus $\leq |\Ff|^{i/2}\leq 1$. Therefore, the inequality 
  \[
  |t_M(x)| \leq \sum_{-d \leq i \leq 0} |t_{\mathcal{H}^i(M)}(x)| \leq \sum_{-d \leq i \leq 0} \dim \mathcal{H}^i(M)_x
  \] holds, and it suffices to estimate the dimensions of the stalks of $\mathcal{H}^i(M)$. From~(\ref{eq-6pullback}) applied to the morphism $x\colon \Spec(\Ff)\to X$, we derive the estimates 
  \begin{equation}\label{eqn:complexityboundsgenericrank}
  \dim \mathcal{H}^i(M)_x\leq c_{u\circ x}(x^*M) \ll c_{u\circ x, u}(x)
  c_u(M)=c(u\circ x)c_u(M)= c_u(M),
  \end{equation}
  and hence the estimate $t_M(x) \ll c_u(M)$ as claimed.

  (2) If we assume, moreover, that $M$ is a geometrically irreducible perverse
  sheaf with support $X$,  by~\cite[Th.\,4.3.1\,(ii)]{BBDG} there exists a dense open subset
  $j\colon U\to X$ and a lisse sheaf $\mathcal{F}$ on $U$ such that $M$
  is the middle extension $j_{!*}\mathcal{F}[d]$. Then $\mathcal{H}^0(M)$ vanishes
  by~\hbox{\cite[Prop.\,2.2.4]{BBDG}}, so that the eigenvalues of Frobenius have modulus \hbox{$\leq |\Ff|^{-1/2}$} and the same argument as in~(1) yields the stronger estimate 
  $$
  |t_M(x) |\ll c_u(M) {|\Ff |}^{-1/2}.
  $$ 
  
  (3) By Theorem~\ref{singloc}, there exists a stratification of $X$,
  defined over~$\Ff$, into locally closed strata $S_j$ such that the
  restriction of~$\mathcal{H}^i(M)$ to each $S_j$ is lisse and the
  degree of $u(S_j)$ is bounded in terms of $c_u(M)$. By
  \cite[Prop.\,2.2.4]{BBDG}, there is an inequality
  $\dim \mathrm{Supp}(\mathcal{H}^i(M))<-i-1$ for all $-i<d$. Since
  $\mathcal{H}^i(M)$ restricts to a lisse sheaf on $S_j$, we
  must have $d_j<-i$ whenever $\mathcal{H}^i(M)$ is non-zero and~$d_j\neq d$.
\end{proof}

\begin{remark}
  The following example shows that the estimates in
  Proposition~\ref{pr-bound} are best possible in general. Assume that the finite field satisfies $|\Ff| \equiv 1\pmod{3}$. Let $\chi$ be a non-trivial
  multiplicative character of order three of $\Ff$ and let
  $\mathcal{L}_\chi$ be the associated Kummer sheaf on~$\Gg_m$. Denote by
  $j\colon U\to \Aa^n$ the inclusion of the complement of the hypersurface $F \subset \Aa^n$
  given by the equation
  $$
  x_1^3+\dots +x_n^3=0.
  $$
  Let $\mathcal{F}$ be the lisse sheaf
  $\mathcal{L}_{\chi(x_1^3+\cdots+x_n^3)}(-n/2)$ on $U$ and let $M=j_{!*}(\mathcal{F}[n])$ be its middle extension to $\Aa^n$, shifted to make it perverse. Then $M$ is a
  geometrically irreducible perverse sheaf of weight zero and the equality $| t_M(x) |=|\Ff |^{-n/2}$ holds for all $x\in U(\Ff)$.  We claim that, at the origin, the trace function of $M$ satisfies 
  $$
  |t_M(0) |=|\Ff |^{-1/2}.
  $$ 
  
  To see this, let $h\colon X'\to \Aa^n$ denote the blow-up of $\Aa^n$
  at the origin, \hbox{$E\simeq \Pp^{n-1}$} the exceptional divisor, $F'$ the
  strict transform of $F$, and $j'\colon U'\to X'$ the inclusion of the
  inverse image $U'=h^{-1}(U)$ of $U$. We define
  $\mathcal{F}'=h^*\mathcal{F}$ and
  $M'=j'_{!*}(\mathcal{F}'[n])$. Since~$\chi$ is a character of order
  three, a local computation shows that $\mathcal{F}'$ extends to a
  lisse sheaf outside~$F'$ (for example, in the affine chart given by
  $y_1=1$ and \hbox{$y_i=x_i/x_1$} for $i=2,\dots, n$, this follows from the
  equality
  $\mathcal{L}_{\chi(x_1^3+\cdots+x_n^3)}=\mathcal{L}_{\chi(x_1^3(1+y_2^3+\cdots+y_n^3))}=\mathcal{L}_{\chi(1+y_2^3+\cdots+y_n^3)}$). Since~$F'$ is a smooth divisor, the perverse sheaf $M'$ is the extension by
  zero of the shift of this lisse sheaf to $X'$. Hence, its trace
  function at a point $y=[y_1\colon\cdots\colon y_n]\in E(\Ff)$ equals
  \[
  t_{M'}(y)=\begin{cases} \chi(y_1^3+\dots + y_n^3) & \text{if }y_1^3+\cdots+ y_n^3\neq 0, \\
  0 & \text{otherwise}. 
  \end{cases} 
  \]

  By the decomposition theorem \cite{BBDG}, the perverse sheaf $M$ is a
  direct factor of $h_!(M')$, and since $h$ is an isomorphism outside
  the origin, the other irreducible components of~$h_!(M')$ are
  skyscraper sheaves supported at zero. By the Grothendieck--Lefschetz
  trace formula, the trace function of $M$ at $0$ is equal to
  $$
  \frac{1}{|\Ff |^{n/2}} \sum_{\underset{y_1^3+\dots + y_n^3\neq
      0}{[y_1:\cdots :y_n]\in \Pp^{n-1}(\Ff)}}\chi(y_1^3+\cdots +
  y_n^3), 
  $$ which up to the normalizing factor $|\Ff |^{-n/2}$ is a sum of Jacobi sums of weight $n-1$ 
by a classical computation (see \cite[Ch.\,8, Th.\,5]{IrelandRosen}). Weyl's equidistribution theorem implies that this sum does not cancel, at least for some extensions of any given finite field.
\end{remark}

\begin{theo}[Quasi-orthogonality]\label{th-rh}
  Let~$\Ff$ be a finite field of characteristic different from~$\ell$. Let~$(X,u)$ be a geometrically irreducible
  quasi-projective algebraic variety of dimension $d$ over~$\Ff$.
  \par
  \emph{(1)} Assume that~$X$ is smooth.  Let~$A$ and~$B$ be
  $\ell$-adic constructible sheaves on~$X$. Suppose that~$A$ and~$B$
  are mixed of weights~$\leq 0$ and that there exists a dense open
  subset~$U$ of~$X$ on which~$A$ and~$B$ are lisse sheaves, pure of
  weight zero, and geometrically irreducible. Then there exists a complex
  number~$\alpha$ of modulus $1$ such that the estimate
  \begin{align*}
  \Bigl|\frac{1}{|\Ff|^{d}}\sum_{x\in
    X(\Ff)}t_A(x)\overline{t_B(x)}- \alpha \delta(A,B) \Bigr| \ll
  c(u)&c_u(A)c_u(B)|\Ff|^{-1/2} \\ 
  &+c(u)(c_u(A)+c_u(B))|\Ff|^{-1}
  \end{align*}
 holds, with~$\delta(A,B)=1$ if $A$ and~$B$ are geometrically isomorphic
  over~$U$ and $\delta(A,B)=0$ otherwise.
  \par
  \emph{(2)} Let~$A$ and~$B$ be geometrically irreducible $\ell$-adic
  perverse sheaves on~$X$. Suppose that~$A$ and~$B$ are pure of
  weight zero. Then there exists a complex number~$\alpha$ of modulus~$1$ such that the estimate 
  $$
  \Bigl|\sum_{x\in
    X(\Ff)}t_A(x)\overline{t_B(x)}- \alpha \delta(A,B) \Bigr| \ll
  c(u)c_u(A)c_u(B)|\Ff|^{-1/2}
  $$
 holds, with~$\delta(A,B)=1$ if $A$ and~$B$ are geometrically isomorphic and $\delta(A,B)=0$.
  \par
  In both estimates, the implied constant only depends on the embedding
  dimension of~$(X, u)$.
\end{theo}

\begin{proof} We start with (1). Since~$B$ restricts (up to a shift and a Tate twist) to a perverse sheaf of weight zero on $U$, the proof of Proposition~\ref{pr-conjugation} shows that $\overline{t_B}$ coincides with the trace function of $\dual(B)$ on $U(\Ff)$. Since $A$ and $B$ are geometrically irreducible, the object $A \otimes \dual(B)$ has non-trivial cohomology with compact support in top degree $2d$ if and only if $A$ and $B$ are geometrically isomorphic, and in that case this cohomology is one-dimensional. By \cite[Lem.\,1.8.1]{mmp}, if $A$ and $B$ are geometrically isomorphic, there is an isomorphism $B \simeq A \otimes \alpha^{\deg}$ for a unique $\alpha \in \Ql$, which has modulus $1$ when viewed as a complex number. Arguing as in the proof of Katz's orthogonality theorem~\hbox{\cite[Th.\,1.7.2]{mmp}}, the combination of the Grothendieck--Lefschetz trace formula and Deligne's Riemann Hypothesis thus gives the estimate 
\[
 \Bigl|\frac{1}{|\Ff|^{d}}\sum_{x\in
    U(\Ff)}t_A(x)\overline{t_B(x)}-\alpha\delta(A,B)\Bigr|\leq \sigma  |\Ff|^{-1/2}, 
\] where $\sigma$ denotes the sum of Betti numbers
  $$
  \sigma=\sum_{i\in\Zz} h^i_c(U_{\overline{\Ff}}, A\otimes \dual(B)).
  $$ By Proposition~\ref{pr-betti} and~\eqref{eq-6tensor}, this quantity is bounded by 
  \[
  \sigma \ll c(A\otimes \dual(B)) \ll c_u(A)c_u(\dual(B)) \ll c(u)c_u(A)c_u(B). 
  \]

On the other hand, on the complement $X-U$, the functions $t_A$ and $\overline{t_B}$ are bounded by the generic rank of $A$ and $B$ respectively, which are in turn bounded by their complexity taking \eqref{eqn:complexityboundsgenericrank} into account. This gives the estimate
\[
 \Bigl|\frac{1}{|\Ff|^{d}} \sum_{x\in
    (X-U)(\Ff)}t_A(x)\overline{t_B(x)}\Bigr| \ll c_u(A)c_u(B)|(X-U)(\Ff)| |\Ff|^{-d}.
\] We then note that~$U$ is contained in the intersection of the maximal open
  set~$U_1$ where~$A$ is lisse and the maximal open set~$U_2$ where~$B$
  is lisse, and hence the estimates 
  $$
  \deg(X-U)\ll \deg(X-U_1)+\deg(X-U_2) \ll c(u)(c_u(A)+c_u(B))
  $$
 hold by Theorem~\ref{singloc} (recall that~$X$ is assumed to be smooth).  We
  conclude using the classical bound
  $|(X-U)(\Ff)|\leq \deg(X-U)|\Ff|^{\dim(X-U)}$.
  
  The proof of (2) is the same, except that the term $|\Ff|^{-d}$ does not appear due to the normalization of weights for perverse sheaves (see Remark \ref{rem:weightsforperverse}) and that we do not need to treat the sum over points of $X-U$. 
\end{proof}

Taking $X=\Aa^n$, the second statement immediately implies Theorem~\ref{th-sample-Riemann}.

\begin{remark}
  In the setting of part~(2) of Theorem~\ref{th-rh}, suppose that one
  knows that there is~$\beta\geq 0$ such that
  $$
  \rmH^i_c(X_{\overline{\Ff}},A\otimes\dual(B))=0
  $$
  for~$i\geq -\beta$ (the case $\beta=0$ corresponds to the assumption
  that~$A$ and~$B$ are not geometrically isomorphic).  Then by the
  same argument and the Riemann Hypothesis, we obtain the stronger
  estimate
  $$
  \sum_{x\in X(\Ff)}t_A(x)\overline{t_B(x)} \ll
 c(u) c_u(A)c_u(B)|\Ff|^{-(\beta+1)/2}.
  $$
The case where~$\beta=\dim(X)-1$ corresponds to full square-root
  cancellation for the sum over~$X(\Ff)$.
\end{remark}

A basic finiteness statement follows from this result:

\begin{cor}\label{cor-finiteness}
  Let~$\Ff$ be a finite field of characteristic different from~$\ell$. Let~$(X,u)$ be a geometrically irreducible quasi-projective
  algebraic variety over~$\Ff$.
  \par
  For any $c\geq 1$, there are, up to geometric isomorphism, only
  finitely many $\ell$-adic perverse sheaves~$A$ on~$X$ of complexity
  $c_u(A)\leq c$ that are geometrically irreducible and pure of
  weight zero.
\end{cor}

\begin{proof}
  By Theorem~\ref{th-rh}, applied to finite extensions of~$\Ff$, we
  first see that there exists a finite extension~$\Ff_n$ of~$\Ff$
  (depending on~$c$) such that the equality of trace functions
  $t_A(\cdot;\Ff_n)=t_B(\cdot;\Ff_n)$ implies that~$A$ and~$B$ are
  geometrically isomorphic, for irreducible perverse sheaves~$A$ and~$B$
  of weight zero on~$X$ with~$c_u(A)\leq c$ and~$c_u(B)\leq c$. Moreover,
  we can also ensure that~$X(\Ff_n)$ is not empty and
  that~$t_A(\cdot;\Ff_n)\not=0$ when $c_u(A)\leq c$.
  \par
  Replacing~$\Ff$ by~$\Ff_n$, it is then enough to prove that there are
  only finitely many functions~$t_A$ for irreducible perverse sheaves of
  weight zero with~$c_u(A)\leq c$, up to geometric isomorphism.
  Let~$C(X(\Ff_n))$ be the finite-dimensional Hilbert space of functions
  $X(\Ff_n)\to\Cc$ with norm
  $$
  \|f\|=\sum_{x\in X(\Ff_n)}|f(x)|^2.
  $$
  Theorem~\ref{th-rh} again implies that for~$A$ not geometrically
  isomorphic to~$B$ with $c_u(A)\leq c$ and~$c_u(B)\leq c$, the
  functions~$t_A/\|t_A\|$ and~$t_B/\|t_B\|$, viewed as elements of the
  unit sphere of~$C(X(\Ff_n))$, make an angle $\geq \theta$ for
  some~$\theta>0$ independent of~$A$ and~$B$. It is well-known that
  there can only be finitely many such vectors.
\end{proof}

\begin{remark}
  This argument can be made quantitative, although it is probably far
  from sharp; see~\cite{fkm-counting} for the case of curves.
\end{remark}

\subsection{Examples}

We collect here for ease of reference some immediate corollaries of
the formalism of Section~\ref{sec-6ops}. These contain and generalize
all the basic ``continuity'' estimates of Fouvry, Kowalski and Michel
in the case of curves.

\begin{example}[Fourier Transform] Let~$\Ff$ be a finite field.
  Let~$n\geq 1$ be an integer and let~$X=\Aa^n$ viewed as a
  commutative algebraic group over~$\Ff$, with the obvious embedding
  in~$\Pp^n$ to define the complexity. Let~$\psi$ be a fixed non-trivial $\ell$-adic additive character
  of~$\Ff$ and consider the Artin\nobreakdash--Schreier sheaf
  $ \mcL=\mcL_{\psi(x\cdot y)}$ of rank~$1$ on~$\Aa^n\times \Aa^n$,
  where $x\cdot y$ is the standard scalar product. Deligne defined the $\ell$-adic Fourier transform~$\ft_{\psi}$ as the
  functor~$\Der(\Aa^n)\to \Der(\Aa^n)$ given by
  $$
  \ft_{\psi}(A)=p_{2,!}(p_1^*A\otimes \mcL)=p_{2,*}(p_1^*A\otimes
  \mcL)
  $$
  (see~\cite[1.2.1.1]{laumon}; the equality of the two expressions is
  a highly non-trivial fact, often referred to as ``the miracle of the Fourier transform'').

  From the results of Section~\ref{sec-6ops} and
  Corollary~\ref{cor-as}, we therefore deduce:

  \begin{prop}\label{pr-fourier}
    There exists an integer~$N\geq 0$, depending only on~$n$, such
    that
     \[
    c(\ft_{\psi}(A))\leq Nc(A)
    \]
    holds for any object~$A$ of~$\Der(\Aa^n)$.
  \end{prop}

  The results of Section~\ref{sec-effective}, together with
  Propositions~\ref{pr-explicit-uniform} and~\ref{pr-as}, lead to explicit
  estimates for~$N$. These are growing at least as fast as~$n!$, and it
  might be interesting to have a better estimate for this ``norm'' of the
  $\ell$-adic Fourier transform.

  For $n=1$,  Fouvry, Kowalski and
  Michel proved the inequality
  $$
  c_{\FKM}(\ft_{\psi}(A))\leq 10 c_{\FKM}(A)^2
  $$
  (for middle-extension Fourier sheaves) in~\cite[Prop.\,8.2]{fkm1}; this
  estimate plays an essential role in many analytic applications, and one
  can expect a similar use of Proposition~\ref{pr-fourier}.
\end{example}

\begin{example}[Other cohomological transforms] Let $(X,u)$ and
  $(Y,v)$ be quasi-pro\-jective varieties over~$k$ and $K$ an object of~$\Der(X\times Y)$. Let \hbox{$T_K\colon \Der(X)\to \Der(Y)$} denote the
  ``cohomological transform with kernel~$K$'', i.e., the functor such
  that
  $$
  T_K(A)=p_{2,!}(p_1^*A\otimes K),
  $$
  where $p_1\colon X\times Y\to X$ and~$p_2\colon X\times Y\to Y$ are
  the projections.

  Applying again the formalism, there exists a constant~$N_K$,
  depending on~$K$, such that the inequality 
  $$
  c_v(T_K(A))\leq N_K c_u(A)
  $$
  holds for any object~$A$ of $\Der(X)$. Precisely, this holds with
  $$
  N_K=c_{u\boxtimes v,v}(p_2)c_{u\boxtimes v,v}(p_1)c_{u\boxtimes v}(K)
  $$
  (where~$u\boxtimes v$ is used to denote the composition of $u\times v$
  with the Segre embedding of the product of the projective spaces
  target of~$u$ and~$v$; as in the case of the Fourier transform,
  other embeddings of~$X\times Y$ might be possible).

  In the very special case where~$X=Y=\Aa^1$ (over finite fields) and
  $K$ is a rank~$1$ Artin--Schreier or Kummer sheaf on~$\Aa^2$, we can
  apply Corollary~\ref{cor-as} to estimate the complexity of the
  kernel sheaf; a weaker form of the resulting estimate was proved by
  Fouvry, Kowalski and Michel~\cite[Th.\,2.3]{fkm-toulouse}.

  Other special cases that have been already considered (when~$n=1$)
  for the varieties $X=Y=\Aa^1$ or~\hbox{$X=Y=\Gg_m$} are additive convolution and
  multiplicative convolution.

  More generally, let~$G$ be a commutative quasi-projective algebraic
  group over~$k$, with a given locally closed embedding~$u$.
  Let~$\sigma\colon G\times G\to G$ be the addition morphism. We have
  two convolution functors
  $$
  *\colon \Der(G)\times \Der(G)\to\Der(G),\quad\quad *_!\colon
  \Der(G)\times \Der(G)\to \Der(G)
  $$
  defined by
  $$
  A*B=\sigma_*(A\boxtimes B),\quad\quad A*_!B=\sigma_!(A\boxtimes B).
  $$

  Use the composition of the Segre embedding with $u\times u$ to embed
  $G\times G$ in projective space.\footnote{\ For certain groups, such
    as affine groups, one can also use other embeddings.} Then
  by~\eqref{eq-6pushforward} and~\eqref{eq-6extzero} and
  Proposition~\ref{pr-external}, there exists an integer~$N$
  such that the inequalities
  $$
  c_u(A*B)\leq Nc_u(A)c_u(B),\quad\quad c_u(A*_!B)\leq N c_u(A)c_u(B)
  $$
  hold for all objects~$A$ and~$B$ of~$\Der(G)$.

\end{example}

\begin{example}[Gowers uniformity sheaves] Let~$(G, u)$ be a
  commutative quasi-pro\-jective algebraic group over a field~$k_0$,
  with the group law written additively. For each integer $d\geq 1$, we denote by $\pi_d \colon G^{d+1}\to G$ the projection $(x,h_1,\ldots,h_d)\mapsto x$. Given a complex~$A$ of $\Der(G)$, we define the \emph{uniformity object} $U_d(A)$ as
  \[
  U_{d}(A)=\pi_{d,!}\Bigl(\bigotimes_{I} \dual^{\circ |I|}(\sigma_I^*A)\Bigr), 
  \]
  where the tensor product runs over all subsets $I \subset \{1,\ldots,d\}$. For each of them, $\dual^{\circ |I|}$ denotes the Verdier duality functor if $|I|$ is odd and the identity functor if~$|I|$ is even, and the morphism $\sigma_I\colon G^{d+1} \to G$ is given by 
  \[
  \sigma_{I}(x,h_1,\ldots,h_d)=x+\sum_{i\in I}h_i. 
  \] 
  
  The point of this construction is that, in the case where $k_0$ is a finite field $\Ff$ and $A$ is a perverse sheaf pure of weight zero, the trace function of~$U_d(A)$ satisfies  
   \[
 \sum_{x\in\Ff}t_{U_d(A)}(x)=\|t_A\|_{U_d}^{2^d}, 
  \]
  where~$\|\cdot \|_{U_d}$ is the $d$-Gowers norm for a complex-valued
  function on the finite group~$G(\Ff)$, see \cite[Def.\,11.2]{taovu}.   For~$G=\Gg_a$, this construction is implicitly used
  in~\cite{fkm-gowers} to obtain an ``inverse theorem'' for Gowers
  norms of one-variable trace functions, and various results from \loccit can actually be interpreted as bounds for the complexity of $U_d(A)$ in terms of that of~$A$. Thanks to Theorem \ref{pushpull}, we can generalize them to any group: there exists an integer $N_d\geq 0$ such that the inequality 
  \[
  c_u(U_d(A))\leq N_dc_u(A)^{2^d} 
 \] holds for all objects $A$ of $\Der(G)$. 
  
\end{example}

\begin{example}[Sum-product sheaves] In~\cite{KMS} and~\cite{KMS2},
  special cases of the following construction are applied to
  Kloosterman sheaves: given an input object~$A$ in~$\Der(\Aa^1)$ and
  an integer $l\geq 1$, one constructs a ``sum-product'' object 
  $$
  B_l=\bigotimes_{1\leq i\leq l}f_i^*A\otimes f_{i+l}^*\dual(A)
  $$
  on $\Aa^{2+2l}$, with coordinates $(r,s,\uple{b})$, by
means of the functions $f_i(r,s,\uple{b})=r(s+b_i)$. It follows from the general formalism that, performing this
  construction with \emph{any} input object $A$ in~$\Der(\Aa^1)$, we
  have
  $$
  c_v(B_l)\leq N_l c_u(A)^{2l}
  $$
  for some constant $N_l$ depending only on~$l$ (the embeddings are
  the standard embeddings $u\colon \Aa^1\to \Pp^1$ and
  $v\colon \Aa^{2l+1}\to \Pp^{2l+1}$).
\end{example}

\subsection{Equidistribution}

Using the theory of complexity developed in this paper, the
``horizontal'' version of Deligne's~Equidistribution Theorem
formulated by Katz in the case of curves~\cite[Ch.\,3]{gksm} admits a
straightforward extension to higher-dimensional varieties. As before, we fix an isomorphism $\iota \colon \Ql \to \Cc$ through which both fields are
identified. 

\begin{theo}\label{th-katz}
  Let~$N\geq 1$ be an integer and let~$(X,u)$ be a smooth and connected
  quasi\nobreakdash-projective scheme over~$\Spec(\Zz[1/N])$ with geometrically
  irreducible generic fiber. For each prime $p$ not dividing $N$, let
  $(X_{\Ff_p}, u_p)$ denote the reduction of $(X, u)$ modulo~$p$. Let~$\mathcal{P}$ be an infinite set of primes not
  dividing~$N\ell$, and assume that we are given, for
  each~$p\in\mathcal{P}$, a lisse $\ell$-adic sheaf~$\mathcal{F}_p$ on $X_{\Ff_p}$
  that is pure of weight zero and whose complexity satisfies \hbox{$c_{u_p}(\mathcal{F}_p)\ll 1$}.
  \par
  Assume that, for each $p\in\mathcal{P}$, the geometric and
  arithmetic monodromy groups of~$\mathcal{F}_p$ coincide and are
  isomorphic to a fixed semisimple (i.e., the connected
    component of the identity is semisimple) algebraic
  group~$G$. Let~$K$ be a maximal compact subgroup of~$G(\Cc)$.
  \par
  For $p\in\mathcal{P}$ and $x\in X(\Ff_p)$, let~$\theta_{p}(x)$
  be the unique conjugacy class in~$K$ that intersects the conjugacy
  class of the semisimplification of the image in~$G$ of the Frobenius
  at~$x$ relative to~$\Ff_p$, so that the equality 
  $$
  \Tr(\theta_p(x))=t_{\mathcal{F}_p}(x)
  $$
  holds. Then the families $(\theta_p(x))_{x\in X(\Ff_p)}$ become
  equidistributed as $p\to +\infty$ in the space of conjugacy classes
  of~$K$, with respect to the Haar probability measure.
\end{theo}

\begin{proof}
  Let~$d$ be the relative dimension of~$X$
  over~$\Spec(\Zz[1/N])$. Shrinking $\mathcal{P}$ and~$X$ if necessary,
  we may assume that $X_{\Ff_p}$ is a smooth and geometrically connected
  variety of dimension~$d$ and that $X(\Ff_p)$ is not empty for all
  primes $p$ in $\mathcal{P}$. In fact, the number of $\Ff_p$-points
  behaves asymptotically like $|X(\Ff_p)|\sim p^d$ as $p\to +\infty$ by
  the Lang--Weil estimate.
  \par
  Each lisse $\ell$-adic sheaf $\mathcal{F}_p$ corresponds to a representation of
  the fundamental group~$\pi_{1}(X_{\Ff_p})$ of~$X_{\Ff_p}$. According to the Weyl Criterion, equidistribution follows from the vanishing
  \begin{equation}\label{eq-weyl}
    \lim_{p\to+\infty}
    \frac{1}{|X(\Ff_p)|}\sum_{x\in X(\Ff_p)}
    \Tr(\rho(\theta_p(x)))=0
  \end{equation}
  for any non-trivial irreducible representation~$\rho$ of~$K$. By the
  correspondance between semisimple algebraic groups and compact Lie
  groups, such a representation~$\rho$ corresponds uniquely to an
  algebraic irreducible representation of the algebraic group~$G$,
  which is still denoted by~$\rho$. The lisse sheaf~$\rho(\mathcal{F}_p)$
  on~$X_{\Ff_p}$ (see Section~\ref{sec-tannaka}) satisfies
  $$
  \Tr(\rho(\theta_p(x)))=t_{\rho(\mathcal{F}_p)}(x)
  $$
  for all~$x\in X(\Ff_p)$. Moreover, this sheaf is of weight zero (the group $G$ being semisimple, it is a subsheaf of some tensor power $\mathcal{F}^{
\otimes a}$ by Lemma \ref{lm-tannaka}), geometrically irreducible (since its arithmetic and geometric monodromy group are equal), and non-trivial. By the
  Riemann Hypothesis (Theorem~\ref{th-rh}\,(1)), applied with~$B=\Ql$,
  the estimate 
  $$
  \frac{1}{p^d} \sum_{x\in X(\Ff_p)}\Tr(\rho(\theta_p(x))) \ll
  c_{u_p}(\rho(\mathcal{F}_p))p^{-1/2}
  $$
  holds for all~$p\in\mathcal{P}$. By Proposition~\ref{pr-tannaka}, there
  exists an integer~$a\geq 0$, depending only on~$\rho$, such that
  $$
  c_{u_p}(\rho(\mathcal{F}_p))\ll c_{u_p}(\mathcal{F})^a.
  $$
  Since we assumed that
  $c_{u_p}(\mathcal{F}_p)$ is bounded, this estimate implies~\eqref{eq-weyl}.
\end{proof}

In practice, some multi-variable cases of this theorem could be proved
by applying the Deligne--Katz equidistribution theorems to families of
curves covering~$X$.
\par
Katz--Sarnak~\cite[\S 9.6]{katz-sarnak} and Katz~\cite[Ch.\,12]{mmp}
have proved earlier statements of the same kind, but assuming
that~$\mathcal{F}_p$ is the base change to~$\Ff_p$ of a ``common''
sheaf or perverse sheaf on~$X$ over $\Zz[1/N]$.
\par
In fact, we now show that Theorem~\ref{th-katz} gives a positive
answer to the question of Katz~\cite[p.\,8 and 12.6.6]{mmp} concerning
equidistribution of certain higher-dimensional families of additive
character sums. This illustrates that the theory of complexity can, to
a certain extent, obviate the lack of a theory of exponential sums
over~$\Zz$ (the second part implies Theorem~\ref{th-sample-2} from the
introduction).

Let~$n \geq 1$ and $d \geq 1$ be integers. Let $P(n,d)$ be the space of
polynomials of degree~$d$ in~$n$ variables, and $P(n,d,\text{odd}) \subset P(n,d)$ the
subspace consisting of odd polynomials, by which we mean that only
monomials of odd degree have non-zero coefficients. For a prime number~$p$ and~$f\in P(n,d)(\Ff_p)$, set
\[
  S(f;p)=\frac{1}{p^{w(f)/2}}\sum_{x\in\Ff_p^n}e\Bigl(\frac{f(x)}{p}\Bigr),
\]
where~$w(f)$ is the smallest integer such that the vanishing
\[
\rmH^i_c(\Aa^n_{\overline{\Ff}_p},\mathcal{L}_{\psi(f)})=0
\]
holds for all~$i>w(f)$, and we recall the notation $e(z)=\exp(2i\pi z)$.

For an even integer $N$, we denote by 
\[
\USp_N(\Cc)=\Un_N(\Cc) \cap \Sp_N(\Cc)
\] the group of unitary
symplectic matrices of size~$N$, which is a maximal compact subgroup of the
symplectic group~$\Sp_N(\Cc)$.

\begin{cor}\label{cor-katz}
  Let~$n \geq 1$ and $d \geq 1$ be
  integers. Set~$K_n=\Un_{(d-1)^n}(\Cc)$ and, for odd~$d$,
  \[
  K_{n,odd}=\begin{cases}
  \USp_{(d-1)^n}(\Cc) & \text{if~$n$ is odd}, \\
  \Ort_{(d-1)^n}(\Cc) & \text{if~$n$ is even}.
  \end{cases}
 \]
 \begin{enumerate}
 \item\label{cor-katz:item1} The families
   $(S(f;p))_{f\in P(n,d)(\Ff_p)}$ become equidistributed as
   $p\to+\infty$ with respect to the measure which is the image under
   the trace of the probability Haar measure on $K_n$.

 \item\label{cor-katz:item2} Suppose that~$d$ is odd. The families
   $(S(f;p))_{f\in P(n,d,\text{odd})(\Ff_p)}$ become equidistributed as
   $p\to+\infty$ with respect to the measure which is the image under
   the trace of the probability Haar measure on $K_{n,odd}$.
 \end{enumerate}
\end{cor}

\begin{proof}
  Note that $P(n, d)$ is a dense open subset of the affine space of
  polynomials of degree $\leq d$ in $n$ variables.
  Let~$D(n,d)\subset P(n,d)$ denote the dense open subset of Deligne
  polynomials (namely, those for which the homogeneous part of highest
  degree defines a smooth hypersurface in~$\Pp^{n-1}$) and let
  $D(n,d,\text{odd}) \subset P(n,d,\text{odd})$ be the subset of odd
  Deligne polynomials. Both~$D(n,d)$ and $D(n,d,\text{odd})$ are smooth
  schemes over~$\Zz$.
  \par
  Because (by the Riemann Hypothesis) the estimate
  $$
  |S(f;p)|\leq \sum_{i\leq
    w(f)}h^i_c(\Aa^n_{\overline{\Ff}_p},\mathcal{L}_{\psi(f)})\ll
  c(\mathcal{L}_{\psi(f)})\ll 1
  $$
  holds for all~$p$ and~$f\in P(n,d)(\Ff_p)$, it is enough to prove the
  equidistribution of the sums~$S(f; p)$ for~$f\in D(n,d)(\Ff_p)$ (\resp
  $f\in D(n,d,\text{odd})(\Ff_p)$).

  We first handle separately the cases~$d=1$ and~$d=2$. If~$d=1$, the
  sum~$S(f;p)$ vanishes for any non-zero linear polynomial~$f$. Since
  $K_0=\Un_0(\Cc)$ is the trivial group, whose only element has trace
  zero, the equidistribution holds in that case.
  
  If~$d=2$, then~$K_2=\Un_1(\Cc)$ is the unit circle. Write a
  polynomial $f\in D(n,2)(\Ff_p)$ as
  \[
  f(x)=Q(x)+\lambda(x)+\mu, 
  \]
  where~$Q$ is a non-degenerate quadratic form, $\lambda$ a linear
  form and~$\mu$ a constant. For each prime~$p\geq 3$, it is an elementary consequence of the fact that normalized
  Gauss sums have modulus one that~$|S(f;p)|=1$. Let~$h\geq 1$ be an
  integer. Then
  \[
  \sum_{f\in D(n,2)(\Ff_p)}S(f;p)^h= \frac{1}{p^{nh/2}} \sum_{x_1,\ldots,
    x_h\in \Ff_p^n} \sum_{Q,\lambda}
  e\Bigl(\frac{1}{p}\Bigl(\sum_i(Q(x_i)+\lambda(x_i)\Bigr)\Bigr)
  \sum_{\mu \in\Ff_p}e\Bigl(\frac{h\mu}{p}\Bigr).
  \]
  This vanishes as soon as~$p>h$. If~$h\leq -1$, we obtain the same
  conclusion after noting that
  $S(f;p)^{-1}=\overline{S(f;p)}=S(-f;p)$. Thus, 
  $$
  \lim_{p\to+\infty} \frac{1}{|D(n,2)(\Ff_p)|}\sum_{f\in
    D(n,2)(\Ff_p)} S(f;p)^h=0
  $$
  for any non-zero integer~$h$, which proves equidistribution
  in~$\Un_1(\Cc)=K_2$.
  \par
  We now assume that~$d\geq 3$. For a prime~$p$,
  let~$\psi$ be an $\ell$-adic character with trace function
  $x\mapsto e(x/p)$, under the identification of~$\Ql$ and~$\Cc$. On
  $\Aa^n\times D(n,d)$ over~$\Ff_p$ with coordinates~$(x,f)$ we have the
  Kummer sheaf~$\mcL_{\psi(f(x))}$.  Define
  \[
  \mathcal{D}_p=p_{2,!}\mcL_{\psi(f(x))}[n](n/2),
  \]
  where $p_2 \colon \Aa^n\times D(n,d) \to D(n,d)$ denotes the
  projection to the second factor, so that the equality
  \[
    S(f;p)=(-1)^{n}t_{\mathcal{D}_p}(f)
  \]
 holds by the trace formula.  This is \emph{a priori} an object of the
  derived category of~$D(n,d)$ over~$\Ff_p$, but Deligne has shown (see \cite[Cor.\,3.5.11, Cor.\,3.5.12]{mmp}) that~$\mathcal{D}_p$ is in fact a lisse sheaf of rank $N=(d-1)^n$ and pure
  of weight zero for each prime~$p$ not dividing $d$. 
  \par
  We first prove~\eqref{cor-katz:item2}, which fits exactly the
  statement of Theorem~\ref{th-katz}. The relevant monodromy groups have
  been computed by Katz~\cite[Th.\,12.6.3]{mmp}: if~$p\geq 7$ and
  $p\nmid d$, the geometric and arithmetic monodromy groups of the
  restriction~$\mathcal{F}_p$ of~$\mathcal{D}_p$ to~$D(n,d,\text{odd})$
  coincide and are equal to $\Sp_{N}$ if~$n$ is odd and to~$\Ort_N$
  if~$n$ is even (precisely, these references show that the geometric
  monodromy group is as stated, but by~\cite[Th.\,3.10.6]{mmp}, these
  sheaves are arithmetically self-dual, so their arithmetic monodromy
  groups cannot be bigger). Hence, $K_{n,odd}$ is a maximal compact
  subgroup of the complex points of the geometric monodromy group for any~$p$.  Moreover,
  denoting by~$i$ the closed immersion $D(n,d,\text{odd})\to D(n,d)$,
  the estimate
  $$
  c(\mathcal{F}_p)=c(i^*p_{2,!}\mcL(\psi(f(x)))) \ll
  c(i)c(p_2)c(\mcL_{\psi(f(x))})
  $$
  holds by~(\ref{eq-6pullback}) and~(\ref{eq-6extzero}) for any
  prime~$p$, where the complexities are computed using the embeddings
  of~$\Aa^n$, $D(n,d)$ and their product and subschemes are induced by
  the natural embeddings of affine spaces in the projective space of the
  same dimension.
  \par
  By the bound in Section~\ref{sec-ask} and by Theorem~\ref{uniformity},
  we deduce that $c(\mathcal{F}_p)\ll 1$ for all~$p$, where the implied
  constant depends only on~$(d,n)$. Thus, the equidistribution
  in~\eqref{cor-katz:item2} follows from Theorem~\ref{th-katz}.
  \par
  We now come back to~\eqref{cor-katz:item1}. Here we note that, as
  above, we have $c(\mathcal{D}_p)\ll 1$ for all~$p$, but the setting is
  not exactly that of Theorem~\ref{th-katz}. Indeed, Katz
  proved in~\cite[Th. 6.8.34]{mmp} that the geometric monodromy group
  of~$\mathcal{D}_p$, for~$p\geq 7$ that does not divide $d(d-1)$, is the
  group
  $$
  G_{2p}=\{g\in \GL_N\,\mid\, \det(g)^{2p}=1\},
  $$
  which depends on~$p$. We argue then by repeating the use of the Weyl
  Criterion in the proof of Theorem~\ref{th-katz}.
  \par
  For $f\in D(n,d)(\Ff_p)$, we denote by $\theta_p(f)$ the conjugacy
  class in~$G_{2p}$ corresponding to the Frobenius at~$f$, whose trace
  is equal to $S(f;p)$. Let~$\rho$ be a non-trivial irreducible
  representation of~$\Un_N(\Cc)$, which can also be viewed as a
  representation of~$\GL_N$. The lisse sheaf $\rho(\mathcal{D}_p)$
  satisfies $c(\rho(\mathcal{D}_p))\ll 1$ for all~$p$ (as in
  \loccit). The restriction of~$\rho$ to~$G_{2p}$ is a direct sum of a
  bounded number of irreducible representations. We claim that if~$p$ is
  large enough, depending on~$\rho$, then this restriction does not
  contain the trivial representation of~$G_{2p}$. Indeed, by Frobenius
  reciprocity, the multiplicity of the trivial representation is equal
  to the sum over~$h\in\Zz$ of the multiplicity of the
  character~$\det(\cdot)^{2ph}$ in~$\rho$, i.e., it is equal to one
  if~$\rho=\det(\cdot)^{2ph}$ for some non-zero~$h$ (because~$\rho$ is
  non-trivial), and zero otherwise. The first case cannot occur if~$p$
  is large enough, hence the claim.
  \par
  Now applying the Riemann Hypothesis as above to each irreducible
  subrepresentation of~$\rho(\mathcal{D}_p)$, we obtain
  $$
  \frac{1}{|D(n,d)(\Ff_p)|} \sum_{f\in D(n,d)(\Ff_p)}
  \Tr(\rho(\theta_p(f)))\ll p^{-1/2}
  $$
  for all~$p$, where the implied constant depends only on
  $(n,d,\rho)$. This implies the equidistribution in~\eqref{cor-katz:item2}.
\end{proof}

In~\cite{ffk}, the three last-named authors use the theory of
complexity, among other tools, to study the equidistribution of families
of exponential sums arising as discrete Fourier--Mellin transforms of
trace functions on commutative algebraic groups, generalizing the
equidistribution theorems of Deligne~\cite{weilII} (for powers
of~$\Gg_a$, through the Fourier transform) and Katz~\cite{katz-mellin}
(for~$\Gg_m$), see \cite{fresan} for a survey. Among other things, this
has applications, also discussed in~\cite{ffk}, to the study of the
variance of arithmetic functions over function fields in arithmetic
progressions, improving results of Hall, Keating and
Roddity-Gershon~\cite{variance}. We state here (a form of) the basic
result, in the ``vertical'' direction (see~\cite[Th.~2]{ffk}).

\begin{theo}
  Let $(G,u)$ be a connected commutative algebraic group over a finite
  field~$\Ff$ with a given quasi-projective embedding. Denote by~$\Ff_n$
  the extension of~$\Ff$ of degree~$n$ in an algebraic closure of~$\Ff$
  and by $\widehat{G}(\Ff_n)$ the group of characters of~$G(\Ff_n)$. 

  Let~$A$ be a geometrically irreducible perverse sheaf on~$G$ which is
  pure of weight zero. There exists a complex reductive algebraic
  group~$G_A$ with a maximal compact subgroup~$K_A$ such that the sums
  $$
  S(M,\chi)=\sum_{x\in G(\Ff_n)}\chi(x)t_A(x;\Ff_n),
  $$
  defined for $\chi\in\widehat{G}(\Ff_n)$, become equidistributed on
  average in~$\Cc$ with respect to the image under the trace map of the Haar probability
  measure~$\mu$ on $K_A$, i.e. the equality
  $$
  \lim_{N\to+\infty} \frac{1}{N} \sum_{1\leq n\leq N}
  \frac{1}{|G(\Ff_n)|} \sum_{\chi\in\widehat{G}(\Ff_n)} f(S(M,\chi))=
  \int_{K_A}f(\Tr(g))d\mu(g)
  $$
  holds for any continuous and bounded function $f\colon \Cc\to \Cc$.
\end{theo}


\section{Effective bounds}\label{sec-effective}

It is clear that the implied constants in our bounds for all of
Grothendieck's six functors can be made effective as long as that for the Betti numbers of a tensor product of complexes in Theorem \ref{main} can be made effective. We state here such
an effective bound and sketch the proof.

\begin{theo}\label{th-main-effective}
  For all objects $A$ and $B$ of $\Der(\Pp^n_k)$, the following estimate holds: 
  \[ 
    \sum_{i \in \Zz} h^i(\Pp^n_k, A \otimes B) \leq
    \frac{2^{16}}{3^4} e^{4/13}13^n (n+2)!
    \ c(A) c(B). 
  \]
\end{theo}


\begin{proof}[Sketch of proof] The primary issue is to control the bilinear form of Corollary \ref{genericbettibound}. It is convenient to calculate with this bilinear form in a basis generated by constant sheaves (identical to that of Lemma \ref{constantsheafbasis} up to a sign). Let $e_m$ be the class of $\overline{\CC(K_m)}$ for~$K_m$ the constant sheaf on an $m$-dimensional subspace, so that $e_0, \dots, e_n$ form a basis of the vector space $\CH_n\left( \overline{ T^*\Pp^n}\right) \otimes \Qq$. By Theorem \ref{geneufor}, the intersection pairing 
\[
e_{m_1} \cdot e_{m_2} = \overline{\CC(K_{m_1})} \cdot \overline{\CC(K_{m_2})}
\] is equal to $(-1)^n$ times the Euler--Poincaré characteristic of the intersection of a general $m_1$\nobreakdash-dimensional subspace and a general $m_2$-dimensional subspace. This intersection is a projective space of dimension $m_1+m_2-n$, and hence has Euler--Poincaré characteristic $\min(m_1+m_2+1-n,0)$, so the intersection number is equal to $(-1)^{n} \min(m_1+m_2+1-n,0)$.

By construction, the function of Lemma \ref{pullbacklinearity} is easy to calculate with respect to this basis: it sends $e_m$ to $e_{m-1}$. Therefore, the bilinear form of Lemma \ref{genericbettibound} satisfies $f(e_{m_1},e_{m_2}) = g(m_1+m_2+1-n)$ with 
\[
g(x)= \sum_{k=0}^{x}  4^k  (x-k)= \frac{4^{x+2} -3x -4 }{9}
\] for $x \geq 0$ and $g(x)=0$ for $x \leq 0$. For simplicity, we will upper-bound it by the simpler bilinear form $f$ such that $f(e_{m_1},e_{m_2}) = 4^{m_1+m_2+3-n}/9$.

Set $e_i' = \sum_{j \leq i}  2^{i-j} e_j$. Then $e_0',\dots,e_n' $ also form a basis. Fix a norm $\|\cdot\|$ on the vector space $\CH_n\left( \overline{ T^*\Pp^n}\right) \otimes \Rr$ which is the $\ell^\infty$ norm in this basis. This basis is convenient because the intersection of~$e_i'$ with the characteristic cycle of a test sheaf supported on~$\Pp^m$  is $1$ if~$m+i = n$ and $0$ otherwise. To check this, it suffices to check that the intersection number of $e_i$ with the characteristic cycle of a test sheaf supported on $\Pp^m$ is equal to~$1$ if~$i + m =n$, to~$-2 $ if $i + m - n =1$, and to $0$ otherwise. This follows from the fact that the tensor product of $K_i$ with the test sheaf supported on $\Pp^m$ is the test sheaf supported on $\Pp^{i+m-n}$, which has Euler--Poincaré
characteristic equal to~$1$ if $i+ 2-n=0$, to~$-2$ if~$i + m - n=1$, and to~$0$ otherwise.  

Because the intersection number of $e_i'$ with the characteristic cycle of a test sheaf is equal to~$1$ on $\Pp^m$ if $m+i=n$ and to~$0$ otherwise, Proposition \ref{characteristicclassbound} then shows that for a perverse sheaf $A$, the norm $\|\overline{\CC(A)}\|$ is bounded by $4(n+1) c(A)$. 

Rewriting the bilinear form $f$ in the basis $e_i'$ gives
\begin{align*}
f(e_{i_1}',e_{i_2}') &= \sum_{j_1 \leq i_1} \sum_{j_2 \leq i_2} \frac{ 2^{ i_1-j_1} 2^{i_2-j_2} 4^{j_1+j_2+3-n}}{9} \\
&=\frac{ 2^{i_1+i_2} 4^{3-n}}{9} \Bigl( \sum_{j_1 \leq i_1} 2^{j_1}\Bigr)\Bigl( \sum_{j_2 \leq i_2} 2^{j_2}\Bigr) \\
&\leq \frac{ 2^{i_1+i_2} 4^{4-n}}{9}  2^{ i_1 + 1} 2^{i_2+1} = \frac{ 4^{i_1+i_2+ 4-n}}{9}. 
\end{align*}
Hence, the total norm of the bilinear form is at most 
\[ 
  \sum_{i_1=0}^n \sum_{i_2=0}^n \frac{ 4^{i_1+i_2 4-n}}{9} =
  \frac{4^{4-n}}{9} \Bigl( \sum_{i=0}^n 4^{i}\Bigr)^n =
  \frac{4^{4-n}}{9} \Bigl( \frac{4^{n+1}}{3}\Bigr)^2=
  \frac{4^{6+n}}{3^4}
\]
and the constant of Corollary \ref{genericbettibound} involves an extra
factor of $16(n+1)^2$ coming from the constant of Proposition
\ref{characteristicclassbound}, for a total of $(n+1)^2 4^{8+n}/3^4$.

Let $b_n$ be the constant in Theorem \ref{main} in dimension $n$. Then
we can see from the induction argument that there are $n$ terms (it may
appear to be $n+1$, but we may set~\hbox{$\lambda_1=1$} at the start by
scaling the whole matrix, which does not affect the automorphism of
$\Pp^m$), each of size $13 b_{n-1} c(A) c(B)$, plus one term coming from
Corollary \ref{genericbettibound} of size
$4^{8+n}(n+1)^2 c(A) c(B)/3^4$, so we obtain:
\begin{align*}
  b_n &=13 n b_{n-1} +  \frac{4^{8+n}}{3^4} (n+1)^2
  =\sum_{k =0}^n 13^{n-k} \frac{n!}{k!}  \frac{4^{8+k}(k+1)^2}{3^4}
  \\ &= n! 13^n \sum_{k =0}^n   \frac{4^{8+k}(k+1)^2}{3^4 13^k k!} 
  \leq (n+1)^2 n! 13^n \frac{2^{16}}{3^4} e^{4/13} \leq (n+2)! 13^n
  \frac{2^{16}}{3^4} e^{4/13},
\end{align*}
as we wanted to show.
\end{proof}

\begin{remark}
  Being more careful in the numerical arguments would lead to a
  significantly improved constant (the dominant terms in the sum
  defining $b_n$ are those for small~$k$, and for such values our bounds
  could see significant improvement, \eg we could use the constant~$1$
  in Corollary \ref{genericbettibound} for $k=0$ instead
  of~$2^{16}/3^4$), and some minor adjustments to the algebraic geometry
  can lower the base of the exponent $13^n$, but we do not know how to
  improve on the factorial growth, except in characteristic zero where a
  completely different argument offers exponential growth. \end{remark}

From this statement, it is completely straightforward to make explicit
the inequalities of Section \ref{sec-6ops}, because all the implicit
constants in those inequalities come from repeated applications of
Theorem \ref{main}, which can be replaced with this effective version.


\bibliographystyle{amsplain}

\bibliography{references}

\end{document}